\documentclass[%12pt, preprint, 
3p
]{elsarticle}
\usepackage[centertags]{amsmath}
\usepackage{amsfonts}
\usepackage{amssymb}
\usepackage{amsthm}
\usepackage{newlfont}
\usepackage{color}

\usepackage[applemac]{inputenc}
\usepackage[english]{babel}
\usepackage{XTocinc} %Include Table of Contents as the first entry in TOC
\newlength{\defbaselineskip}
\newcommand{\setlinespacing}[1]%
           {\setlength{\baselineskip}{#1 \defbaselineskip}}

\DeclareMathOperator*\ess{ess\,sup}

\newcommand{\R}{\mathbb R} %GF
\newcommand{\N}{\mathbb{N}} %GF

\theoremstyle{plain}
\newtheorem{thm}{Theorem}[section]
\newtheorem{cor}[thm]{Corollary}
\newtheorem{lem}[thm]{Lemma}
\newtheorem{prop}[thm]{Proposition}

\theoremstyle{definition}
\newtheorem{defn}{Definition}[section]
\theoremstyle{remark}
\newtheorem{rem}{Remark}[section]
\theoremstyle{example}
\newtheorem{ex}{Example}[section]
\numberwithin{equation}{section}

\usepackage{geometry} % see geometry.pdf on how to lay out the page. There's lots.
\date{} % delete this line to display the current date

\begin{document}
\begin{frontmatter}



\title{Approximate controllability for nonlinear degenerate parabolic problems with bilinear
control\tnoteref{T}}
\tnotetext[T]{
%\textcolor{red}{
This work was supported by the ÒInstituto Nazionale di Alta MatematicaÓ (INdAM), through the GNAMPA Research Project 2014: 
\lq\lq Controllo moltiplicativo per modelli diffusivi nonlineari'' (coordinator G. Floridia).%}
\\
Moreover, this research was %has been 
performed in the framework of the GDRE CONEDP (European Research Group on \lq\lq Control of Partial Differential Equations'') issued by CNRS, INdAM and Universit\'e de Provence.
}

%\fnref{AB}{ASD}
%\corref{aghdg}
\author{Giuseppe Floridia\corref{cor1}%\thanks{INDAM}
}
\address{Dipartimento di Matematica,\\
        Universit\`a di Roma \lq\lq Tor Vergata'',\\
        I-00161 Roma, Italy\fnref{label3}}
\ead{floridia@mat.uniroma2.it}
%\fnref[I]{fmmnvfmfmvnf}
%\fntext[label2]{A}
\cortext[cor1]{Post Doc \textit{Istituto Nazionale di Alta Matematica} (INdAM) \lq\lq F. Severi'', Roma.}
%\fntext[label3]{C}
%Post Doc \textit{Istituto Nazionale di Alta Matematica} (INdAM) \lq\lq F. Severi'', Roma.\\
%This research has been performed in the framework of the GDRE CONEDP (European Research Group -GDRE- on \lq\lq Control of Partial Differential Equations'' -CONEDP- issued by CNRS, INdAM and Universit\'e de Provence).

\begin{abstract}
In this paper, we study the global approximate multiplicative controllability for nonlinear degenerate parabolic Cauchy-Neumann problems.
%$$
%\left\{\begin{array}{l}
%\displaystyle{v_t-(a(x) v_x)_x =\alpha (t,x)v\,\,\qquad \mbox{in} \qquad Q_T \,=\,(0,T)\times(-1,1) }\\ [2.5ex]
%\displaystyle{a(x)v_x(t,x)|_{x=\pm 1} = 0\,\,\qquad\qquad\qquad\,\, t\in (0,T) }\\ [2.5ex]
%\displaystyle{v(0,x)=v_0 (x) \,\qquad\qquad\qquad\qquad\quad\,\, x\in (-1,1)}~,
%\end{array}\right.
%$$
%with bilinear controls $\alpha(t,x)\in L^\infty (Q_T).$ The problem is strongly degenerate in the sense that $a\in C^1([-1,1]),$ positive on $(-1,1),$ is allowed to vanish at $\pm 1$ provided that a certain integrability
%condition is fulfilled.
First, we will obtain embedding results for weighted Sobolev spaces, %}
 that have proved decisive in reaching well-posedness for nonlinear degenerate problems.
Then, we show that the above systems can be steered in
%\textcolor{red}{
$L^2%(-1,1) %PRG
$
%}
from any nonzero, nonnegative initial state into any neighborhood of any desirable nonnegative target-state by bilinear piecewise static controls. % (x-dependent only).
%In the particular case $a(x)=1-x^2,$ the above system represents the Budyko-Sellers one-dimensional climatology model.
%In this case we study the controllability properties in large time relating to understand man's possible actions on the environment, in order to intervene the evolution of the temperature.
Moreover, we extend the above result relaxing the %nonnegative
sign constraint on the %$v_0$
initial date.
%\colorbox{red}{
%Some embedding results for weighted Sobolev spaces,%}
%obtained in this paper, have proved decisive in reaching the well-posedness for the degenerate problems considered.
\end{abstract}

\end{frontmatter}

\section{Introduction}

%%%%%%%%INTRODUZIONE%%%%%%%%%%%%%%%%%%%%%%%%%%%%%%%%%%%%%%%%%%%%%%%%%%%%%%%%%%%%%%%%%%%%%%%%
This paper is concerned with the analysis of semilinear parabolic control systems in one space dimension, 
governed in the bounded domain $(-1,1)$ by means of the \textit{bilinear control} $\alpha (t,x),$ of the form
%%\vspace{1.5cm}
\begin{equation}
\label{Psemilineare}
\left\{\begin{array}{l}
\displaystyle{u_t-(a(x) u_x)_x =\alpha(t,x)u+ f(t,x,u)\,\quad \mbox{ in } \; Q_T \,:=\,(0,T)\times(-1,1) }\\ [2.5ex]
\displaystyle{a(x)u_x(t,x)|_{x=\pm 1} = 0\,\,\qquad\qquad\qquad\qquad\qquad\;\;\,\,\, t\in(0,T) }\\ [2.5ex]
\displaystyle{u(0,x)=u_0 (x) \,\qquad\qquad\qquad\qquad\quad\qquad\qquad\;\; \,x\in(-1,1)}~.
\end{array}\right.
\end{equation}
The equation in the \textit{Cauchy-Neumann}
problem 
%$(\ref{Psemilineare})$
above is a degenerate parabolic equation, because the diffusion coefficient, positive on $(-1,1),$ is allowed to vanish at the extreme points of $[-1,1]$.\\ 
The main physical motivations for studying degenerate parabolic problems with the above structure come from mathematical models in climate science as we explain below.
%%%%%%%%%%%%%%%%%%%%%%%%%%%%%%%%%%%%%%%%%%%%%%%%%%%%%%%%%%%%%%%%%%%%%%%%
%%%%%%%%%%%%%%%%%%%%%%%%%%%%%%%%%%%%%%%%%%%%%%%%%%%%%%%%%%%%%%%%%%%%%%%%
%\susection{Motivations}
\subsection{Physical motivations: Climate models and degenerate parabolic equations}
%TRA I MODELLI DI BILANCIAMENTO DELL'ENERGIA UNO DEI Pi CELEBRI  BS\\
%\colorbox{red}{
Climate depends on various parameters such as temperature,
humidity, wind intensity, the effect of greenhouse gases, and so
on. It is also
%and is
affected by a complex set of interactions in the atmosphere,
oceans and continents, %which they speak of
 that involve  physical, chemical,
geological and biological processes.\\%} %\\
%% The weather is also involved in
% Climatology deals with the problem to understand
%% understanding the problem of
% global warming on the Planet.
One of the first attempts to model the effects of the interaction between large ice masses and solar radiation on
climate is the one due, independently, to Budyko \cite{B1}, \cite{B2},
and Sellers \cite{S} %, the mathematical analysis of those
%models can be shown, for example, in
 (see also \cite{D2%,D1,D
 }--\cite{DH}, \cite{H}, \cite{TV}, \cite{Prov1}, \cite{Prov2} and the
references therein). 
% Such a model studies
%how extensive the climate response is to an event such as a sharp
%increase in greenhouse gases; in this case we talk about climate
%sensitivity. A process that changes climate sensitivity is called
%\textit{feedback}. If the process increases the intensity of
%response we say that it has
% \textit{positive feedback}, whereas it has  \textit{negative feedback} if it reduces the intensity of response.
The Budyko-Sellers  model is an {\it energy balance model}, which studies the role played by
continental and oceanic areas of ice on climate change.
%At the
%beginning of an ice age, the climate cools, then the ice on the
%continents and oceans gains. The average albedo of the Earth and
%significantly increases the amount of energy absorbed by the
%surface.\\
%In such a model, the sea level mean
%zonally averaged temperature $u(t, x)$ on the Earth, where $t$
%denotes time and $x$ the sine of latitude, satisfies the
%following degenerate \textit{Cauchy-Neumann} problem (\ref{BS}) in
%the bounded domain $(-1,1)$.
%\vspace{-0.5cm}
%\begin{equation}\label{1}
%u_t-(a(x) u_x)_x =\alpha (t,x)u + f (t,x,u)\,\,\qquad \mbox{in}
%\qquad Q_T \,:=\,(0,T)\times(-1,1),
%\end{equation}
%\vspace{-0.5cm}
%\begin{equation}\label{2}
%a(x)u_x(t,x)|_{x=\pm 1} = 0, \,\,\qquad\qquad\quad t\in (0,T),
%\end{equation}
%\vspace{-0.5cm}
%\begin{equation}\label{3}
%\,u(0,x)=u_0 (x), \,\,\qquad\qquad\qquad\qquad x\in (-1,1),
%\end{equation}
%%%%%%%%%%%%%%%%%%%%%%%%%%%%%%%%%%%%%%%%%%%%%%%%%%%%%%%%%%%%%%%%%%
%\newpage
%%%%%%%%%%%%F1tolto il 27 09 2013%%%%%%%%%%%%%%%%%%%%%%%%%%%%%%%%%%%%%%%%%%%%%%%%%%
The effect of solar radiation on climate can be summarized in the following: %figure\\
%%%%%%%%%%%GF Figure%%%%%%%%%%%%%%%%%%%%%%%%%%%%%%%%%%%%%%%%%%%%%%%%%%%%%%%%%%%%%%%%%%
%\begin{figure}[htbp]
%\begin{center}
%\hskip-.5cm
%\includegraphics[width=5.5cm, height=6cm]{EnergyBalance_fe290}
%\caption{\scriptsize{www.edu-design-principles.org (copyrighted by DPD)}}
%%\label{default}
%\end{center}
%\end{figure}
%%%%%%%%%%%%%%%%%%%%%%%%%%%%%%%%%%%%%%%%%%%%%%%%%%%%%%%%%%%%%%%%%%%%%%%%%%%%%%%%%%%%%%
%We have the following {\it energy balance}:\hspace{0.5cm}
%%\fbox{
%%%\rule[-.5cm]{0cm}{1cm}
%%\parbox{2.3in}{
$$\!\!\!\!\mbox{\textit{Heat variation}}=R_a-R_e+D,$$
where $ R_a$ is the \textit{absorbed
energy}, $ R_e$ is the \textit{emitted energy} and $ D$ is the \textit{diffusion part}.\\
%%{\small\begin{itemize}
%%\item $ R_a$ = absorbed
%%energy
%%\item $ R_e$ = emitted energy
%%\item $ D$ = diffusion
%%\end{itemize}}
%%}}
%%%%%%%%%%%%%%%%%%%%%%%%%%%%%%%%%%%%%%
%%%%%%%%%%%%%%%%%%%%%%%%%%%%%%%%%%%%%%%%%%%%%%%%%%%%%%%%%%%%%%%%%%%
The general formulation of the Budyko-Sellers model on a compact surface $\mathcal M$ without boundary is as follows
%%considering
% %compact surface without boundary (typically $ S^2$),\\
%%then the following equation is obtained
$$u_t-\Delta_{\mathcal M} u= R_a(t,X,u)-R_e(t,X,u),$$
\noindent where $u(t,X)$ is the distribution of temperature, $\Delta_{\mathcal M}$ is the classical Laplace-Beltrami operator,
%moreover, %we have
%%\pause
%%\begin{itemize}
%%\item
$ R_a(t,X,u)=Q(t,X)\beta(X,u).$
%%\item $ R_e(t,x,u)=A(t,x)+B(t,x)u$
%%\end{itemize}
In the above,
%\textcolor{red}{ 
$ \,Q$ is the \textit{insolation} function,
that is, %the incoming solar flux, 
the incident solar radiation at the top of the atmosphere.
In annual models, when the time scale is long enough, one may assume that the insolation function doesn't depend on time $t,$ i.e. $Q= Q(X)$. But, when the time scale is smaller, as in seasonal
models, one uses a more realistic description of the incoming solar flux by assuming that $Q$ depends on $t,$ i.e. $Q = Q(t, X).$
%%\qquad\qquad
%PRG
$\,\beta$ %} 
is the \textit{coalbedo} function, that is, 1-\textit{albedo function}.
%%\quad
%%\pause
%%%%%%%%%%%%%%%%%%%%%%%Wikipedia
%%Albedo, or reflection coefficient, is the diffuse reflectivity  or reflecting power of a surface. It is defined as the %ratio of reflected radiation from the surface to incident radiation upon it. Being a dimensionless  fraction, it may %also be expressed as a percentage, and is measured on a scale from zero for no reflecting power of a perfectly black %surface, to 1 for perfect reflection of a white surface.\\
%%%%%%%%%%%%%%%%%%%%%%%%%%%%%%%%%%%%%%%%%%%
Albedo is the reflecting power of a surface. It is defined as the ratio of reflected radiation from the surface to incident radiation upon it. It may also be expressed as a percentage, and is measured on a scale from zero, for no reflecting power of a perfectly black surface, to 1, for perfect reflection of a white surface.\\
%%%%%%%%%%%GF Figure%%%%%%%%%%%%%%%%%%%%%%%%%%%%%%%%%%%%%%%%%%%%%%%%%%%%%%%%%%%%%%%%%%
%%\begin{figure}[htbp]
%%\begin{center}
%%%\hskip-.5cm
%%\includegraphics[width=4cm, height=5cm]{albedo}
%%\caption{www.esr.org (copyrighted by ESR)}
%%%\label{default}
%%\end{center}
%%\end{figure}
%%%%%%%%%%%%%%%%%%%%%%
%The main difference between Budyko's model and the one by Sellers, %model,
%is that in
%the former %Budyko
%the coalbedo function is discontinuous, while %that of Sellers
%in the latter it is a
%continuous function. 
%%In fact, we have\\
%%%%%%%%%%%%%%%%%
%%%fbox{
%%%%\rule[-.5cm]{0cm}{1cm}
%%%\parbox{3in}{
%%\begin{itemize}
%%\item in Budyko:
%%$$
%%\hspace{-.5cm}\beta(u)=
%%\left\{\begin{array}{l}
%%\displaystyle{\beta_0
%%\,\,\qquad u<-10 }\\ [2.5ex]
%%\displaystyle{[\beta_0,\beta_1]\quad u=-10  }\\ [2.5ex]
%%\displaystyle{\beta_1 \,\,\qquad u>-10 }~,
%%\end{array}\right.
%%$$
%%%\pause
%%\item in Sellers:
%%$$
%%\hspace{-.5cm}\beta(u)=
%%\left\{\begin{array}{l}
%%\displaystyle{\beta_0
%%\,\,\qquad u<u_- }\\ [2.5ex]
%%\displaystyle{\mbox{line}\quad u_-\leq u\leq u_+  }\\ [2.5ex]
%%\displaystyle{\beta_1 \,\,\qquad u>u_+ }~,
%%\end{array}\right.
%%$$
%%where $u_\pm=-10\pm\delta, \delta>0.$
%%\end{itemize}
%%%}}
%%%\vspace{8.0cm}
%Moreover, in Budyko we have
%$ R_e(t,X,u)=A(t,X)+B(t,X)u,$
%while in Sellers
%$R_e(t,X,u)\simeq C\,u^4.$
%%%%%%%%%%%%%%%%
On $\mathcal M=\Sigma^2$ the Laplace-Beltrami operator is
$$\Delta_{\mathcal M}\,u=\frac1{\sin \phi}\Big\{\frac{\partial}{\partial \phi}\Big(\sin \phi \frac{\partial u}{\partial\phi}\Big)+\frac1{\sin \phi}\,\frac{\partial^2u}{\partial \lambda^2}\Big\},$$
where $\phi$ is the \textit{colatitude} and $\lambda$ is the \textit{longitude}.
In the one-dimensional Budyko-Sellers we
take the average of the temperature at $ x=\cos\phi,$ where $\phi$ is the \textit{colatitude}.
In such a model, the sea level mean
zonally averaged temperature $u(t, x)$ on the Earth, where $t$
denotes time, %and $x$ the sine of latitude,
 satisfies the
following \textit{Cauchy-Neumann} degenerate problem %(\ref{BS})
 in
the bounded domain $(-1,1)$ %and the Budyko-Sellers model reduces to
%$$u_t-\big((1-x^2)u_x\big)_x=g(t,x)\,h(x,u)+f(t,x) \qquad x\in ]-1,1[$$
%\vspace{.5mm}
%\\
%$$(1-x^2)u_{x|_{x=\pm 1}}=0$$
%\end{cases}
%\end{equation*}
\begin{equation*}%\label{BS}
  \left\{\begin{array}{l}
\displaystyle{u_t-\big((1-x^2)u_x\big)_x=g(t,x)\,h(%x,
u)+f(t,x,u), \qquad x\in (-1,1),}\\ [2.5ex]
\displaystyle{%\textcolor{red}{
(1-x^2)u(t,x)_{x}|_{x=\pm 1}=0, \qquad\qquad\qquad\qquad\quad\quad\;\, t\in(0,T) %}
}~,
\end{array}\right.
\end{equation*}
%\textcolor{red}{
where the meaning of this boundary condition will be clarified in Section 3.%}
%Observe that the leading part of the differential operator in $(\ref{BS})$ satisfies assumptions $(A.4).$

\subsection{Mathematical motivations, contents and structure%: Why multiplicative controllability?
}
Interest in degenerate parabolic equations dates back by almost a century. Significant contributions are due to Fichera's and Oleinik's studies (see e.g., respectively, \cite{Fichera} and \cite{OR}).\\
In control theory, boundary and interior locally distributed controls are usually employed (see, e.g., \cite{CMV2}%, \cite{CMV1}, 
--\cite{CMV3}, \cite{FC}, \cite{FCZ}, \cite{FI}, \cite{BDDM1} and \cite{BDDM2}).
These controls are additive terms in the equation and have localized support. %As regards of applications it seem that these controls cannot model all physical process but only process that do not change their principal physical characteristics due to the control action.
%In this way many new technologies are excluded,
However, such models are unfit to study several interesting applied problems such as chemical reactions controlled by catalysts, and also smart materials, which are able to change their principal parameters under certain conditions.\\ %From this the growing interest in the study of the \textit{multiplicative controllability}.
Additive control problems for the Budyko-Sellers model have been studied by J.I.Diaz, in the work \cite{D1} (see also the interesting papers \cite{D2}, \cite{D} and \cite{DH}).\\
%Even in Budyko-Sellers model, modeling the control action by an additive term would require huge amounts of energy, which may not always be realistic to afford.
%On the other hand, one could imagine to influence the so-called \textit{albedo} by some kind of device as predicted by J. Von Neumann
%\vspace{0.5cm}
%
% \lq\lq\textit{Microscopic layers of colored matter spread on an icy surface,
%or in the atmosphere above one, could inhibit the reflection-radiation
%process, melt the ice and change the local climate.}'' (J. von Neumann, \textit{Nature}, 1955)
%
%\vspace{0.5cm}
%
%and
%
%\vspace{0.5cm}
%
%\lq\lq\textit{Probably intervention
%in atmospheric and climate matters will come in a few decades, and will
%unfold on a scale difficult to imagine at present}'' (J. von Neumann, \textit{Nature}, 1955)\,.
%
%\vspace{0.5cm}
%
%From the mathematical view point such a 
In the present work, the control action would take the form of a bilinear control, that is, a control given by a multiplicative coefficient.
%This explains the growing interest in \textit{multiplicative controllability}.
General references for \textit{multiplicative controllability} %see also the following references
are, e.g., \cite{K1}--\cite{KB}
%, \cite{K2}, \cite{K3}, \cite{K4}, 
 and \cite{BS}.
%%%%%%%%%%%%%%%%%%%%
%\subsection{Preliminary papers, structure and contents
%}
%In the Ph.D Thesis, we study the approximate controllability of $(\ref{Psemilineare})$  by bilinear controls with initial state $u_0\in H^1_a(-1,1)$.\\
Our approach  is inspired by \cite{K} and \cite{CK}. %and \cite{CFproceedings1}. %in particular
In \cite{K}, A.Y. Khapalov studied the global nonnegative approximate controllability of the one dimensional \textit{non-degenerate} semilinear convection-diffusion-reaction equation governed in a bounded domain via bilinear control. %$\alpha\in L^\infty (Q_T).$ 
%in the additive reaction term.
In \cite{CK}, P. Cannarsa and A.Y. Khapalov derived the same approximate controllability property in suitable classes of functions that change sign.\\
Then, I considered, in collaboration with P. Cannarsa, the linear degenerate problem associated to $(\ref{Psemilineare})$  (i.e. when $f\equiv 0$) in two distinct kinds of set-up. 
%(In Chapter 2 the \textit{Well-posedness} and in Chapter 4 the \textit{Controllability}), 
Namely, first, in \cite{CF2} and \cite{CF3} we considered
%\begin{itemize}
  %\item
  the \textit{weakly degenerate} problems (WD), that is, when $\frac{1}{a}\in L^1(-1,1);$ then,
  %\item 
 in \cite{CFproceedings1} and \cite{CF3} we considered the \textit{strongly degenerate} problems (SD), that is, when $\frac{1}{a}\not\in L^1(-1,1).$
%\end{itemize}
Observe that the Budyko-Sellers model is an example of SD operator.\\
The WD case is somewhat similar to the uniformly parabolic case.
Indeed, it turns out that all functions in the domain of the corresponding differential operator possess a trace on the boundary, in spite of the fact that the operator degenerates at such points. In the WD case, we are able to study a %equation
\textit{Cauchy-Robin} boundary problem, and
%%%%%%%F1tolto il 27 09 1983%%%%%%%%%%%%%
%\begin{equation*}%\label{P2W}
%\begin{cases}
%\qquad v_t-(a(x) v_x)_x =\alpha (t,x)v \,\,\qquad\hfill \mbox{ in } \;\; Q_T \,:=\,(0,T)\times(-1,1) \\
%\begin{cases}
%\beta_0 v(t,-1)+\beta_1 a(-1)v_x(t,-1)= 0 \,\,\qquad\hfill \;\;t\in (0,T)\,\\
%\gamma_0\, v(t,1)\,+\,\gamma_1\, a(1)\,v_x(t,1)= 0\,\,\hfill \;\;t\in (0,T)\,
%\end{cases}
%\\
%\qquad v(0,x)=v_0 (x) \,\,\,
% \qquad\qquad\qquad\qquad \;\;\;\;\; x\in (-1,1)\,.
%\end{cases}
%\end{equation*}
%%$$v_t-(a(x) v_x)_x =\alpha (t,x)v \,\,\qquad\hfill \mbox{ in } \;\; Q_T \,=\,(0,T)\times(-1,1)\,,$$
%%with general Robin boundary conditions
%%$$\begin{cases}
%%\beta_0 v(t,-1)+\beta_1 a(-1)v_x(t,-1)= 0 \,\,\qquad\hfill \;\;t\in (0,T)\,\\
%%\gamma_0\, v(t,1)\,+\,\gamma_1\, a(1)\,v_x(t,1)= 0\,\,\hfill \;\;t\in (0,T)\,.
%%\end{cases}$$
%For this \textit{Cauchy-Robin} problem 
we obtain a result of
global nonnegative approximate multiplicative controllability
 in $L^2(-1,1)$. So, we show that the above system can be steered, in
the space
of square-summable functions, from any nonzero, nonnegative initial
state
into any neighborhood of any desirable nonnegative target-state by
bilinear
static controls. Moreover, we extend the above result relaxing the
sign constraint on the
initial-state.\\
On the other hand, in the SD case %(Section 2.3 and Section 4.2) 
one is forced to restrict to the Neumann type boundary conditions %in (\ref{Psemilineare})
 (as in the Budyko-Sellers model).
%%%%%%%F1 tolto il 27 09 2013 
%\begin{equation*}
%%\label{P2slineare}
%%{\displaystyle{
%%	\cases{\displaystyle\  v_t-(a(x) v_x)_x =\alpha (t,x)v  ~~~~~ \\\\  &$(t,x)\in Q_T
%%%\,=\,(0,T)\times(-1,1),
%% $\cr \\
%%\displaystyle\ a(x)v_x(t,x) %\stackrel{x \rightarrow\pm 1}{\longrightarrow}
%%|_{x=\pm 1} = 0  ~~~~~ \\\\ &$t\in (0,T) $\cr \\
%%\displaystyle\ v(0,x)=v_0 (x)  ~~~~~ \\\\  &$x\in (-1,1), $\cr
%%}}}
%%%%%%%%%%%%%%%%%%%%
%\left\{\begin{array}{l} \displaystyle{v_t-(a(x) v_x)_x =\alpha
%(t,x)v\,\,\qquad \mbox{in} \qquad Q_T \,:=\,(0,T)\times(-1,1) }\\
%[2.5ex] \displaystyle{a(x)v_x(t,x)|_{x=\pm 1} =
%0\,\,\qquad\qquad\qquad\,\, t\in (0,T) }\\ [2.5ex]
%\displaystyle{v(0,x)=v_0 (x) \,\qquad\qquad\qquad\qquad\quad\,\,
%x\in (-1,1)}~.
%\end{array}\right.
%%%%%%%%%%%%%%%%%
%%\begin{cases}
%%v_t-(a(x) v_x)_x =\alpha (t,x)v,\,\,\qquad \mbox{in} \qquad Q_t \,=\,(0,T)\times(-1,1),
%%\\
%%a(x)v_x(t,x)
%%%\stackrel{x \rightarrow\pm 1}{\longrightarrow}
%%|_{x=\pm 1} = 0, \,\,\qquad\qquad\qquad\,\, t\in (0,T),
%%\\
%%v(0,x)=v_0 (x) \,\qquad\qquad\qquad\qquad\qquad\, x\in (-1,1),
%%\end{cases}
%\end{equation*}
Even in this case (SD linear case), we establish the global nonnegative approximate multiplicative controllability in $L^2(-1,1)$, after proving the compact embedding in $L^2(-1,1)$ of the weighted Sobolev space
$H^1_a(-1,1)%:=\{u\in L^2(-1,1): u \mbox{ locally absolutely continuous in } (-1,1),\; \sqrt{a} u_x \in L^2(-1,1)\},
$ ($H^1_a(-1,1)$ is the space of all functions $u\in L^2(-1,1)$ such that $u$ is locally absolutely continuous in $(-1,1)$ and $\sqrt{a}\,u_x\in L^2(-1,1)$),
under the assumption $\xi_a\in L^1(-1,1),$ where $\xi_a(x)=\int_0^x \frac{ds}{a(s)}.$\\
%\vspace{0.5cm}
In this paper we focus just on semilinear strongly degenerate problems, %thus including the operator of Budyko-Sellers model, 
and we obtain the global nonnegative approximate controllability of $(\ref{Psemilineare})$ by bilinear piecewise static controls with initial state $u_0\in L^2(-1,1).$\\
 %The Well-posedness and Chapter 5 (Controllability) of this thesis.\\
 %For brevity, 
 The technique of this paper is inspired by A.Y. Khapalov in \cite{K}, for uniformly parabolic equations. 
%The main difference between the problem $(\ref{Psemilineare}),$ a strongly degenerate problem, and the uniformly parabolic equation 
% is the lower regularity of the weighted Sobolev spaces $H^1_a(-1,1),$ which are fundamental for the well-posedness of the problem $(\ref{Psemilineare}),$ instead of the space $H^1_0(-1,1),$ which is fundamental for the well-posedness in the uniformly parabolic case. 
%Precisely, 
The main technical difficulty to overcome with respect to the uniformly parabolic case, %treated in \cite{K}, 
is the fact that functions in $H^1_a(-1,1)$ need not be necessarily bounded when the operator is strongly degenerate.
Thus, some embedding results for weighted Sobolev spaces obtained in this article have proved decisive in reaching the desired controllability.
In particular, using the above embedding results and some results found in 
%the Ph.D Thesis
 \cite{CF3}, we obtain the well-posedness of $(\ref{Psemilineare})$ with initial state in $L^2(-1,1).$\\ 
In \cite{CF3}, we established the existence and uniqueness of solution to $(\ref{Psemilineare})$ with initial data in $H^1_a(-1,1)$ and, 
in order to obtain this result, we followed the classical method which consists in obtaining a local result by fixed point arguments, and then show that the solution is global in time by proving an a priori estimate (see \cite{CF3} and also Appendix B).
In fact, first, the nonlinear system $(\ref{Psemilineare})$ has been addressed in \cite{CF3}, assuming sufficient regularity on the initial data, that is, $u_0\in H^1_a(-1,1,)$ and obtaining an approximate controllability result in large time.
Such a regularity was necessary to develop the approach of \cite{CF3}, that was confined to strict solutions of $(\ref{Psemilineare})$ (see Section 3 for the definition of strict solution). On the other hand, the above procedure has some drawbacks, such as the restriction of the admissible target states to functions $u_d\in H^1_a(-1,1,)$ satisfying $\langle u_0, u_d\rangle_{1,a}>0.$   
The main purpose of this paper is to extend the analysis of \cite{CF3}, relaxing the regularity assumptions on $u_0,\,u_d$ to $u_0,\,u_d\in L^2(-1,1)$ and $u_0,\,u_d\geq0,$ with $u_0\neq0.$  \\
The structure of this paper is the following. Section 2 deals with the problem formulation and gives the main results. Section 3 deals with well-posedness for semilinear equations %$(\ref{Psemilineare})$
with initial state in $L^2(-1,1),$ and includes some new embedding results for weighted Sobolev spaces.
%In the same section, initially we introduce the concept of strict solution and strong solution (see, respectively, definition 3.2 and definition 3.3).
%%,
%%after recalling the existence and uniqueness results for the linear problems associated with $(\ref{Psemilineare}).$
% Once well-posedness is established in Section 3,  
In Section 4, we prove %the analysis of 
 the global nonnegative approximate controllability of $(\ref{Psemilineare})$ via bilinear controls.
 Moreover, in Appendix A we recall the proof of a result for singular Sturm-Liouville problems obtained in \cite{CFproceedings1} %, in collaboration with P. Cannarsa, 
 (this result is used in the proofs of the main results) 
%\textcolor{red}{
and we remind a classical regularity result of the positive and negative part of a given function. %}
 In Appendix B, we recall the proofs of the existence and uniqueness results for problem $(\ref{Psemilineare})$ with initial state in $H^1_a(-1,1),$ previously obtained by the author in 
 %the PhD Thesis 
\cite{CF3}.\\ 
Now, let us consider some open questions pertaining to this paper. First of all, %we would to derive similar results for weakly degenerate semilinear control systems, we are confident that the approach of this paper also applies to semilinear weakly degenerate parabolic equations with Robin boundary conditions. We will consider such generalizations in future works. 
%Moreover, 
in the future we intend to investigate similar problems in higher space dimensions on domains with specific geometries, first in the uniformly parabolic case (see, e.g., the preprint \cite{CFK}), then in the degenerate parabolic case.
Finally, once the above two issues have been addressed, we would like to extend our approach to other nonlinear systems of parabolic type, such as the systems of fluid dynamics (see, e.g., \cite{CFKP}).

\section{Problem formulation and main results}
This section gives the problem formulation and the main results of controllability of the system %is concerned with the analysis of semilinear parabolic control systems 
(\ref{Psemilineare}).
% in one space dimension, 
%governed in the bounded domain $(-1,1)$ by means of the \textit{bilinear control} $\alpha (t,x),$ of the form
%%%\vspace{1.5cm}
\subsection{Problem formulation}
In this paper, we consider the problem (\ref{Psemilineare})
\begin{equation*}
%\label{Psemilineare}
\left\{\begin{array}{l}
\displaystyle{u_t-(a(x) u_x)_x =\alpha(t,x)u+ f(t,x,u)\,\quad \mbox{ in } \; Q_T \,:=\,(0,T)\times(-1,1) }\\ [2.5ex]
\displaystyle{a(x)u_x(t,x)|_{x=\pm 1} = 0\,\,\qquad\qquad\qquad\qquad\qquad\;\;\,\,\, t\in(0,T) }\\ [2.5ex]
\displaystyle{u(0,x)=u_0 (x) \,\qquad\qquad\qquad\qquad\quad\qquad\qquad\;\; x\in(-1,1)}~,
\end{array}\right.
\end{equation*}
under the following assumptions:%%%%%%%%%%%%%%%%%%%%%%%%%
 \begin{enumerate}%\renewcommand{\labelenumi}{\Roman{enumi})}
  \item[(A.1)] $u_0 \in L^2(-1,1);$ %H^1_a(-1,1):=\{u\in L^2(-1,1): u \mbox{ locally absolutely continuous in } (-1,1),$\\ $\qquad\qquad\qquad\qquad\qquad\qquad\qquad\hfill\sqrt{a} u_x \in L^2(-1,1)\};$ %L^2(-1,1)$
 \item[(A.2)] $\alpha \in L^\infty (Q_T);$ %is a bilinear control;
%  \item[(A.3)] $f:(-1,1)\times\R\rightarrow \R$ is a Carath\'{e}odory function %(i.e. $f$ is Lebesgue
%%measurable in $x$ for every $u\in \R,$ and continuous in $u$ for almost %all
%%every $x\in\,(-1,1)$)
%such that\\
%there exist $\vartheta\,>1$, $\gamma_0>0,$ and $\gamma_1>0$ such that
%\begin{equation}\label{Superlinearit}
%|f(x,u)|\leq\gamma_0\,|u|^\vartheta, \mbox{ for a.e. } x\in(-1,1), \forall u\in \R\,,
%\end{equation}
%and
%\begin{multline}%\label{lipi}
%|f(x,u)-f(x,v)|\\
%\leq\gamma_1\left(1+|u|^{\vartheta-1}+|v|^{\vartheta-1}\right)|u-v|, \mbox{ for a.e. } x\in(-1,1),\;\forall u,v\in \R;
%\end{multline}
%there exists a %nonnegative constant $\nu,$ 
%locally integrable function $\nu:(0,+\infty)\rightarrow\R,\; \nu(t)\geq 0\; \forall t\in (0,+\infty),$
%such that
%\begin{equation}\label{fsigni}
%f(x,u)\,u \leq \nu(t)
%\,u^2,\qquad \mbox{ for a.e. } x\in(-1,1),\;\;\;  \forall u\in \R\,;
%\end{equation}
%Below we will put $\nu_T=e^{\nu T};%e^{\int_0^T \nu(t)\,dt}
\item[(A.3)] $f:Q_T\times\R\rightarrow \R$ is such that
\begin{itemize}%\item[[A.3.1]] 
\item $(t,x)\longmapsto f(t,x,u)$ is measurable $\forall u\in\R,$
\item $u\longmapsto f(t,x,u)$ is locally absolutely continuous for a.e. $(t,x)\in Q_T,$
\item%\colorbox{red}{ 
$t\longmapsto f(t,x,u)$ is locally absolutely continuous for a.e. $x\in (-1,1), \forall u\in\R,$\,(\footnote{
%\textcolor{red}{
This assumption is used only for well-posedness, see \ref{AppA}.
%}
})%}
\item there exist constants $\gamma_0\geq 0, \vartheta\in(1,3)$ and $\nu%\overline{\nu}
\geq0$
%and a locally integrable function $\nu:[0,+\infty)\longrightarrow \R,\;\nu\geq0$ %\mbox{ and } \int^T_0\nu(t)\,dt\leq\overline{\nu}\,T,\;\;\forall\, T\in [0,+\infty)
%$%} 
 such that
%$f:Q_T\times\R\rightarrow \R$ is a Carath\'{e}odory function, and\\ %(i.e. $f$ is Lebesgue
%%measurable in $x$ for every $u\in \R,$ and continuous in $u$ for almost %all
%%every $x\in\,(-1,1)$)
%%\item[[A.3.2]] 
%there exist $1\leq\vartheta<3$ and $\gamma_0>0$ %and $\gamma_1>0$ 
%such that
\begin{equation}\label{Superlinearit}
|f(t,x,u)|\leq\gamma_0\,|u|^\vartheta, \mbox{ for a.e. } (t,x)\in Q_T, \forall u\in \R\,;
\end{equation}
%\item[[A.3.2]]
%$u\rightarrow f(t,x,u)$ is locally absolutely continuous in $u,$ for almost every $(t,x)\in Q_T,$ and \\
%The function $f$ satisfies the following conditions:\\
%\item[[A.3.4]] 
%\begin{multline}\label{lipi}
%|f(x,u)-f(x,v)|\\
%\leq\gamma_1\left(1+|u|^{\vartheta-1}+|v|^{\vartheta-1}\right)|u-v|, \mbox{ for a.e. } x\in(-1,1),\;\forall u,v\in \R;
%\end{multline}
%\item[[NL3]]
% there exists a nonnegative constant $\mu$ and a
%locally integrable function $\nu:(0,+\infty)\rightarrow\R,\; \nu(t)\geq 0,\; \forall t\in (0,+\infty),$
%such that
\begin{equation}\label{fsigni}
-\nu%(t) 
\big(1+|u|^{\vartheta-1}\big)\leq f_u(t,x,u)\leq\nu%(t)
,\; \mbox{ for a.e.  }  (t,x)\in Q_T, \forall u\in \R;
%f(x,u)\,u \leq \nu(t)
%\,u^2,\qquad \mbox{ for a.e. } x\in(-1,1),\;\;\;  \forall u\in \R\,;
\end{equation}
\begin{equation*}%\label{fsigni}
  f_t(t,x,u)\,u\geq-\nu\, %(t)
  u^2,\; \mbox{ %\colorbox{red}{
 for a.e.%} 
 } \,(t,x)\in Q_T, \forall u\in \R;\,%\textcolor{red}{
 (^1)%}
%f(x,u)\,u \leq \nu(t)
%\,u^2,\qquad \mbox{ for a.e. } x\in(-1,1),\;\;\;  \forall u\in \R\,;
\end{equation*}
\end{itemize}
%Below we will put $\nu_T=e^{\nu T};%e^{\int_0^T \nu(t)\,dt}
  \item[(A.4)] $a \in C^1([-1,1])$ is such that
    $$a(x)>0, \,\, \forall \, x \in (-1,1),\quad a(-1)=a(1)=0,$$
    and, the function $\xi_a(x)=\int_0^x \frac{ds}{a(s)}$ satisfies the following
    \begin{equation}\label{Lintrod}
    \xi_a\in L^{q_\vartheta}(-1,1),
    %\cap L^{2\vartheta-1}(-1,1).
  \end{equation}
  where $$q_\vartheta=\max\Big\{\frac{1+\vartheta}{3-\vartheta}, 2\vartheta-1\Big\}.$$
\end{enumerate}
\begin{rem}\label{rem1}{
%\colorbox{red}{
The inequalities %}
 \eqref{fsigni}, in assumption $(A.3),$ imply the following conditions on the function $f$
 \begin{equation*}%\label{Remabsf}
\big|f_u(t,x,u)\big|\leq%\overline{
\nu%^*%}
%(t)
(1+|u|^{\vartheta-1}), \; \mbox{ for a.e. } (t,x)\in Q_T, \forall u,v\in \R;
\end{equation*}
\begin{equation}\label{Remf}
\big(f(t,x,u)-f(t,x,v)\big)(u-v)\leq\nu
%(t)
(u-v)^2, \; \mbox{ for a.e. } (t,x)\in Q_T, \forall u,v\in \R, (\footnote{%\textcolor{red}{
Since, for a.e. $(t,x)\in Q_T,$ $f(t,x,u)$ is locally absolutely continuous respect to $u,$ we have\\
$$\big(f(t,x,u)-f(t,x,v)\big)(u-v)=(u-v)\int^{u}_{v}f_u(t,x,\xi)d\xi\leq(u-v)\int^{u}_{v}\nu\,d\xi\leq\nu(u-v)^2,$$
$$\big|f(t,x,u)-f(t,x,v)\big|\leq\int^{max\{u,v\}}_{min\{u,v\}}|f_u(t,x,\xi)|d\xi\leq\nu\int^{max\{u,v\}}_{min\{u,v\}}(1+|\xi|^{\vartheta-1})d\xi\leq\nu(1+|u|^{\vartheta-1}+|v|^{\vartheta-1})|u-v|,$$
for a.e. $(t,x)\in Q_T,$ for every $u,\,v\in\R.$ 
%}
})%PRG
\end{equation}
%\textcolor{blue}{and,
%there exist a constant $\gamma_1\geq0$ %and %a locally integrable function $\nu^*:(0,+\infty)\rightarrow\R,\;$ %\overline{\nu}
%%$\nu%^*(t)
%%\geq 0,%\; \forall t\in (0,+\infty),
%%$
%such that}
\begin{equation}\label{Remlip}
\big|f(t,x,u)-f(t,x,v)\big|\leq
%\textcolor{blue}{\gamma_1}
%\textcolor{red}{
\nu%}
(1+|u|^{\vartheta-1}+|v|^{\vartheta-1})|u-v|,\;\;\mbox{ for a.e. } (t,x)\in Q_T, \forall u,v\in \R. \,%\textcolor{red}{
(^2)%}%PRG
 \end{equation}
 %and
}
%\textcolor{red}{
\begin{rem}\label{remmon}
We note that all the results of this paper hold true replacing the assumption that
 $u\longmapsto f(t,x,u)$ is locally absolutely continuously by the mere continuity of such a function, for a.e. $(t,x)\in Q_T.$ 
That is, in (A.3), it suffices to assume that:
\begin{itemize}
\item $(t,x,u)\longmapsto f(t,x,u)$ is a Carathodory function on $Q_T\times\R,$
\item $t\longmapsto f(t,x,u)$ is locally absolutely continuous for a.e. $x\in (-1,1), \forall u\in\R,$
\end{itemize}
and to substitute inequality \eqref{fsigni} by the two more general inequalities \eqref{Remf} and \eqref{Remlip}.
\end{rem}
%}%PRG
\end{rem}
\begin{rem}
%\begin{itemize}
%\colorbox{red}{
%\item 
The equation in %}
 the \textit{Cauchy-Neumann}
problem 
$(\ref{Psemilineare})$ is a degenerate parabolic equation because the diffusion coefficient, positive on $(-1,1),$ is allowed to vanish at the extreme points of $[-1,1]$.
In particular, since $\frac{1}{a}\not\in L^1(-1,1),$ this problem is \textit{strongly degenerate}.
A sufficient condition for %the diffusion coefficient $a(x)$
this is that $a^\prime(\pm1)\neq0$ (if $a\in C^2([-1,1])$ the above condition is also necessary). \\
%\end{rem}
%\colorbox{red}{Il Remark sotto  da sistemare}
%\begin{rem}
%\begin{itemize}
%\vspace{0.5cm}
%\item If $f(t,x,u)$ belongs to the space $C^1(\R),$ with respect to $u,$ a sufficient condition for the assumption (\ref{lip}) is that, for some $\vartheta>1$ and $\gamma_1>0,$
%$$|f_u(x,u)|\leq\gamma_1|u|^{\vartheta-1}\qquad \mbox{ for a.e. } x\in(-1,1),\; \forall u\in\R.$$
      %SISTEMARE LA PROOF
      % LA DIMOSTR. IN MANIERA ESPLICITA L'EQUIVALENZA???
 % \item $\frac{1}{a}\not\in L^1(-1,1),$
  %so $a(\cdot)$ is strongly degenerate.
  %\item %We observe that
The principal part of the operator in $(\ref{Psemilineare})$ coincides with that of the Budyko-Sellers model for
$a(x)=1-x^2$. In this case, $\xi_a(x)=\frac{1}{2}\ln\left(\frac{1+x}{1-x}\right),$ %
  so  $\xi_a\in L^p(-1,1),$ for every $p\geq 1.$
 \end{rem}
\begin{rem}
The assumption (\ref{Remf}) is more general than the classical sign assumption $\int_{-1}^{1}f(t,x,u)u\,dx\leq 0$(\footnote{This integral condition is used %by A. Y. Khapalov 
in \cite{K}, in the uniformly parabolic case, but also there it can be generalized by a condition similar to (\ref{Remf}).}), indeed the last condition is equivalent to $f(t,x,u)\,u \leq 0, \mbox{ for a.e. } (t,x)\in Q_T,\;\;  \forall u\in \R.$
%%\begin{itemize}
%%  \item if $a\in C^2([-1,1])$ the above condition is also necessary.
%%\end{itemize}
%\end{itemize}
\end{rem}
% We observe that the principal part of the operator in (\ref{P2}) coincides with that of the Budyko-Sellers model for
%$a(x)=1-x^2$.

\begin{ex}
An example of function $f$ that satisfies the assumptions $(A.3)$ is the following
$$f(t,x,u)=c(t,x)\min\{|u|^{\vartheta-1},1\}u-|u|^{\vartheta-1}u,$$
where $c$ %\in L^\infty(Q_T)$ and the function %$0\leq c(x)\leq K_0,\,\forall x\in (-1,1), \;K_0\in \R.$$t\longmapsto c(t,x)$ 
is a Lipschitz continuous function. %on $[0,T]$, 
%for a.e. $x\in (-1,1).$
\end{ex}

\subsection{Main results}
We are interested in studying the nonnegative multiplicative controllability of
%problem
 $(\ref{Psemilineare})$  by the \textit{bilinear control} $\alpha (t,x) %\in L^\infty (Q_T)
$. %In particular, for the
Let us start with the following definitions.
%\begin{defn}
%We say that a %bilinear control
%function $\alpha\in L^\infty(Q_T)$ is piecewise static, if $\alpha(\cdot,x)$ is piecewise constant in $t$ and $\alpha(t,\cdot)\in L^\infty(-1,1),\; t\in (0,T).$
%\end{defn}
\begin{defn}\label{static}
We say that a %bilinear control
function $\alpha\in L^\infty(Q_T)$ is {\it piecewise static}, if there exist $n\in\N,$ $c_i(x)\in L^\infty(-1,1)$ and $t_i\in (0,T), \,t_{i-1}<t_i,\, i=1,\dots,n$ with $t_0=0 \mbox{ and } t_n=T,$ 
such that $$\alpha(t,x)=c_1(x)\chi_{[t_{0},t_1]}(t)+\sum_{i=2}^n c_i(x)\chi_{(t_{i-1},t_i]}(t),$$ where $\chi_{[t_{0},t_1]}\,  \mbox{  and  }  \,\chi_{(t_{i-1},t_i]}$ are the indicator function of $[t_{0},t_1]$ and $(t_{i-1},t_i]$, respectively.
\end{defn}
\begin{defn}
We say that the system $(\ref{Psemilineare})$  is nonnegatively globally approximately controllable in $L^2(-1,1),$
% function $\alpha\in L^\infty(Q_T)$ is piecewise static, if $\alpha(\cdot,x)$ is piecewise constant in $t$ and $\alpha(t,\cdot)\in L^\infty(-1,1),\; t\in (0,T).$
if for every $%\forall
\varepsilon>0$ and for any nonnegative $%\forall
 u_0,\,u_d\in %L^2 (-1,1)
 L^2(-1,1),$ with $u_0\neq0$ %u_d\geq 0$ and any $%\forall
 %u_0\in L^2(-1,1)$ such that
there are a %\exists\,
$ T=T(\varepsilon,u_0,u_d)\geq 0$
and a bilinear control %\exists
$\alpha=
\alpha (t,x),\,\alpha
\in L^\infty(Q_T)$ such that for the corresponding strong solution (\footnote{See Definition \ref{strong}, for the precise definition of strong solutions.}) $u(t,x)$ of $(\ref{Psemilineare})$
we obtain
$$\|u(T,\cdot)-u_d\|_{L^2(-1,1)}\leq \varepsilon\,.$$
\end{defn}
\noindent The \textit{ nonnegative global approximate
controllability} results are obtained for the semilinear system
$(\ref{Psemilineare})$  in the following theorem. %\ref{T1SL}
\begin{thm}\label{T1}
The semilinear system $(\ref{Psemilineare})$  is nonnegatively globally approximately controllable in $L^2(-1,1),$ by means of piecewise static bilinear controls $\alpha%=\alpha(t,x)
.$
Moreover, the corresponding strong solution $(^4)$ to $(\ref{Psemilineare})$  remains nonnegative a.e. in $Q_T$.
\end{thm}
\noindent Moreover, we obtain the following result.
%{\theorem{
\begin{thm}\label{T1SL}%}{ \it
%\noindent
 For any $%\forall
 u_d\in %L^2 (-1,1)
 L^2(-1,1), u_d\geq 0$ and any $%\forall
 u_0\in L^2(-1,1)$ such that
\begin{equation}
\label{H2}
%\int^1_{-1}u_0 u_d dx>0,
\langle u_0, u_d \rangle_{L^2(-1,1)}>0,
\end{equation}
for every $%\forall
\varepsilon>0,$ there are %\exists\,
$ T=T(\varepsilon,u_0,u_d)\geq 0$
and a piecewise static bilinear control %\exists
$\alpha=
\alpha (t,x),\,\alpha
\in L^\infty(Q_T)$ such that
$$\|u(T,\cdot)-u_d\|_{L^2(-1,1)}\leq \varepsilon\,,$$
where $u$ is the strong solution $(^4)$ to $(\ref{Psemilineare}).$
%The linear system (\ref{P2}) is globally approximately
%controllable in the way that can be steered in $L^2 (-1,1)$ from any initial states $v_0 \in L^2 (-1,1)$ such that %(\ref{H2}) holds into any neighborhood of any desirable nonnegative target state $v_d\in L^2 (-1,1).$
\end{thm}
%%%%%%%%%%%%%%%%%%%%%%%%%%%%%%%%%%
%\def\baselinestretch{1}

\section{Well-posedness for nonlinear problems}
%\def\baselinestretch{1.66}
%%% ----------------------------------------------------------------------
%%% ----------------------------------------------------------------------
%\goodbreak
%\section{Introduction}
In this section, first we obtain embedding results for weighted Sobolev spaces (Section 3.2), then we prove the existence and uniqueness of the strong solution to nonlinear problem $(\ref{Psemilineare})$ (Section 3.5).
\subsection{The function spaces ${\cal{B}}(Q_T)$ and ${\cal{H}}(Q_T)$}
 %is concerned with the analysis of semilinear parabolic control systems 
%\subsection*{%Well-posedness in
%Weighted Sobolev spaces}
%Adesso introduciamo i seguenti Sobolev weighted spaces:
In order to deal with the well-posedness of nonlinear degenerate problem $(\ref{Psemilineare})$, it is necessary to introduce the weighted Sobolev spaces $H^1_a(-1,1)$ and $H^2_a(-1,1)$ (see also \cite{CFproceedings1} and \cite{CF3}).\\
%\textcolor{blue}{
%We denote by $H^1_a(-1,1)$ the space of all functions $u\in L^2(-1,1)$ such that $u$ is locally absolutely continuous in $(-1,1)$ and $\sqrt{a}\,u_x\in L^2(-1,1)$.\\
%Moreover, we define\\}
%\textcolor{red}{
We define
\begin{align*}
H^1_a(-1,1)&:=\{u\in L^2(-1,1)|\,u \text{ is locally absolutely continuous in } (-1,1) \text{ and }\;\sqrt{a}\,u_x\in L^2(-1,1)\},\\%}
%$$\hspace{-9.0cm} H^1_a(-1,1):=$$
%$$:=\{u\in L^2(-1,1): u \mbox{ locally absolutely continuous in } (-1,1), \sqrt{a} u_x \in L^2(-1,1)\}$$
%and
%\begin{multline*}
H^2_a(-1,1)&:=\{u\in H^1_a(-1,1)| \, au_x \in H^1 (-1,1)\},%\\
%\textcolor{blue}{=\{u \in L^2 (-1,1)| %u %\mbox{ locally absolutely continuous in }(-1,1),$$
%%$$\hspace{2.5cm}
%au\in H^1_0 (-1,1), \, au_x \in H^1 (-1,1)\mbox{  and }
%(a\,u_x)(\pm 1) %|_{x=\pm 1}
%=0\},}%PRG
%\,(\footnote{ \mbox{
%Si puï¿½ osservare che se
%One can show that if} $u\in H^2_a$ \mbox{ or } $u\in H^1_a \mbox{ then } (au_x)(t,x)|_{x=\pm 1}=0.$
%} )
\end{align*}
%\end{multline*}
respectively with the following norms
$$\|u\|_{1,a}^2:=\|u\|_{L^2(-1,1)}^2+|u|_{1,a}^2 \mbox{ and } \|u\|_{2,a}^2:=\|u\|_{1,a}^2+\|(au_x)_x\|^2_{L^2(-1,1)},$$
where $|u|_{1,a}^2:=\|\sqrt{a}u_x\|_{L^2(-1,1)}^2$ is a seminorm.\\
%\textcolor{red}{
 In \cite{CMV2}, see Proposition 2.1 and the Appendix, the authors prove the following result %that $D(A_0)=H^2_a(-1,1)$ %the following result
  (see also Lemma 2.5 in \cite{CMP}).
 \begin{prop}\label{caratH2}
For every $u\in H^2_a(-1,1)$ we have%=D(A_0)
%=\{u \in H^1_a(-1,1)| %u \mbox{ loc. abs. continuous in }(-1,1),%$$
%$$\hspace{2.5cm}
%au\in H^1_0 (-1,1), \,
%au_x \in H^1 (-1,1)\mbox{ and }
$$\lim_{x\rightarrow\pm1}a(x)u_x(x)%(a\,u_x)(\pm 1) %|_{x=\pm 1}
=0%\}
%\{u \in L^2 (-1,1)| %u %\mbox{ locally absolutely continuous in }(-1,1),$$
%%$$\hspace{2.5cm}
%au\in H^1_0 (-1,1), \, au_x \in H^1 (-1,1)\mbox{ and } \lim_{x\rightarrow\pm1}a(x)u_x(x)%(a\,u_x)(\pm 1) %|_{x=\pm 1}
%=0\}
\qquad\qquad\text{ and }\qquad\qquad
%and
au\in H^1_0 (-1,1)\;\; (\footnote{
%\textcolor{red}{
$H^1_0(-1,1)=\{u\in L^2(-1,1)| u_x\in L^2(-1,1) \text{ and } u(\pm1)=0\}.$%}
}).
$$ %\;\;\;\forall u\in H^2_a(-1,1).$$
 \end{prop}%}%PRG
$H^1_a(-1,1)$
%(\footnote{In \cite{CFproceedings1} we proved that the imbedding $H^{1}_a (-1,1)\hookrightarrow L^2(-1,1)$ is compact.})
and $H^2_a(-1,1)$ are Hilbert spaces with their natural scalar products, 
%\textcolor{red}{
and we denote %}
with $\langle\cdot,\cdot\rangle_{1,a}$ the scalar product of $H^1_a(-1,1)$.\\
\noindent In the following, we will sometimes use $\|\cdot\|,\;\langle\cdot,\cdot\rangle$ %$\Omega$
 instead of
%the bounded open interval (-1,1)
$\|\cdot\|_{L^2(-1,1)}, \langle\cdot,\cdot\rangle_{L^2(-1,1)}$, respectively, and $\|\cdot\|_\infty$ instead of $\|\cdot\|_{L^\infty(Q_T)}.$  \\
 %and the associated norms
%%Con le norme
%%respectively with the following norms
%$$\|u\|^2_{1,a}%_{H_a^1}
%:=\|u\|^2_{L^2 (-1,1)} \, + \,
%|u|^2_{1,a}%\|\sqrt{a}u_x\|^2_{L^2 (-1,1)}
%$$
%and
%$$\|u\|^2_{2,a}%_{H_a^2}
%:=\|u\|^2_{1,a}%_{H_a^1}
% \, + \, \|(au_x)_x\|^2_{L^2
%(-1,1)},$$
%where $|u|_{1,a}:=\,\|\sqrt{a}u_x\|_{L^2 (-1,1)}.$\\ %is a seminorm.\\
%\vspace{0.2cm}
 %the following result
%%%%%%%%%%%%%%%%%%%%%%%%%%%%%%%%%%%%%%%%%%%%%%%%%%%%%%%%%%%%%%%%%%%%%%%%%%%%%%%%%%%%%%%
%{\lemma{\label{sob2}}{ \it
%\begin{equation*}
%\qquad\qquad \mbox{ with compact embedding.}
%\end{equation*}
%}}
%%%%%%%%%%%%%%%%%%%%%%%%%%%%%%%%%%%%%%%%%%%%%%%%%%%%%%%%%%%%%%%%%%%%%%%%%%%%%%%%%
%%%%%%%%%%%%%%%%%%%%%%%%%%%%%%%%%%%%%%%%%%%%%%%%%%%%%%%%%%%%%%%%%%%%%%%%%%%%%%%%%%%%%%%%%%%%%%%%
%\vspace{1.0cm}

Given $T>0,$ let us define the function spaces:
$$\mathcal{B}(Q_T):=C([0,T];L^2(-1,1))\cap L^2(0,T;H^1_a (-1,1))$$
%let us define
with the following norm
\begin{equation*}
\label{normaB}
 \|u\|^2_{\mathcal{B}(Q_T)}= \sup_{t\in
[0,T]}\|u(t,\cdot)\|^2%_{L^2(-1,1)}
+2\int^T_{0}\int^1_{-1}a(x)u^2_x
dx\,dt\,,%,\,\,\forall u \in \mathcal{B}(Q_T)\,.
\end{equation*}
and
%Given $T>0,$ let us consider the function space %subspace of $\cal{B}(Q_T)$
$$
{\cal{H}}(Q_T):=L^{2}(0,T;%D(A)
H^2_a(-1,1))\cap H^{1}(0,T;L^2(-1,1))\cap C([0,T];H^{1}_a(-1,1))
$$
%let us define
with the following norm %of $\cal{B}(Q_T)$
\begin{equation*}
\label{normaH}
\|u\|^2_{\mathcal{H}(Q_T)}=%\|u\|_{\mathcal{B}(Q_T)},
%\sup_{t\in[0,T]}\left(\|u(t,\cdot)\|^2+\|\sqrt{a(\cdot)}u_x(t,\cdot)\|^2\right)+\int_0^T\left(\|u_x(t,\cdot)\|^2+\|(au_x)_x(t,\cdot)\|^2\right)\,dt\;\; \|u\|_{\mathcal{H}(Q_T)}=%\|u\|_{\mathcal{B}(Q_T)},
 \sup_{%t\in
 [0,T]}\left(\|u\|%_{L^2(-1,1)}
 ^2+\|\sqrt{a}u_x\|%_{L^2(-1,1)}
 ^2\right)+\int_0^T\left(\|u_t\|%_{L^2(-1,1)}
 ^2+\|(au_x)_x\|%_{L^2(-1,1)}
 ^2\right)\,dt.
 \;\;(\footnote{
 %\textcolor{red}{
 It's well known that this norm is equivalent to the Hilbert norm $$\displaystyle\||u|\|^2_{\mathcal{H}(Q_T)}=%\|u\|_{\mathcal{B}(Q_T)},
%\sup_{t\in[0,T]}\left(\|u(t,\cdot)\|^2+\|\sqrt{a(\cdot)}u_x(t,\cdot)\|^2\right)+\int_0^T\left(\|u_x(t,\cdot)\|^2+\|(au_x)_x(t,\cdot)\|^2\right)\,dt\;\; \|u\|_{\mathcal{H}(Q_T)}=%\|u\|_{\mathcal{B}(Q_T)},
 \int_0^T\left(\|u\|%_{L^2(-1,1)}
 ^2+\|\sqrt{a}u_x\|%_{L^2(-1,1)}
 ^2+\|u_t\|%_{L^2(-1,1)}
 ^2+\|(au_x)_x\|%_{L^2(-1,1)}
 ^2\right)\,dt.$$%}
 })%\;\; \forall u\in \cal{H}(Q_T). %= \sup_{t\in
%[0,T]}\|u(t,\cdot)\|^2_{L^2(-1,1)}+2\int^T_{0}\int^1_{-1}a(x)u^2_x
%dx\,.%,\,\,\forall u \in \mathcal{B}(Q_T)\,.
\end{equation*}
%%%%%%%%%%%%%%%%%%%%%%%%%%%%%%%%%%%%%%%%%%%%%%%%%%%%%%%%%%%%%%%%%%%%%%%%%%%%%%%%%%%%
%\vspace{0.5cm}
\begin{rem}
We observe that ${\cal{B}}(Q_T)$ and ${\cal{H}}(Q_T)$ are Banach %\colorbox{red}{
spaces %} 
(see, e.g., \cite{edw}). 
\end{rem}

\subsection{Some embedding theorems for weighted Sobolev spaces }
Let $\xi_a(x)=\int_0^x\frac{1}{a(s)}\,ds,$
 %Moreover,
then we have the following
\begin{lem}\label{sob1}
If $\xi_a%(x)=\int_0^x\frac{1}{a(s)}\,ds
\in L^p (-1,1),$ for some $p\geq 1,$ then
\begin{equation*}
H^{1}_a (-1,1)\hookrightarrow L^{2p}(-1,1)\,.% \qquad\qquad \forall 1<p<\frac{3-\alpha}{\alpha-1}
\end{equation*}
%and
Moreover,%, we have
$$\|u\|_{L^{2p}(-1,1)}\leq c\,%\sqrt{\|\xi_a\|_p}
\|u\|_{1,a},$$
where $c$ is a positive constant.
\end{lem}
%\begin{rem}
%$$\frac{3-\alpha}{\alpha-1}\geq 2$$
%\end{rem}
\begin{proof}
Let $u\in H^1_a(-1,1).$ %(theoremref{sob1}).
%\begin{multline}\label{S1.1}
%\int_0^1 |u(x)|^p\, dx= \int_0^1 (x^\frac{\alpha-2}{2}|u(x)|)(x^\frac{2-\alpha}{2}|u(x)|^{p-1})\, dx\:\leq\\
%\leq\left(\int_0^1 x^{\alpha-2}|u(x)|^2\, dx \right)^\frac{1}{2}\left(\int_0^1x^{2-\alpha}|u(x)|^{2(p-1)}\, %dx\right)^\frac{1}{2}\\
%\end{multline}
First, for every $x\in(-1,1),$ we have the following estimate
\begin{equation}\label{S1.2}
\!|u(x)-u(0)|= \left|\int_0^x u'(s)\, ds\right|\,
\leq
%\textcolor{red}{
\left|\int_0^x a(s)|u'(s)|^2 ds\right|^\frac{1}{2}\!\left|\int_0^x \frac{1}{a(s)} ds\right|^\frac{1}{2}\!\!\!\leq\sqrt{|\xi_a(x)|}\:|u|_{1,a}.
%}
\end{equation}
Moreover, keeping in mind that $\xi_a\in L^p(-1,1)$, we have
\begin{equation*}
\int_{-1}^1|u(0)|\,dx\leq\int_{-1}^1|u(x)-u(0)|\,dx+\int_{-1}^1|u(x)|\,dx
\leq\,|u|_{1,a}\int_{-1}^1\sqrt{
%\textcolor{red}{
|\xi_a(x)|
%}
}\,dx+\sqrt{2}\|u\|.
\end{equation*}
Thus,
\begin{equation}\label{S1.3}
|u(0)|\leq\,c_a\,|u|_{1,a}+
%\textcolor{red}{
\frac{\sqrt{2}}{2}%}
\|u\|
%\textcolor{red}{
\leq\max\Big\{c_a,\frac{\sqrt{2}}{2}\Big\}\|u\|_{1,a}
%}
, \qquad \mbox{ where } c_a=\frac{1}{2}\int_{-1}^1\sqrt{
%\textcolor{red}{
|\xi_a(x)|
%}
}\,dx.
\end{equation}
Finally, by (\ref{S1.2}) and (\ref{S1.3}) we have (\footnote{
%\textcolor{red}{
We remember that, for every $a,b\in[0,+\infty),$ the following numerical inequality holds true:
%$$(a+b)^q\leq 2^q(a^q+b^q), \text{ for every } q\geq0,$$
$$(a+b)^q\leq 2^{q-1}(a^q+b^q), \text{ for every } q\geq1.$$%}
}) % we obtain
\begin{multline*}
%$$
\int_{-1}^1 |u(x)|^{2p}\,dx\leq 
%c
%\textcolor{red}{
2^{2p-1}%}
\int_{-1}^1 \left(|u(x)-u(0)|^{2p}+|u(0)|^{2p}\right)\,dx\\
\leq 
%\textcolor{red}{
2^{2p-1}%}
\,|u|^{2p}_{1,a}\!\int_{-1}^1 
%\textcolor{red}{
|\xi_a(x)|%}
^p
%\xi_a^p(x)
\,dx+
%\textcolor{red}{
2^{2p}\Big(\max\Big\{c_a,\frac{\sqrt{2}}{2}\Big\}\Big)^{2p}\|u\|^{2p}_{1,a}.%}
%\,
%\textcolor{blue}{\left(|u|_{1,a}+\|u\|\right)^{2p}\!\!\!
%\leq c\,%\left(1+\|\xi_a\|_p^{p}\right)
%\|u\|^{2p}_{1,a}}
%$$%%PRG
\end{multline*}
%\textcolor{red}{
Since $\xi_a\in L^p(-1,1)$, applying H\"{o}lder inequality (\footnote{
%\textcolor{red}{
We note that 
$\int_{-1}^1|\xi_a(x)|^p\,dx\leq
\big(\int_{-1}^1\,dx\big)^{1-2p}\big(\int_{-1}^1|\xi_a(x)|^{\frac{1}{2}}\,dx\big)^{2p}=2^{1-2p}(2c_a)^{2p}=2c_a^{2p}.$
%}
}),
we deduce 
$$\int_{-1}^1 |u(x)|^{2p}\,dx\leq 2^{2p}c_a^{2p}
\,|u|^{2p}_{1,a}+2^{2p}\Big(\max\Big\{c_a,\frac{\sqrt{2}}{2}\Big\}\Big)^{2p}\|u\|^{2p}_{1,a}\leq 2^{2p}\Big(\max\Big\{c_a,\frac{\sqrt{2}}{2}\Big\}\Big)^{2p}\|u\|^{2p}_{1,a}.$$
%}
\end{proof}
%$%\qquad\qquad\qquad\qquad\qquad
%\hfill\blacksquare$

%So, you have trivially the following corollary
%\begin{cor}
%If $\xi_a(x)=\int_0^x\frac{1}{a(s)}\,ds\in L^p (-1,1), \quad p\geq 1$
%\begin{equation}
%H^{1}_a (-1,1)\hookrightarrow L^{2p}(-1,1)\,.% \qquad\qquad \forall 1<p<\frac{3-\alpha}{\alpha-1}
%\end{equation}
%and moreover, we have
%$$\|u\|_{L^{2p}(-1,1)}\leq c(a,p)\,%\sqrt{\|\xi_a\|_p}
%\|u\|_{1,a}.$$
%\end{cor}

%%%%%%%%%%%%%%%%%%%%%%%%%%%%%%%%%%%%%%%%%%%%%%%%%%%%%%%%%%%%%%%%%%%%%%%%%%%%%%%%%

%Now, in this paper we obtain the following result

%%%%%%%%%%%%%%%%%%%%%%%%%%%%%%%%%%%%%%%%%%%%%%%%%%%%%%%%%%%%%%%%%%%%%%%%%%%%%%%%%%
\begin{lem}\label{sob2}
Let $T>0.$ %\frac{1}{2}\leq p<2$ and let $\xi_a(x)=\int_0^x\frac{1}{a(s)}\,ds\in L^{\frac{p}{2-p}}(-1,1)$ be, then
If $\xi_a\in L^{\frac{p}{2-p}}(-1,1)$ for some $p\in\left[%\frac{1}{2}
%\textcolor{red}{
1%}
,2\right),$ then
\begin{equation*}
L^{2}(0,T;H^{1}_a (-1,1))\cap L^\infty(0,T;L^2(-1,1))\subset L^{2p}(Q_T)\;\;(\footnote{
\;\;
%\textcolor{red}{
$\displaystyle
\|u\|^2_{L^{2}(0,T;H^{1}_a (-1,1))}\!\!=\!
 \int_0^T\left(\|u\|^2+\|\sqrt{a}u_x\|^2\right)\,dt%\;\; \forall u\in \cal{H}(Q_T). %= \sup_{t\in
%[0,T]}\|u(t,\cdot)\|^2_{L^2(-1,1)}+2\int^T_{0}\int^1_{-1}a(x)u^2_x
%dx\,.%,\,\,\forall u \in \mathcal{B}(Q_T)\,.
\;\;\;\;\text{   and   }\;\;\;\; \|u\|^2_{L^\infty(0,T;L^2(-1,1))}\!=\!\ess_{%t\in
[0,T]}\|u\|^2\,.%_{L^2(-1,1)}
$%}
}) 
\end{equation*}
and %the following inequality holds
$$\|u\|_{L^{2p}(Q_T)}\leq c\,T^{\frac{1}{2p}\left(1-\frac{p}{2}\right)}\|u\|_{{\cal{B}}(Q_T)}
%\|u\|^{\frac{1}{2}}_{L^2\left(0,T;H^1_a(-1,1)\right)}\|u\|^{\frac{1}{2}}_{L^\infty\left(0,T;L^2(-1,1)\right)}
\,,$$
where $c$ is a positive constant.
\end{lem}
\begin{proof}
  For every $u\in L^{2}(0,T;H^{1}_a (-1,1))$ we have
\begin{%multline
equation*}
\int_{Q_T} |u|^{2p}\,dx\,dt = \int_0^T \int_{-1}^1 |u|^{p}\,|u|^{p}\,dx\,dt
\leq \int_0^T\left(\int_{-1}^1|u|^2\,dx\right)^{\frac{p}{2}}\,\left(\int_{-1}^1|u|^{\frac{2p}{2-p}}\,dx\right)^{\!\!\!\!\frac{2-p}{2}}\!\!\!\!dt.
%\leq c\, \|u_0\|^p_{L^2(-1,1)}\int_0^T \|u\|^p_{L^\frac{2p}{2-p}(-1,1)}\,dt\,.\\
\end{%multline
equation*}
Recalling that
%Then, keeping in mind that
 $u\in L^\infty(0,T;L^2(-1,1)),$ by Lemma \ref{sob1} we obtain
%the previous theorem and applying the H\"older inequality, we obtain
%\begin{multline*}
$$\int_{Q_T} |u|^{2p}\,dx\,dt\leq \|u\|^p_{L^\infty(0,T;L^2(-1,1))}\int_0^T \|u\|_{L^\frac{2p}{2-p}(-1,1)}^p\,dt
\leq c\,%\|\xi_a\|^\frac{p}{2}_{\frac{p}{2-p}}
\|u\|^p_{L^\infty(0,T;L^2(-1,1))}\int_0^T \|u\|_{1,a%H^1_a(-1,1)
}^p\,dt\:.
$$%\end{multline*}
Moreover, using H\"older's inequality, %and keeping in mind that $u\in L^2\left(0,T;H^1_a(-1,1)\right)$
 we have
\begin{equation*}
\int_0^T \|u\|_{H^1_a(-1,1)}^p\,dt\leq \left(\int_0^T\,dt\right)^{1-\frac{p}{2}}\left(\int_0^T\,\|u\|_{1,a%H^1_a(-1,1)
}^2\,dt\right)^{\frac{p}{2}}\leq T^{1-\frac{p}{2}}\|u\|_{L^2\left(0,T;H^1_a(-1,1)\right)}^p\,.
\end{equation*}
%The conclusion follows
From the last two inequalities, it follows that %we obtain\\
\begin{equation*}
\int_{Q_T} |u|^{2p}\,dx\,dt\leq c\,%\|\xi_a\|^\frac{p}{2}_{\frac{p}{2-p}}\,
T^{1-\frac{p}{2}} \,\|u\|_{L^2\left(0,T;H^1_a(-1,1)\right)}^p\, \|u\|^p_{L^\infty(0,T;L^2(-1,1))}
\leq c\,%\|\xi_a\|^\frac{p}{2}_{\frac{p}{2-p}}
\,T^{1-\frac{p}{2}}\|u\|^{2p}_{{{\cal{B}}}(Q_T)}.
%\hfill\blacksquare
\end{equation*}
\end{proof}
Taking $p=\frac{\vartheta+1}{2},\:%\textcolor{red}{
1%}%PRG
\leq\vartheta<3,$ in the previous lemma, we obtain the following corollary.
\begin{cor}\label{sob2cor}
Let $T>0.$ %\frac{1}{2}\leq p<2$ and let $\xi_a(x)=\int_0^x\frac{1}{a(s)}\,ds\in L^{\frac{p}{2-p}}(-1,1)$ be, then
If $\xi_a\in L^{\frac{1+\vartheta}{3-\vartheta}}(-1,1)$ for some %\colorbox{red}{
$\vartheta\in\left[
%\textcolor{red}{
1%}%PRG
,3\right),$
%}
 then
\begin{equation*}
\cal{B}(Q_T)\subset L^{1+\vartheta}(Q_T)
\end{equation*}
and %the following inequality holds
$$\|u\|_{L^{1+\vartheta}(Q_T)}\leq c\,T%\textcolor{red}{
^{\frac{3-\vartheta}{4(1+\vartheta)}}\|u\|_{\cal{B}(Q_T)}\,,%}
%\textcolor{blue}{^{\frac{1+\vartheta}{4(3-\vartheta)}}\|u\|_{\cal{B}(Q_T)}}
%\|u\|^{\frac{1}{2}}_{L^2\left(0,T;H^1_a(-1,1)\right)}\|u\|^{\frac{1}{2}}_{L^\infty\left(0,T;L^2(-1,1)\right)}
$$
where $c$ is a positive constant.
\end{cor}

%%%%%%%%%%%%%%%%%%%%%%%%%%%%%%%%%%%%%%%%%%%%%%%%%%%%%%%%%%%%%%%%%%%%%%%%%%%%%%%%%

%%%%%%%%%%%%%%%%%%%%%%%%%%%%%%%%%%%%%%%%%%%%%%%%%%%%%%%%%%%%%%%%%%%%%%%%%%%%%%%%%
\begin{lem}\label{lemma sob3}
Let $T>0,\: p\geq1.$ %\frac{1}{2}\leq p<2$ and let $\xi_a(x)=\int_0^x\frac{1}{a(s)}\,ds\in L^{\frac{p}{2-p}}(-1,1)$ be, then
If $\xi_a\in L^{2p-1}(-1,1),$ %for some $p\in\left[\frac{1}{2},2\right),$
then
\begin{equation*}
%L^{2}(0,T;D(A))\cap
H^{1}(0,T;L^2(-1,1))\cap L^\infty(0,T;H^{1}_a(-1,1))\subset L^{2p}(Q_T)\;\;\;\;(\footnote{
%\begin{equation*}
%\label{normaH}
\;\;$\displaystyle
%\textcolor{red}{
%
\|u\|^2_{H^{1}(0,T;L^2(-1,1))}\!\!=\!%\|u\|_{\mathcal{B}(Q_T)},
%\sup_{t\in[0,T]}\left(\|u(t,\cdot)\|^2+\|\sqrt{a(\cdot)}u_x(t,\cdot)\|^2\right)+\int_0^T\left(\|u_x(t,\cdot)\|^2+\|(au_x)_x(t,\cdot)\|^2\right)\,dt\;\; \|u\|_{\mathcal{H}(Q_T)}=%\|u\|_{\mathcal{B}(Q_T)},
 \sup_{%t\in
 [0,T]}\|u\|%_{L^2(-1,1)}
 ^2+\int_0^T\|u_t\|%_{L^2(-1,1)}
 ^2\,dt%\;\; \forall u\in \cal{H}(Q_T). %= \sup_{t\in
%[0,T]}\|u(t,\cdot)\|^2_{L^2(-1,1)}+2\int^T_{0}\int^1_{-1}a(x)u^2_x
%dx\,.%,\,\,\forall u \in \mathcal{B}(Q_T)\,.
\;\;\text{  and  }\;\; \|u\|^2_{L^\infty(0,T;H^{1}_a(-1,1))}\!=\!\sup_{%t\in
[0,T]}
\left(\|u\|^2+\|\sqrt{a}u_x\|^2\right)\,.%_{L^2(-1,1)}
% ^2+\int_0^T\|(au_x)_x\|%_{L^2(-1,1)}
% ^2\,dt.%\;\; \forall u\in \cal{H}(Q_T). %= \sup_{t\in
%}
$
%\end{equation*}%PRG
}
)
\end{equation*}
and %the following inequality holds
$$
\|u\|_{L^{2p}(Q_T)}%\,dx\,dt
%\\
\leq c\,%\|\xi_a\|^\frac{p}{2}_{\frac{p}{2-p}}\,
%^{1-\frac{p}{2}}
 T^{\frac{1}{2p}}\,\|u\|^{\frac{1}{2p}}_{H^1(0,T;L^2(-1,1))}\,\|u\|_{L^\infty(0,T;H^1_a(-1,1))}^{1-\frac{1}{2p}},
% %_{H^1\left(0,T;H^1_a(-1,1)\right)}\, \|u\|^{2p-1}_{C([0,T];H^{1}_a(-1,1))}\,\\
%\|u\|_{L^{2p}(Q_T)}\leq c(a,p)\,T^{\frac{1}{2p}\left(1-\frac{p}{2}\right)}\|u\|_{B(Q_T)}
%\|u\|^{\frac{1}{2}}_{L^2\left(0,T;H^1_a(-1,1)\right)}\|u\|^{\frac{1}{2}}_{L^\infty\left(0,T;L^2(-1,1)\right)}
%\,.
$$
where $c$ is a positive constant.
\end{lem}
\begin{proof}
  For every $u\in H^{1}(0,T;L^2(-1,1))\cap L^\infty(0,T;H^{1}_a(-1,1))
%C([0,T];H^{1}_a(-1,1)) %L^{2}(0,T;H^{1}_a (-1,1))
$ we have
\begin{equation*}
\int_{Q_T} |u|^{2p}\,dx\,dt = \int_0^T \int_{-1}^1 |u|\,|u|^{2p-1}\,dx\,dt
\leq \int_0^T\left(\int_{-1}^1|u|^2\,dx\right)^{\frac{1}{2}}\,\left(\int_{-1}^1|u|^{4p-2}\,dx\right)^{\frac{1}{2}}dt\,.
%\leq c\, \|u_0\|^\vartheta_{L^2(-1,1)}\int_0^T \|u\|^\vartheta_{L^\frac{2\vartheta}{2-\vartheta}(-1,1)}\,dt\,.\\
\end{equation*}
Recalling that
%Then, keeping in mind that
 $u\in H^{1}(0,T;L^2(-1,1)),$ by the Lemma \ref{sob1} and since $\xi^a\in L^{2p-1}(-1,1)$, we obtain
 %the previous theorem and applying the H\"older inequality, we obtain
%\begin{multline*}
$$\int_{Q_T} |u|^{2p}\,dx\,dt\leq \|u\|_{H^{1}(0,T;L^2(-1,1))}\int_0^T \|u\|_{L^{4p-2}(-1,1)}^{2p-1}\,dt
\leq c\,%\|\xi_a\|^\frac{p}{2}_{\frac{p}{2-p}}
\|u\|_{H^{1}(0,T;L^2(-1,1))}\int_0^T \|u\|_{1,a
%H^1_a(-1,1)
}^{2p-1}\,dt\:.
$$%\end{multline*}
%Moreover, using H\"older's inequality %and keeping in mind that $u\in L^2\left(0,T;H^1_a(-1,1)\right)$
%, we have
%\begin{multline*}
%\int_0^T \|u\|_{H^1_a(-1,1)}^p\,dt\leq %\left(\int_0^T\,dt\right)^{1-\frac{p}{2}}\left(\int_0^T\,\|u\|_{H^1_a(-1,1)}^2\,dt\right)^{\frac{p}{2}}\\
%\leq T^{1-\frac{p}{2}}\|u\|_{L^2\left(0,T;H^1_a(-1,1)\right)}^p\,.
%\end{multline*}
%%The conclusion follows
From the last inequality, it follows that %we obtain\\
\begin{equation*}
\int_{Q_T} |u|^{2p}\,dx\,dt
\leq c\,%\|\xi_a\|^\frac{p}{2}_{\frac{p}{2-p}}\,
T%^{1-\frac{p}{2}}
 \,\|u\|_{H^1\left(0,T;L^2(-1,1)\right)}\, \|u\|^{2p-1}_{L^\infty(0,T;H^{1}_a(-1,1))}.\,
%\leq c(a,p)\,%\|\xi_a\|^\frac{p}{2}_{\frac{p}{2-p}}
%\,T^{1-\frac{p}{2}}\|u\|^{2p}_{B(Q_T)}.
%\hfill\blacksquare
\end{equation*}
\end{proof}
By Lemma \ref{lemma sob3} one directly obtains the following.
\begin{cor}\label{sob3}
Let $T>0,\: \vartheta\geq1.$ %\frac{1}{2}\leq p<2$ and let $\xi_a(x)=\int_0^x\frac{1}{a(s)}\,ds\in L^{\frac{p}{2-p}}(-1,1)$ be, then
If $\xi_a\in L^{2\vartheta-1}(-1,1),$ %for some $p\in\left[\frac{1}{2},2\right),$
then
\begin{equation*}
%L^{2}(0,T;D(A))\cap
{\cal{H}}(Q_T)\subset L^{2\vartheta}(Q_T)
\end{equation*}
and %the following inequality holds
$$
\|u\|_{L^{2\vartheta}(Q_T)}%\,dx\,dt
%\\
\leq c\,%\|\xi_a\|^\frac{p}{2}_{\frac{p}{2-p}}\,
%^{1-\frac{p}{2}}
 T^{\frac{1}{2\vartheta}}\,\|u\|_{{\cal{H}}(Q_T)},
% %_{H^1\left(0,T;H^1_a(-1,1)\right)}\, \|u\|^{2p-1}_{C([0,T];H^{1}_a(-1,1))}\,\\
%\|u\|_{L^{2p}(Q_T)}\leq c(a,p)\,T^{\frac{1}{2p}\left(1-\frac{p}{2}\right)}\|u\|_{B(Q_T)}
%\|u\|^{\frac{1}{2}}_{L^2\left(0,T;H^1_a(-1,1)\right)}\|u\|^{\frac{1}{2}}_{L^\infty\left(0,T;L^2(-1,1)\right)}
%\,.
$$
where $c$ is a positive constant.
\end{cor}
%%%%%%%%%%%%%%%%%%%%%%%%%%%%%%%%%%%%%%%%%%%%%%%%%%%%%%%%%%%%%%%%%%%%%%%%%%%%%%%%%
\subsection{Existence and uniqueness of solutions of linear problems}
%\vspace{0.5cm}
First,
we recall an existence uniqueness result for the linear problems corresponding to $(\ref{Psemilineare})$, obtained in \cite{CMP} (see also \cite{ACF} and \cite{CF3}),
defined by
 \begin{equation}\label{D(A0)}
   \left\{\begin{array}{l}
\displaystyle{D(A_0)=
%\textcolor{red}{=\{u \in H^1_a(-1,1)| %u \mbox{ loc. abs. continuous in }(-1,1),%$$
%%$$\hspace{2.5cm}
%%au\in H^1_0 (-1,1), \,
%au_x \in H^1 (-1,1)\mbox{ and }\lim_{x\rightarrow\pm1}a(x)u_x(x)%(a\,u_x)(\pm 1) %|_{x=\pm 1}
%=0\}}%PRG
%
%%\textcolor{red}{=\{u \in L^2 (-1,1)| u \mbox{ loc. abs. continuous in }(-1,1),%$$
%%%$$\hspace{2.5cm}
%%au\in H^1_0 (-1,1), \, au_x \in H^1 (-1,1)\mbox{ and } \lim_{x\rightarrow\pm1}a(x)u_x(x)%(a\,u_x)(\pm 1) %|_{x=\pm 1}
%%=0\}}%PRG
%\textcolor{blue}{
H^2_a (-1,1)%}
 }\\ [2.5ex]
\displaystyle{A_0u=(au_x)_x\,, \,%\qquad\qquad\qquad\qquad\quad
\,\, \forall \,u \in D(A_0)}\,.
\end{array}\right.
 \end{equation}
% \textcolor{red}{where
%$H^1_0(-1,1)=\{u\in L^2(-1,1)| u_x\in L^2(-1,1) \text{ and } u(\pm1)=0\}$.}\\%PRG
For the following linear results it is sufficient that the diffusion coefficient
 %function
  %$a(x)$
  $a(\cdot)$ satisfy %respectively 
  the assumption %$(A.2)$ and 
   $(A.4)$ with $\xi_a%(x)=\int_0^x \frac{ds}{a(s)}
\in L^{1}(-1,1),$ instead of the condition $(\ref{Lintrod})$.  %\:(\footnote{For the assumptions $(A.1)-(A.4)$ see Section 1. }).
%%%%%%%%%%%%%%%%%%%%%%%%%%%%%%%
Next, given $\alpha\in L^\infty (-1,1),$ let us introduce the operator
 \begin{equation}\label{D(A)}
 \left\{\begin{array}{l}
\displaystyle{D(A)=D(A_0)%\textcolor{red}{\:=H^2_a(-1,1)} 
}\\ [2.5ex]
\displaystyle{A = A_0 + \alpha I\, %\qquad\qquad\qquad\qquad\quad
%\,\, \forall \,u \in D(A_0)
}\,.
\end{array}\right.
 \end{equation}

We consider
%Observe that
the following linear problem %(\ref{Psemilineare}) can be recast
in the Hilbert space $L^2(-1,1)$
\begin{equation}\label{Ball}
 \left\{\begin{array}{l}
\displaystyle{u^\prime(t)=A\,u(t)\,+g(t),\qquad  t>0 }\\ [2.5ex]
\displaystyle{u(0)=u_0\, %\qquad\qquad\qquad\qquad\quad
%\,\, \forall \,u \in D(A_0)
}~,
\end{array}\right.
\end{equation}
where $A$ is the operator in (\ref{D(A)}), $g\in L^1(0,T;L^2(-1,1)),\,u_0\in L^2(-1,1)$.\\ %(\footnote{If $g(t)\equiv0$ we obtain the problem (\ref{P2slineare}).}).\\
%\vspace{0.5cm}

We recall that a \textit{weak solution} of (\ref{Ball}) is a function $u\in C^0([0,T];L^2(-1,1))$ such that for every $v\in D(A^*)$ ($A^*$ denotes the adjoint of A) the function $\langle u(t),v\rangle$ is absolutely continuous on $[0,T]$
and
$$\frac{d}{dt}\langle u(t),v\rangle=\langle u(t),A^*v\rangle\,+\langle g(t),v\rangle,$$
for almost all $t\in [0,T]$ (see \cite{Ba}). %(\footnote{$A^*$ denotes the adjoint of A.})
\\
For every $\alpha\in L^\infty%(0,T)\times
(-1,1)$ (\footnote{By repeated applications of %Proposition \ref{MaxReg}
this result, one can obtain an existence and uniqueness result when $\alpha$ is piecewise static (%$\alpha(\cdot,x)$ piecewise constant in $t,$ and $\alpha(t,\cdot)\in L^\infty(-1,1), \forall t\in(0,T)$
see Definition \ref{static}). The same result holds for $\alpha\in L^\infty(Q_T),$ but for the purposes of the present paper the piecewise static case will suffice.}) %(\footnote{See also note $(^7)$. The same remark applies to the present context. %In this result we can make a similar observation.
%})
and every $u_0 \in L^2(-1,1)$, there exists a unique weak solution
%$u\in \mathcal{B}(Q_T) %C^0([0,T];L^2(-1,1))\cap L^2(0,T;H^1_a (-1,1))
%$
of (\ref{Ball}), which is given by the following representation %e^{tA}u_0
$e^{tA}u_0+\int_0^t e^{(t-s)A} g(s)\,ds,$ \,$t\in[0,T]$ \,(see also \cite{CFproceedings1}).\\
Now, using a \textit{maximal regularity} result
in the Hilbert space $L^2(-1,1)$(\footnote{By \textit{maximal regularity} we mean that $u^\prime$ and $Au$ have the same regularity of $g$.}), 
%\textcolor{red}{
by Theorem 3.1 in Section 3.6.3 of \cite{BDDM1}, pp. $79-82,$ we derive the following result (see also \cite{CV} and \cite{CF3}).%}
\begin{prop}\label{MaxReg}
Given $T>0$ and $g\in L^2(0,T;L^2(-1,1))%=L^2(Q_T)
$\,(\footnote{We observe that $L^2(0,T;L^2(-1,1))=L^{2}(Q_{T}).$}). %\vartheta >1, \xi_a\in L^{2\vartheta-1}(-1,1)$ %$u\in{{\cal{H}}_T}$ %$u_0\in H^1_a(-1,1),$ $\alpha\in L^\infty(Q_T)$ and
For every $\alpha\in L^\infty %(0,T)\times
(-1,1)(^{11})$ and every $u_0 \in H^1_a(-1,1)$, there exists a unique
solution
$u\in{\cal{H}}(Q_T) %C^0([0,T];L^2(-1,1))\cap L^2(0,T;H^1_a (-1,1))
$
of (\ref{Ball}). %, which is given by $e^{tA}u_0$. %\,(see \cite{Ba}).
Moreover, a positive constant $C_0(T)$ exists (nondecreasing in $T$), such that the following inequality holds
$$\|u\|_{{\cal{H}}(Q_T)}\leq C_0(T)\left[\|u_0\|_{1,a}+\|g\|_{L^{2}(Q_{T})}\right].$$
\end{prop}
%\textcolor{blue}{
%\begin{defn}
%The solution $u\in{\cal{H}}(Q_T)$ of Proposition \ref{MaxReg} above is called \textit{strict solution}.
%\end{defn}
%}

\subsection{Some results for singular Sturm-Liouville problems}

%%%%%%%%%%%%%%%%%%%%%%%%%%%%%%%%%%%%%%%%%%%%%%%%%%%%%%%%%%%%%%%
In \cite{CFproceedings1}, in collaboration with P. Cannarsa, we prove the following results (see also \cite{CF3}).
\begin{prop}\label{compact imbedding}
Assume that $\xi_a\in L^1(-1,1),$ where $\xi_a(x)=\int_0^x\frac{ds}{a(s)}.$ Then,
$$H^1_a(-1,1)\hookrightarrow L^2(-1,1)\qquad \mbox{ with compact embedding }.$$
\end{prop} 
Let $A = A_0 + \alpha I,$ where the operator $A_0$ is defined in (\ref{D(A0)}) and $\alpha\in L^\infty (-1,1).$
Since $A$ is self-adjoint and $D(A)\hookrightarrow L^2(-1,1)$ is compact (see Proposition \ref{compact imbedding}), we have the following (see also \cite{BR}).
%\vspace{-0.5cm}
\begin{lem}\label{spectrum}
There exists an increasing sequence $\{\lambda_k\}_{k\in\N},$ with
$\lambda_k\longrightarrow +\infty\, \mbox{ as } \, k \, \rightarrow\infty\,,$
such that the eigenvalues of $A$ are given by $\{-\lambda_k\}_{k\in\N}$, and the corresponding eigenfunctions $\{\omega_k\}_{k\in\N}$ form a complete orthonormal system in $L^2(-1,1)$.
\end{lem}
\begin{rem}\label{Legendre}
In the case %that the coefficient
$a(x)=1-x^2$, so that $A_0=\left((1-x^2)u_x\right)_x,$ then the orthonormal eigenfunctions of $A_0$ are reduced to Legendre's polynomials $P_k(x)$, and the eigenvalues are $\mu_k=(k-1)k, k\in\N.$ $P_k(x)$ is equal to $\sqrt{\frac{2}{2k-1}} L_k(x),$ where $L_k(x)$ is assigned by \textit{Rodrigues's
formula}:
$$L_k(x)=\frac{1}{2^{k-1} (k-1)!} \frac{d}{dx^{k-1}} \, (x^2-1)^{k-1} \qquad (k \geq 1). $$
\end{rem}
%\vspace{0.5cm}
%\vspace{0.5cm}
In \cite{CFproceedings1} (see also \cite{CF3}) we obtain the following result.
% ???
\begin{lem}\label{Autof}
Let $v\in C^\infty([-1,1]), v >0$ on $[-1,1],$ let $\alpha_*(x)=-\frac{(a(x)v_{x}(x))_x}{v(x)},\, x\in
(-1,1).$ Let A be the operator defined in (\ref{D(A)}) with $\alpha=\alpha_*$
\begin{equation}\label{operalfastella}
\left\{\begin{array}{l}
\displaystyle{D(A)=H^2_a (-1,1)}\\ [2.5ex]
\displaystyle{A = A_0 + \alpha_* I}~,
%\displaystyle{v(0,x)=v_0 (x) \,\qquad\qquad\qquad\qquad\,\,\,\,\, x\in (-1,1)}~,
\end{array}\right.
\end{equation}
%where the operator $A_0$ is defined in (\ref{D(A0)}).
and let $\{\lambda_k\}, \{\omega_k\}$ be the eigenvalues and eigenfunctions of $A,$ respectively, given by Lemma \ref{spectrum}.
Then $$\lambda_1=0\,\, \mbox{ and } \,\,|\omega_1|=\frac{v}{\|v\|%_{L^2(-1,1)}
}.$$
Moreover, $\frac{v}{\|v\|%_{L^2(-1,1)}
}$ and $-\frac{v}{\|v\|%_{L^2(-1,1)}
}$ are the only eigenfunctions of $A$ with norm $1$ that do not change sign in $(-1,1)$.
\end{lem}

\begin{rem}
This problem %(\ref{operalfastella})
 is equivalent to the following singular %differential
 Sturm-Liouville problem
\begin{equation*}%\label{stella}
\left\{\begin{array}{l}
\displaystyle{(a(x) \omega_x)_x +\alpha_* (x)\omega+\lambda\,\omega=0\,\,\qquad \mbox{in} \qquad (-1,1) %\,=\,(0,T)\times(-1,1)
}\\ [2.5ex]
\displaystyle{a(x)\omega_x(x)|_{x=\pm 1} = 0\,\,%\qquad\qquad\qquad\,\, %t\in (0,T) \qquad\qquad\qquad\qquad %(\ref{P2})
\,\,\, }~.
%\displaystyle{v(0,x)=v_0 (x) \,\qquad\qquad\qquad\qquad\,\,\,\,\, x\in (-1,1)}~,
\end{array}\right.
\end{equation*}
\end{rem}
\noindent The proof of Lemma \ref{Autof} is recalled in \ref{PSSL}.

\subsection{Existence and uniqueness of solutions of semilinear problems}

Observe that the nonlinear problem $(\ref{Psemilineare})$  can be recast in the Hilbert space $L^2(-1,1)$ as
\begin{equation}\label{NL}
 \left\{\begin{array}{l}
\displaystyle{u^\prime(t)=A\,u(t)+\phi(u)\,,\qquad  t>0 }\\ [2.5ex]
\displaystyle{u(0)=u_0\, %\qquad\qquad\qquad\qquad\quad
%\,\, \forall \,u \in D(A_0)
}~,
\end{array}\right.
\end{equation}
where $A$ is the operator defined in (\ref{D(A)}), $\alpha\in L^\infty(-1,1),$ %the initial date
$u_0\in L^2(-1,1),$ and, for every $u\in {\cal{B}(Q_T)},%L^2(-1,1)
$ %belongs to $H^1_a(-1,1)$ and
\begin{equation}\label{phimap}%PRG $$
\phi(u)(t,x):=f(t,x,u(t,x)),\;\;\forall (t,x)\in Q_T. 
\end{equation}%$$ PRG
%\textcolor{red}{
By the next lemmas (Lemma \ref{f in L2 cor} and Lemma \ref{f in L2}) we will deduce the following theorem.
\begin{thm}\label{loclip}
Let $T>0, \;1\leq\vartheta<3,\: \xi_a\in  L^{q_\vartheta}(-1,1),$ where $\,q_\vartheta=\max\Big\{\frac{1+\vartheta}{3-\vartheta}, 2\vartheta-1\Big\}.$
%{\frac{1+\vartheta}{3-\vartheta}}(-1,1),$ 
%and let
%$u\in {\cal{B}}(Q_T).$
Let $f:Q_T\times\R\rightarrow \R$ be a function that satisfies assumption $(A.3),$ then 
 $\phi:{\cal{B}(Q_T)}\longrightarrow L^{1+\frac{1}{\vartheta}}(Q_T)$ 
 %\textcolor{red}{
is a locally Lipschitz continuous map %} %PRG
 and $\phi ({\cal{H}(Q_T)})\subseteq L^2(Q_T).$
 \end{thm}%}
%\textcolor{red}{
We start with the following lemma.%} %PRG
\begin{lem}\label{f in L2}
%Sia
Let $T>0, \vartheta\geq1, \xi_a\in L^{2\vartheta-1}(-1,1),$ %$u\in{{\cal{H}}_T}$ %$u_0\in H^1_a(-1,1),$ $\alpha\in L^\infty(Q_T)$ and
%la soluzione di
%the solution
and let
$u\in {\cal{H}}(Q_T).$ %of system \eqref{P2}.
Let %us consider the function
$f:Q_T\times\R\rightarrow \R$ be a function that satisfies assumption $(A.3)$(\footnote{%\colorbox{red}{
We observe %} 
that the assumption $(H.3)$ of Appendix B would be sufficient to place of $(A.3)$.}). %there exists %$\vartheta\,>1$ and
%$\bar{\gamma}_0>0$ such that
Then, the function
$(t,x)\longmapsto f(t,x,u(t,x))$
belongs to $L^2(Q_T)$ and the following estimate holds
%\begin{equation*}%\label{superl}
%|f(t,x,u(t,x))| \leq \bar{\gamma}_0 \,|u(t,x)|^\vartheta, \mbox{ for  a.e. } (t,x)\in Q_T,
%\end{equation*}
%then
%$$f(t,x,\cdot)\circ u(t,x)\in L^2(Q_T).$$
%Moreover, we have the following estimate
$$\int_{Q_T}|f(t,x,u(t,x))|^2\,dx\,dt\leq c\,T\,%\|u\|^{2\vartheta}_{{\cal{H}}(Q_T)}
\|u\|%^{\frac{1}{2p}}
_{H^1(0,T;L^2(-1,1))}\,\|u\|_{L^\infty(0,T;H^1_a(-1,1))}^{2\vartheta-1},$$
for some positive constant $c.$
%soddisfa la seguente disuguaglianza:
%satisfies the following inequality
%\begin{equation}\label{LA}
%%\|u\|_{L^2(Q_T)}\,,
%\|u\|_{B(Q_T)}\leq \nu_T\,e^{\|\alpha^+\|_\infty T}\,\|u_0\|_{L^2(-1,1)}\,,
%\end{equation}
%where $\alpha^+$ %(x)=\max\{\alpha(x),0\},\;x\in (-1,1)$ is
%denotes the positive part of $\alpha\,.$ \qquad (\footnote{ We recall that  $\nu_T=e^{\int_0^T \nu(t)\,dt}.$ })
\end{lem}
\begin{proof}
 %(lemref{L2})\\
By Lemma \ref{lemma sob3}, since $\xi_a\in L^{2\vartheta-1}(-1,1)$ then $u\in L^{2\vartheta}(Q_T).$ By (\ref{Superlinearit}) (see assumption (A.3)) we obtain
$$%\begin{multline*}
\int_{Q_T}|f(t,x,u(t,x))|^2\,dx\,dt
\leq \gamma_0^2\int_{Q_T}|u|^{2\vartheta}\,dx\,dt\;
\leq 
%\textcolor{red}{
k%}
%\textcolor{blue}{c}
\,T\,\|u\|%^{\frac{1}{2p}}
_{H^1(0,T;L^2(-1,1))}\,\|u\|_{L^\infty(0,T;H^1_a(-1,1))}^{2\vartheta-1}%\|u\|^{2\vartheta}_{}
\,<+\infty,
$$%\end{multline*}
from wich the conclusion follows.
%$\hfill\blacksquare$\\
\end{proof}
\begin{cor}\label{f in L2 cor}
%Sia
Let $T>0, \vartheta \geq1, \xi_a\in L^{2\vartheta-1}(-1,1),$ %$u\in{{\cal{H}}_T}$ %$u_0\in H^1_a(-1,1),$ $\alpha\in L^\infty(Q_T)$ and
%la soluzione di
%the solution
and let
$u\in {\cal{H}}(Q_T).$ %of system \eqref{P2}.
Let %us consider the function
$f:Q_T\times\R\rightarrow \R$ be a function that satisfies assumption $(A.3)\,(^{14}).$ %there exists %$\vartheta\,>1$ and
Then, 
we have the following estimate
$$\int_{Q_T}|f(t,x,u(t,x))|^2\,dx\,dt\leq c\,T\,\|u\|^{2\vartheta}_{{\cal{H}}(Q_T)},$$
for some positive constant $c.$
\end{cor}

%%%%%%%%%%%%%%%%%%%%%%%%%%%%%%%%%%%%%%%%%%%%%%%%%%%%%%%%%%%%%%%%%%%%%%%%%%%%%%%%
\begin{lem}\label{l3}
Let $T>0, 1\leq\vartheta<3,\: \xi_a\in L^{\frac{1+\vartheta}{3-\vartheta}}(-1,1).$ 
%\textcolor{blue}{and let
%$u\in {\cal{B}}(Q_T).$}
Let $f:Q_T\times\R\rightarrow \R$ be a function that satisfies assumption $(A.3).$ 
Then, 
%\textcolor{red}{
\begin{enumerate}
\item %\textcolor{red}{
for every $u\in {\cal{B}}(Q_T),$%}
the function
$(t,x)\longmapsto f(t,x,u(t,x))$
belongs to $L^{1+\frac{1}{\vartheta}}(Q_T)$ and the following estimate holds
$$\int_{Q_T}|f(t,x,u(t,x))|^{1+\frac{1}{\vartheta}}\,dx\,dt\leq c\,T^{\frac{3-\vartheta}{4}}
\|u\|^{\vartheta+1}
_{\cal{B}(Q_T)}\,%\|u\|_{L^\infty(0,T;H^1_a(-1,1))}^{2\vartheta-1}
,$$
for some positive constant $c;$
\item
 $\phi:{\cal{B}(Q_T)}\longrightarrow L^{1+\frac{1}{\vartheta}}(Q_T)$(\footnote{%\textcolor{red}{
 The map $\phi$ is defined in (\ref{phimap}).%}
 }) is a locally Lipschitz continuous map
 and, for every $R>0,$ the following estimate holds
 \begin{equation}\label{llc}
 \|\phi(u)-\phi(v)\|_{L^{1+\frac{1}{\vartheta}}(Q_T)}\leq C_R(T)\|u-v\|_{\cal{B}(Q_T)}, \; \forall u,v\in{\cal{B}(Q_T)}, \|u\|_{{\cal{B}(Q_T)}}\leq R, \|v\|_{{\cal{B}(Q_T)}}\leq R\,,
 \end{equation}
 where $C_R(T)$ is a positive constant increasing in $T$.
\end{enumerate}%}
\end{lem}
\begin{proof}
 %(lemref{L2})\\
By %\textcolor{red}{
Corollary \ref{sob2cor}, %} %PRG
since $\xi_a\in L^{\frac{1+\vartheta}{3-\vartheta}}(-1,1),$ then %\textcolor{red}{
$u\in L^{1+%\frac{1}
 {\vartheta}}(Q_T).$ %} %PRG
By (\ref{Superlinearit}) (see assumption (A.3)) we obtain
\begin{%multline
equation*}
\int_{Q_T}|f(t,x,u(t,x))|^{1+\frac{1}{\vartheta}}\,dx\,dt
\leq \gamma_0^{1+\frac{1}{\vartheta}}\int_{Q_T}|u|^{\vartheta(1+\frac{1}{\vartheta})}\,dx\,dt
\leq
%\textcolor{red}{
k
%}
%\textcolor{blue}{c} 
\,T^{\frac{3-\vartheta}{4}}
\|u\|^{\vartheta+1}_{\cal{B}(Q_T)} 
%c\,T\,\|u\|_{}%\,\|u\|_{L^\infty(0,T;H^1_a(-1,1))}^{2\vartheta-1}
\,<+\infty,
\end{equation*}
%multline
from wich %\textcolor{red}{
 the point 1.) follows.\\%}\\
%\textcolor{red}{ 
By \eqref{Remlip} (see Remark \ref{rem1}), applying Corollary \ref{sob2cor},
 %and Proposition \ref{uni},
we have
%$\hfill\blacksquare$\\
%\end{proof} %PRG
%Finally, we can obtain the following result
%\begin{lem}
%Let $T>0, 1\leq\vartheta<3,\: \xi_a\in L^{\frac{1+\vartheta}{3-\vartheta}}(-1,1),$ 
%and let
%$u\in {\cal{B}}(Q_T).$
%Let $f:Q_T\times\R\rightarrow \R$ be a function that satisfies assumption $(A.3).$ 
%Then, the function
%$(t,x)\longmapsto f(t,x,u(t,x))$
%belongs to $L^{1+\frac{1}{\vartheta}}(Q_T)$ and the following estimate holds
%$$\int_{Q_T}|f(t,x,u(t,x))|^{1+\frac{1}{\vartheta}}\,dx\,dt\leq c\,T^{\frac{3-\vartheta}{4}}
%\|u\|^{\vartheta+1}
%_{\cal{B}(Q_T)}\,%\|u\|_{L^\infty(0,T;H^1_a(-1,1))}^{2\vartheta-1}
%,$$
%for some positive constant $c.$
%\end{lem}
%\begin{proof} %PRG
\begin{multline*}
\!\!\!\!\!\!\|\phi(u)-\phi(v)\|^{1+\frac{1}{\vartheta}}_{L^{1+\frac{1}{\vartheta}}(Q_T)}\!\!\!=\!\!\int_{Q_T}|f(t,x,u)-f(t,x,v)|^{1+\frac{1}{\vartheta}}\,dx\,dt\;
\!\!\leq\!
%\gamma_1^{1+\frac{1}{\vartheta}}
c\int_{Q_T}(1+|u|^{\frac{\vartheta^2-1}{\vartheta}}\!\!+|v|^{\frac{\vartheta^2-1}{\vartheta}})%^{1+\frac{1}{\vartheta}}
|u-v|^{1+\frac{1}{\vartheta}}\!dx\,dt\;\,\\
\leq c\Big(\int_{Q_T}(1+|u|^{\vartheta+1}+|v|^{\vartheta+1})\,dx\,dt\Big)^{1-\frac{1}{\vartheta}}\;\Big(\int_{Q_T}|u-v|^{\vartheta+1}\,dx\,dt\Big)^{\frac{1}{\vartheta}}\\
\leq c\Big(T^{1-\frac{1}{\vartheta}}+\|u\|_{L^{\vartheta+1}(Q_T)}^{\frac{\vartheta^2-1}{\vartheta}}+\|v\|_{L^{\vartheta+1}(Q_T)}^{\frac{\vartheta^2-1}{\vartheta}}\Big)\|u-v\|_{L^{\vartheta+1}(Q_T)}^{1+\frac{1}{\vartheta}}\\
\leq c T^{\frac{3-\vartheta}{4\vartheta}}\Big(T^{1-\frac{1}{\vartheta}}+T^{\frac{(3-\vartheta)(\vartheta-1)}{4\vartheta}}\|u\|_{{\cal{B}}(Q_T)}^{\frac{\vartheta^2-1}{\vartheta}}+T^{\frac{(3-\vartheta)(\vartheta-1)}{4\vartheta}}\|v\|_{{\cal{B}}(Q_T)}^{\frac{\vartheta^2-1}{\vartheta}}\Big)\|u-v\|_{{\cal{B}}(Q_T)}^{1+\frac{1}{\vartheta}}\\
=c T^{\frac{3\vartheta-1}{4\vartheta}}\Big(1+T^{\frac{3-\vartheta}{4}}\|u\|_{{\cal{B}}(Q_T)}^{\frac{\vartheta^2-1}{\vartheta}}+T^{\frac{3-\vartheta}{4}}\|v\|_{{\cal{B}}(Q_T)}^{\frac{\vartheta^2-1}{\vartheta}}\Big)\|u-v\|_{{\cal{B}}(Q_T)}^{1+\frac{1}{\vartheta}}, \text{ for every } u,v\in {\cal{B}}(Q_T).%\\
%\leq c(\nu_T)e^{\|\alpha\|_\infty T(1+\frac{1}{\vartheta})} T^{\frac{(\vartheta+1)^2}{4\vartheta(3-\vartheta)}}(T^{1-\frac{1}{\vartheta}}+\|u_k\|_{L^{\vartheta+1}(Q_T)}^{\frac{\vartheta^2-1}{\vartheta}}+\|u\|_{L^{\vartheta+1}(Q_T)}^{\frac{\vartheta^2-1}{\vartheta}})\|u_k^0-u_0\|_{L^2(-1,1)}^{1+\frac{1}{\vartheta}}.
\end{multline*}
By the last inequalities we obtain the estimate (\ref{llc}).%}
\end{proof}
%%%%%%%%%%%%%%%%%%%%%%%%%%%%%%%%%%%%%%%%%%%%%%%%%%%%%%%%%%%%%%%%%%%%%%%%%%%%%%%%
%\colorbox{yellow}{
We assume, for the following of this section,  %}
that assumptions $(A.2), (A.4)$ are enforced, moreover we assume that assumption $(A.3)$ is enforced with $\vartheta\in[1,3)$ instead of $\vartheta\in(1,3).$\\
For the sequel,
the next definitions are necessary.
%%%%%%%%%%%%%%%%%%%%%%%%%%%%%%%%%%%%%%%%%%%%%%%%%%%%%%%%%%%%%%%%%%%%%%%%%%%%%%%%%
\begin{defn}
If %$f\in L^2(Q_T)$ and 
$u_0\in H^1_a(-1,1),$ u is a \textit{strict solution} of problem \eqref{Psemilineare}, %in $\cal{H}(Q_T)$,
if $u\in\cal{H}(Q_T)$ and
\begin{equation*}
\label{}
\left\{\begin{array}{l}
\displaystyle{u_t-(a(x) u_x)_x =\alpha(t,x)u+ \phi(u)\,\quad \mbox{ a.e. \, in } \;Q_T:=\,(0,T)\times(-1,1) }\\ [2.5ex]
\displaystyle{a(x)u_x(t,x)|_{x=\pm 1} = 0\,\,\qquad\qquad\qquad\qquad\;\;\;\,\,\, 
%\textcolor{red}{
a.e. %}
 \;\;\; t\in(0,T)}\;\;\\ [2.5ex]
\displaystyle{u(0,x)=u_0 (x) \,\qquad\qquad\qquad\qquad\quad\qquad\qquad\quad\; \,x\in(-1,1)}~.
\end{array}\right.(\footnote{
%\textcolor{red}{
Since $u\in{\cal{H}}(Q_T)\subseteq L^2(0,T;H^2_a(-1,1)),$ we have $u(t,\cdot)\in H^2_a(-1,1),$ for $ \text{ a.e. } t\in(0,T)$. Keeping in mind Proposition \ref{caratH2}\;\; $(D(A)=H^2_a(-1,1)),$ we deduce the weighted Neumann boundary condition $\displaystyle\lim_{x\rightarrow\pm1}a(x)u_x(t,x)=0,$ for $ \text{ a.e. } t\in(0,T)$.
%}
})
%\left\{\begin{array}{l}
%\displaystyle{u_t=A u(t) +f(t,u(t)),\,\quad \mbox{ a.e. \, in } \;Q_T}\\ [2.5ex]
%\displaystyle{u(0)=u_0% \,\qquad\qquad\qquad\qquad\quad\qquad\qquad\;\; \,x\in(-1,1)
%}~.
%\end{array}\right.
\end{equation*}
\end{defn}
%\textcolor{red}{
In the Ph.D. Thesis %} 
\cite{CF3} we prove, in more general assumption on $f$ of $(A.3)$ (see, in Appendix B, the assumption $(H.3)$), the following result.
\begin{thm}\label{exB}
For all $u_0\in H^1_a(-1,1)$ there exists a unique strict solution $u\in{\cal{H}(Q_T)}$ to \eqref{Psemilineare}.
\end{thm}
%\colorbox{yellow}{
\noindent The lemmas and the complete %} 
proofs of the results that allow us to get the previous theorem can be found in Appendix B.\\
%}
%\textcolor{blue}{
%\begin{defn}%\label{strong}
%If %$f\in L^{1+\frac{1}{\vartheta}}(Q_T)$ and 
%$u_0\in L^2(-1,1),$ u is a \textit{strong solution} to problem \eqref{Psemilineare} in $\cal{B}(Q_T)$, if there exists a approximating sequence $\{u_k\}$ in $\cal{H}(Q_T)$ such that, as $k\rightarrow\infty,$
%$$u_k\longrightarrow u \mbox{  in }  \cal{B}(Q_T),\;\;\, u_{kt}-(a(x) u_{kx})_x -\alpha(t,x)u_k \longrightarrow \phi(u)
%%u^\prime_k-Au_k\rightarrow f(t,u(t))
% \mbox{  in }  L^{1+\frac{1}{\vartheta}}(Q_T),$$
% $$a(x)u_{kx}(t,x)|_{x=\pm1}=0,\;\; t\in (0,T),\quad\forall k\geq1, $$
%and $$u_k(0,\cdot)\rightarrow u_0 \mbox{  in }  L^{2}(-1,1).$$
%\end{defn}
%} %PRG
%\colorbox{yellow}{We assume hereafter that assumptions} $(A.2)-(A.4)$ are enforced.\\
%\textcolor{red}{

The following notion of \lq\lq {\it strong solutions}'' is classical in PDEs theory, see, for instance, \cite{BDDM1}, pp. 62-64.
\begin{defn}\label{strong}
Let %$f\in L^{1+\frac{1}{\vartheta}}(Q_T)$ and 
$u_0\in L^2(-1,1).$ We say that $u\in\cal{B}(Q_T)$ is a \textit{strong solution} to problem \eqref{Psemilineare}, if $u(0,\cdot)=u_0$ and %$u\in\cal{B}(Q_T)$ and 
there exists a %approximating
 sequence $\{u_k\}_{k\in\N}$ in $\cal{H}(Q_T)$ such that, as $k\rightarrow\infty,$ $u_k\longrightarrow u \mbox{  in }  \cal{B}(Q_T)$ and, for every $k\in\N$, $u_k$ is the strict solution of %strict solution of 
the Cauchy problem  
$$
\left\{\begin{array}{l}
\displaystyle{u_{kt}-(a(x) u_{kx})_x =\alpha(t,x)u_k+\phi(u_k)\,\quad \mbox{ a.e. \, in } \;Q_T:=\,(0,T)\times(-1,1) }\\ [2.5ex]
\displaystyle{a(x)u_{kx}(t,x)|_{x=\pm1}=0\,\qquad\qquad\qquad\qquad\qquad\;\;\,\,\,\mbox{ a.e. in }\;\; %t\in
(0,T)}~,%\\ [2.5ex]
%\displaystyle{u_k(0,x)=u^k_0 (x) \,\qquad\qquad\qquad\qquad\quad\qquad\qquad\;\; \,x\in(-1,1)}~,
\end{array}\right.
$$
with initial datum $u_k(0,x).$
%and, as $k\rightarrow\infty,$
%$$u_k\longrightarrow u \mbox{  in }  \cal{B}(Q_T)\;\;\qquad\text{ and }\;\;\qquad u_k(0,\cdot)\longrightarrow u_0 \mbox{  in }  L^{2}(-1,1).$$
\end{defn}
\begin{rem}
We note that, thanks to the definition of the $\cal{B}(Q_T)-$norm %$\Big(\cal{B}(Q_T),\|\cdot\|_{\cal{B}(Q_T)}\Big)$
(see Section 3.1), by the fact that, as $k\rightarrow\infty,$ $u_k\longrightarrow u \mbox{  in }  \cal{B}(Q_T),$ from the Definition \ref{strong} we deduce that %, as $k\rightarrow\infty, 
$\,u_k(0,\cdot)\longrightarrow u_0 \mbox{  in }  L^{2}(-1,1).$\\
Moreover, since $\phi$ is locally Lipschitz continuous (see Theorem \ref{loclip}), %as $k\rightarrow\infty,$
 $$\phi(u_k)\longrightarrow\phi(u), \qquad \text{ in } L^{1+\frac{1}{\vartheta}}(-1,1)
 .\,%\footnote{This limit condition}
$$ 
\end{rem}
%} %PRG
\begin{prop}\label{uni}
Let $T>0,  u_0, v_0\in L^2(-1,1).$ $u,v$ are strong solutions of system \eqref{Psemilineare}, with initial date $u_0, v_0$ respectively.
Then, we have
%$$
\begin{equation}\label{inedc}
\|u-v\|_{\cal{B}(Q_T)}\leq \nu_Te^{\|\alpha^+\|_{\infty}T}\,\|u_0-v_0\|_{L^2(-1,1)},%$$
\end{equation}
where $\alpha^+$%(x)=\max\{\alpha(x),0\},\;x\in (-1,1)$ is
denotes the positive part of $\alpha$
(\footnote{%\textcolor{red}{
$\alpha^+(t,x):=
\max\{\alpha(t,x),0\},\;\forall (t,x)\in Q_T,$ see also \ref{parti}.%}
}) and $\nu_{T}:=e^{\nu T}%{\int_0^T\nu(t)\,dt }
.$
% $\nu_T=e^{%\int_0^T \nu(t)\,dt
%\nu\,T}.$
\end{prop}
\begin{proof}
Let us consider two strong solutions, $u,v\in \cal{B}(Q_T),$ of the problem \eqref{Psemilineare}. Then, there exist $\{u_k\}_{k\in\N},$\\
$\{v_k\}_{k\in\N}\subseteq\cal{H}(Q_T),$ sequences of strict solutions, such that,
as $k\rightarrow\infty,$
%\textcolor{red}{
$$u_k\longrightarrow u,\qquad %\text{ and }
 \qquad v_k\longrightarrow v \qquad \mbox{  in  }\quad  \cal{B}(Q_T),$$
%\textcolor{blue}{ $$\;\;\, u_{kt}-(a(x) u_{kx})_x -\alpha(t,x)u_k\,=\phi(u_k),\qquad\rightarrow \phi(u)
%%u^\prime_k-Au_k\rightarrow f(t,u(t))
% \mbox{  in }  L^{1+\frac{1}{\vartheta}}(Q_T),$$}
 and, for every $k\in\N$,
 $$u_{kt}-(a(x) u_{kx})_x -\alpha(t,x)u_k\,=\phi(u_k), \quad %\text{ and }
  \quad v_{kt}-(a(x) v_{kx})_x -\alpha(t,x)v_k\,=\phi(v_k).$$
 %}
 %\textcolor{blue}{ 
%$$u_k\rightarrow u \mbox{  in }  \cal{B}(Q_T),\;\;\, u_{kt}-(a(x) u_{kx})_x -\alpha(t,x)u_k\,\rightarrow \phi(u)
%%u^\prime_k-Au_k\rightarrow f(t,u(t))
% \mbox{  in }  L^{1+\frac{1}{\vartheta}}(Q_T),$$
% $$v_k\rightarrow v \mbox{  in }  \cal{B}(Q_T),\;\;\, v_{kt}-(a(x) v_{kx})_x -\alpha(t,x)v_k\,\rightarrow \phi(v)
%%u^\prime_k-Au_k\rightarrow f(t,u(t))
% \mbox{  in }  L^{1+\frac{1}{\vartheta}}(Q_T),$$
%and $$u_k(0,\cdot)\rightarrow u_0,\,\,v_k(0,\cdot)\rightarrow v_0  \mbox{  in }  L^{2}(-1,1).$$}
So, for every $k\in \N,$  by definition of $u_k, v_k$ strict solutions, we obtain 
$$
(u_k-v_k)_{t}-\big(a(u_{k}-v_k)_x\big)_x=\alpha (u_k-v_k)+ \phi(u_k)-\phi(v_k),$$ and
multiplying by $u_k-v_k$ both members of the previous equation %present in (\ref{Psemilineare}) 
and integrating %the following
%inequality
on $(-1,1)$ and applying Lemma \ref{f in L2} and condition (\ref{Remf}) (see Remark \ref{rem1}) we obtain
\begin{multline*}
%\label{}
\frac{1}{2} \frac{d}{dt} \int^1_{-1}
(u_k-v_k)^2\,dx+\int^1_{-1}a(x)(u_{k}-v_k)_x^2\,dx\\
=\int^1_{-1}\alpha(t,x) (u_k-v_k)^2
dx+\int^1_{-1}\big(f(t,x,u_k)-f(t,x,v_k)\big)(u_k-v_k)\,dx\\
\leq \int^1_{-1}\alpha^+(t,x) (u_k-v_k)^2+\nu \int^1_{-1}\, (u_k-v_k)^2 dx .
\end{multline*}
Integrating on $(0,t),$ we have
\begin{multline*}%\label{LA1integr}
\frac{1}{2}\|u_k(t,\cdot)-v_k(t,\cdot)\|^2_{L^2(-1,1)}+\int_0^t\int^1_{-1}a(x)(u_{k}-v_k)_x^2(s,x)dx\,ds
\\
\leq\frac{1}{2}
\|u_{k}(0,\cdot)-v_{k}(0,\cdot)\|^2_{L^2(-1,1)}
+\|\alpha^+\|_\infty
\int_0^t\|u_k(s,\cdot)-v_k(s,\cdot)\|^2_{L^2(-1,1)}\,ds+\nu\, \int_0^t \,%\nu(s)
\|u_k(s,\cdot)-v_k(s,\cdot)\|^2_{L^2(-1,1)}\,ds\,.
\end{multline*}
Then we obtain
\begin{multline*}%\label{LA1preGronw}
\|u_k(t,\cdot)-v_k(t,\cdot)\|^2_{L^2(-1,1)}+2\int_0^t\int^1_{-1}a(x)(u_k-v_k)_{x}^2(s,x)dx\,ds
\\\leq
\|u_{k}(0,\cdot)-v_k(0,\cdot)\|^2_{L^2(-1,1)}+% c(\alpha,\gamma_0,p,r) T^{\frac{(2-p)(r+1)}{4p}}\|u_0\|^{r+1}_{L^2(\Omega)}+
\int_0^t\,2\left(\|\alpha^+\|_\infty+\nu\right)\,\|u_k(s,\cdot)-v_k(s,\cdot)\|^2_{L^2(-1,1)}\,ds
\leq %c(\alpha,\gamma_0,p,r) T^{\frac{(2-p)(r+1)}{4p}}\left(
\|u_{k}(0,\cdot)-v_{k}(0,\cdot)\|^2_{L^2(-1,1)} %\|u_0\|^{r+1}_{L^2(\Omega)}\right)+\\
\\+\int_0^t 2\left(\|\alpha^+\|_\infty+\nu%(s)
\right)
\left(\|u_k(s,\cdot)-v_k(s,\cdot)\|^2_{L^2(-1,1)}+2\int_0^s\int^1_{-1}a(x)(u_k-v_k)_{x}^2(\tau,x)dx\,d\tau\right)ds,\;\forall t\in[0,T].
\end{multline*}
%Applicando il  (di Gromwall) con
Applying %\lemref{GW}
 Gronwall's lemma %(integral form) %with $\phi(t)\equiv 2 (c+\|\alpha\|_\infty),\;\psi(t)\equiv 0 $ and $\eta(t)=\|u(t,\cdot)\|^2_{L^2(\Omega)}$,
%si ha
we have
$$%\begin{multline*}%\label{LA1preGronwbis}
\|u_k(t,\cdot)-v_k(t,\cdot)\|^2_{L^2(-1,1)}+2\!\!\int_0^t\int^1_{-1}a(x)(u_k-v_k)_{x}^2(s,x)dx\,ds
\leq %c(\alpha,\gamma_0,p,r) T^{\frac{(2-p)(r+1)}{4p}}
e^{2\|\alpha^+\|_\infty t+2\nu\,t%\int_0^t \nu(s)\,ds
}%\left(
\|u(0,\cdot)-v(0,\cdot)\|^2_{L^2(-1,1)}.%+\|u_0\|^{r+1}_{L^2(\Omega)}\right)\leq
%\leq c(\alpha,\gamma_0,p,r) T^{\frac{(2-p)(r+1)}{4p}}e^{2\|\alpha\|_\infty %T}\left(\|u_0\|^2_{L^2(\Omega)}+\|u_0\|^{r+1}_{L^2(\Omega)}\right)
%%=e^{2(c+\|\alpha\|_\infty)t}\|u_0\|^2_{L^2(\Omega)}
$$%\end{multline*}
%Quindi
Therefore
$$\|u_k-v_k\|^2_{\cal{B}(Q_T)}\leq %c(\alpha,\gamma_0,p,r) T^{\frac{(2-p)(r+1)}{4p}}
%e^{2\Phi}
\nu_T^2\,e^{2\|\alpha^+\|_\infty T}%\left(
\|u_{k}(0,\cdot)-v_k(0,\cdot)\|^2_{L^2(-1,1)}.
$$
Passing to the limit, as $k\rightarrow\infty,$ we obtain
$$\|u-v\|^2_{\cal{B}(Q_T)}\leq 
\nu_T^2\,e^{2\|\alpha^+\|_\infty T}
\|u_{0}-v_0\|^2_{L^2(-1,1)}.$$
\end{proof}
By the previous lemma, applying the inequality \eqref{Superlinearit} (see assumptions (A.3)), we obtain the following Corollary \ref{L2}.
\begin{cor}\label{L2}
%Sia
%\colorbox{red}{
Let $T>0.$ %} %$u_0\in L^2(-1,1)$ and $\alpha\in L^\infty(Q_T).$
%la soluzione di
A strong solution $u\in {\cal{B}}(Q_T)$ of system \eqref{Psemilineare}
satisfies the following a priori estimate
\begin{equation*}%\label{LA}
%\|u\|_{L^2(Q_T)}\,,
\|u\|_{{\cal{B}}(Q_T)}\leq \nu_T\,e^{\|\alpha^+\|_\infty T}\,\|u_0\|_{L^2(-1,1)}\,,
\end{equation*}
where $\alpha^+$%(x)=\max\{\alpha(x),0\},\;x\in (-1,1)$ is
denotes the positive part of $\alpha\,$ ($^{16}$) and $\nu_{T}:=e^{\nu T}%{\int_0^T\nu(t)\,dt }
.$ % \qquad (\footnote{ We recall that  $\nu_{T}=e^{\int_0^T\nu(t)\,dt }.$ })
%e^{\int_0^T \nu(t)\,dt}.$ })
\end{cor}
%\textcolor{red}{
\begin{rem}\label{dcstrict}
We note that Proposition \ref{uni} and Corollary \ref{L2} hold for strict solutions, independently of the notion of strong solution. Indeed, we proved the inequality \eqref{inedc}, first, for strict solutions, then for strong solutions by approximation.
\end{rem}
In this paper, we obtain the result of existence and uniqueness of solutions to \eqref{Psemilineare} with initial state in $L^2(-1,1).$ %in ${\cal{B}}(Q_T).$
\begin{thm}
For all $u_0\in L^2(-1,1)$ there exists a unique strong solution $u\in{\cal{B}}(Q_{T})$ to \eqref{Psemilineare}.
\end{thm}
\begin{proof}
Let $u_0\in L^2(-1,1).$
There exists $\{u^0_k\}_{k\in\N}\subseteq H^1_a(-1,1)$ such that, as $k\rightarrow\infty,$ $u_k^0\rightarrow u_0$ in $L^2(-1,1).$
For every $k\in\N,$ we consider the following problem
\!\!\!\!\!\!\!\!\!\!\!\!\!\!\!\!\!\!\!\!\!\!\!\!\!\!\!\!\!\!
\begin{equation}\label{Pk}
\left\{\begin{array}{l}
\displaystyle{u_{kt}-(a(x) u_{kx})_x =\alpha(t,x)u_k+ %\phi
f(t,x,u_k)\mbox{ \:a.e. \,in}\,Q_T:=\,(0,T)\times(-1,1) }\\ [2.5ex]
\displaystyle{a(x)u_{kx}(t,x)|_{x=\pm 1} = 0\,\,\qquad\qquad\qquad\qquad\:\quad\quad\;\;%\textcolor{red}{
 \text{ a.e. } \!%}
\;t\in(0,T) }\\ [2.5ex]
\displaystyle{u_k(0,x)=u^0_k (x) \,\qquad\qquad\quad\qquad\qquad\qquad\qquad\qquad\;\; \,x\in(-1,1)}~.
\end{array}\right.
\end{equation}
For every $k\in\N,$ by the uniqueness and existence of the strict solution to system \eqref{Pk} (see Theorem \ref{exB}), exists a unique $u_k\in{\cal{H}(Q_T)}$ strict solution to \eqref{Pk}.
Then, we consider the sequence $\{u_k\}_{k\in\N}\subseteq{\cal{H}(Q_T)}$ and by direct application of the Proposition \ref{uni} 
%\textcolor{red}{
(see Remark \ref{dcstrict}) %}
 we prove that $\{u_k\}_{k\in\N}$ is a \textit{Cauchy} sequence in the Banach space $\cal{B}(Q_T)$. Then, there exists $u\in{\cal{B}(Q_T)}$ such that, as $k\rightarrow\infty,$ $u_k\rightarrow u$ in ${\cal{B}(Q_T)}$
%\textcolor{red}{ 
and $\displaystyle u(0,\cdot)\stackrel{L^2}{=}\lim_{k\rightarrow\infty}u_k(0,\cdot)\stackrel{L^2}{=}u_0.$
So, $u\in {\cal{B}(Q_T)}$ is a strong solution.\\
 The uniqueness of the strong solution to \eqref{Psemilineare} is trivial, applying Proposition \ref{uni}.
 \end{proof}
%%%%%%%%%%%%%%%%%%%%%%%%%%%%%%%%%%

%%%%%%%%%%%%%%%%%%%%%%%%%%%%%%%%%%

%\def\baselinestretch{1}

\section{Controllability of nonlinear problems}

%\def\baselinestretch{1.66}

%%% ----------------------------------------------------------------------
In this section we study the global non-negative approximate multiplicative controllability for %linear weakly
semilinear degenerate parabolic Cauchy-Neumann problems. \\%$(\ref{Psemilineare}).$\\
Given $T>0$, let us consider the control system $(\ref{Psemilineare})$ %following
(%\textit{Cauchy-Neumann}
strongly degenerate boundary %semilinear 
problem in divergence form,
governed in the bounded domain $(-1,1)$ by means of the \textit{bilinear control} $\alpha (t,x)$) %\in L^\infty (Q_T)$
%\vspace{1.5cm}
\begin{equation*}
%\label{P2}
%{\displaystyle{
%	\cases{\displaystyle\  v_t-(a(x) v_x)_x =\alpha (t,x)v  ~~~~~ \\\\  &$(t,x)\in Q_T
%%\,=\,(0,T)\times(-1,1),
% $\cr \\
%\displaystyle\ a(x)v_x(t,x) %\stackrel{x \rightarrow\pm 1}{\longrightarrow}
%|_{x=\pm 1} = 0  ~~~~~ \\\\ &$t\in (0,T) $\cr \\
%\displaystyle\ v(0,x)=v_0 (x)  ~~~~~ \\\\  &$x\in (-1,1), $\cr
%}}}
%%%%%%%%%%%%%%%%%%%
%%%%%%%%%%Problema
%(\ref{Psemilineare})
\left\{\begin{array}{l}
\displaystyle{u_t-(a(x) u_x)_x =\alpha(t,x)u+ f(t,x,u)\,\quad \mbox{ in } \; Q_T \,:=\,(0,T)\times(-1,1) }\\ [2.5ex]
\displaystyle{a(x)u_x(t,x)|_{x=\pm 1} = 0\,\,\qquad\qquad\qquad\qquad\qquad\;\;\,\,\, t\in(0,T) }\\ [2.5ex]
\displaystyle{u(0,x)=u_0 (x) \,\qquad\qquad\qquad\qquad\quad\qquad\qquad\;\; \,x\in(-1,1)}~,
\end{array}\right.
%%%%%%%Problema%%%%%%%%
%\begin{cases}
%v_t-(a(x) v_x)_x =\alpha (t,x)v,\,\,\qquad \mbox{in} \qquad Q_t \,=\,(0,T)\times(-1,1),
%\\
%a(x)v_x(t,x)
%%\stackrel{x \rightarrow\pm 1}{\longrightarrow}
%|_{x=\pm 1} = 0, \,\,\qquad\qquad\qquad\,\, t\in (0,T),
%\\
%v(0,x)=v_0 (x) \,\qquad\qquad\qquad\qquad\qquad\, x\in (-1,1),
%\end{cases}
\end{equation*}
under the assumptions $(A.1)-%, (A.2), 
(A.4).$ % and the assumption $(A.3)$ %\colorbox{yellow}{
We will show that this system can be steered in $L^2(-1,1)$ from any nonzero, nonnegative initial state $u_0\in L^2(-1,1)$ into any neighborhood of any desirable nonnegative target-state $u_d\in L^2(-1,1),$ %such that $\langle u_0,u_d\rangle_{1,a}>0,$ 
by bilinear controls. % (x-dependent only).
%%In the particular case $a(x)=1-x^2,$ the above system represents the Budyko-Sellers one-dimensional climatology model.
%%In this case we study the controllability properties in large time relating to understand man's possible actions on the environment, in order to intervene the evolution of the temperature.
Moreover, we extend the above result relaxing the %nonnegative
sign constraint on $u_0.$\\

In the following, we will sometimes use $\|\cdot\|,\;\langle\cdot,\cdot\rangle$ %$\Omega$
 instead of
%the bounded open interval (-1,1)
$\|\cdot\|_{L^2(-1,1)}, \langle\cdot,\cdot\rangle_{L^2(-1,1)}$, respectively, and $\|\cdot\|_\infty$ instead of $\|\cdot\|_{L^\infty(Q_T)}.$  \\

%initial state.

%%% ----------------------------------------------------------------------
%\goodbreak
%\section{Introduction}
\subsection{Some useful lemmas}\label{sul}
%\subsection{Problem formulation and main results}
%\subsection{Problem formulation}
 %%%%%%%%%%%%%%%%%%%%%%%%%%%%%%%%%%%%%%%%%%%%%%%%%%%%%%%%%%%%%%%%%%%%%%%%%%%%%%%%%%%
\noindent In Section \ref{sul}, %we suppose that the semilinear system $(\ref{Psemilineare})$ 
we consider the semilinear system $(\ref{Psemilineare})$ and the associated linear system %(\ref{P2slineare})
\begin{equation}
\label{PL}
%%%%%%%%%%%%%%%%%%%
\left\{\begin{array}{l}
\displaystyle{v_t-(a(x) v_x)_x =\alpha(t,x)v\,\quad \mbox{ in } \; Q_T \,=\,(0,T)\times(-1,1) }\\ [2.5ex]
\displaystyle{a(x)v_x(t,x)|_{x=\pm 1} = 0\,\,\qquad\qquad\quad\;\;\,\, t\in(0,T) }\\ [2.5ex]
\displaystyle{v(0,x)=v_0 (x) \,\qquad\qquad\qquad\quad\quad\;\;\, x\in(-1,1)}~,
\end{array}\right.
%%%%%%%%%%%%%%%%
\end{equation}
where %$\alpha(t,x)$ and the diffusion coefficient $a(x)$ satisfy, respectively, the assumption $(A.2)$ and $(A.4),$ and 
$v_0\in L^2(-1,1),$ and
%%%%%%%%%%%%%%%%%%%%%%%%%%%%
%%%%%%%%%%%%%%%%%%%%%%%%%%%%
%In the following, we assume that 
the coefficients $a(x)$ and $\alpha(t,x)$ %of the associated linear system (\ref{PL}) 
are the same as the semilinear system $(\ref{Psemilineare}).$\\
In this Section 4.1, we obtain some useful results for the proofs of the main theorems.
%RICHIAMARE SISTEMA CON IPOTESI (sempre verificate da qui in poi....)
%%%%%%%%%%%%%%%%%%%%%%%%%%%%%%%%%%%%%%%%%%%%%%%%%%%%%%%%%%%%%%%%%%%%%%%%%%%%%%%%%%%%

%------------------------------------------------------------------------------\\
%DA QUI LEMMA STIMA SU W $\alpha$ generico\\
%------------------------------------------------------------------------------\\

%------------------------------------------------------------------------------
%%%%%%%%%%%%%%%%%%%%%%%%%%%%%%%%%%%%%%%%%%%%%%%%%%%%%%%%%%%%%%%%%%%%%%%%%%%%%%%%%%

%DA QUI LEMMA STIMA SU w con dato u_0 in H^1_a

%%%%%%%%%%%%%%%%%%%%%%%%%%%%%%%%%%%%%%%%%%%%%%%%%%%%%%%%%%%%%%%%%%%%%%%%%%%%%%%%%%
\begin{lem}\label{stima su w}
 %Sia
Let $T>0,$ %\frac{1}{2}\leq p<2,$ let
%\xi_a%=\int_0^x\frac{1}{a(s)}\,ds
%\in L^{\frac{1+\vartheta}{3-\vartheta}}(-1,1), \alpha\in L^\infty(Q_T)$ and $u_0\in H^1_a(-1,1).$ %and $\alpha\in L^\infty(Q_T).$
%%Let $T>0 %, p \mbox{ such that }\, \vartheta\leq \frac{3}{2}p<3
%% \mbox{ and }\,%\alpha(x)=\alpha_*(x)-\|\alpha_*(x)\|_\infty-\rho,$ with $
% \alpha\in L^\infty (Q_T).$ %\, \rho>0.$
%la soluzione di
%the solution of \eqref{P1}
%soddisfa la seguente disuguaglianza:
%satisfying the following inequality:
%\begin{equation}\label{LAalfa}
%\|u\|_{B(Q_T)}\leq C(\|\alpha\|_\infty,T)\|u_0\|_{L^2(-1,1)}
%\end{equation}
%\qquad with\qquad
%$C(\|\alpha\|_\infty,T)=\sqrt{1+e^{KT}[1+KT]},$\qquad
%$K=K(\|\alpha\|_\infty)=2(c+\|\alpha\|_\infty) \mbox{and      } c\in\rea^+$
%ï¿½ la costante presente in
%is the constant present in \eqref{1.2}. \\
%La differenza
%\textcolor{red}{
let $u_0\in L^2(-1,1)$ and let %}
 $u\in{\cal{B}}(Q_T)$ be the strong solution of \eqref{Psemilineare} and $v\in{\cal{B}}(Q_T)$ be the weak solution of \eqref{PL} with %the same control $\alpha\in L^\infty(Q_T)$ and 
 initial state $v_0=u_0$.
Then, the difference $u-v$ belongs to ${\cal{B}}(Q_T)$ and
%tra la soluzione
%between the solution $u$ of \eqref{P2}
%%e la soluzione
%and the solution $v$ of \eqref{P2slineare}
%%con
%with $v_0=u_0$,
%soddisfa la seguente stima:
%$w$
 satisfies %the following estimate
\begin{equation*}%\label{LBalfa}
%\|w\|_{{\cal{B}}(Q_T)}=
\|u-v\|_{{\cal{B}}(Q_T)}\leq
C\,T^\rho\, %e^{(2+\vartheta)\|\alpha^+\|_\infty\,T+\vartheta\,\int_0^T\nu(t)\,dt}
e^{K\,T}%{[(2+\vartheta)\|\alpha^+\|_\infty\,+\vartheta\,\sup_{[0,T]}\nu(t)]\,T}\,%\|u_0\|^{2\vartheta}_{L^2(-1,1)}
\,\|u_0\|^{\vartheta}_{L^2(-1,1)},
%\biggl[(c+2\|\alpha\|_\infty)M+(cC^2e^{(c+2\|\alpha\|_\infty)T}+2\|\alpha\|_\infty+cC^2)
%c(\gamma_0,\vartheta,a)T^{\frac{3}{2}}%e^{\left(3+\vartheta\right)\|\alpha^+\|_\infty T}
%\|u\|^{\vartheta}_{{\cal{H}}(Q_T)},%_{L^2(-1,1)}\,,%\biggr]T
%c(T)%(\gamma_0,\vartheta,a)
 %\,T^{\frac{(1+\vartheta)^2}{4(3-\vartheta)}}\,%e^{2\|\alpha^+\|_\infty\,T}
% \|u_0\|^{\vartheta}_{L^2(-1,1)},\;(\footnote{\quad$\vartheta\in[1,3).$})
\end{equation*}
where $C$ is a positive constant, $\rho=
%\textcolor{red}{
\frac{3-\vartheta}{4}
%}
%\textcolor{blue}{\;\;\frac{(1+\vartheta)^2}{4(3-\vartheta)}}
$ 
and  $ K=(2+\vartheta)\|\alpha^+\|_\infty\,+\vartheta\,%\overline{
\nu%}
$
 ($\alpha^+$%(x)=\max\{\alpha(x),0\},\;x\in (-1,1)$ is
denotes the positive part of $\alpha\,$).%(\gamma_0,\vartheta,a)
%$ is a positive constant such that $c(\cdot)\longrightarrow\infty,$ as $T\rightarrow\infty,$ and
%$c(\cdot)\longrightarrow c_0,$ as $T\rightarrow0,$ where $c_0$ is a positive constant.
\end{lem}

\begin{proof}
%\textbf{
%Proviamo adesso la disuguaglianza
%Now let us prove inequality \eqref{LBalfa}.}\\
%Andiamo a valutare in norma la differenza tra la soluzione
\noindent %\textcolor{blue}{Let $u_0\in L^2(-1,1),$ then there exists $u\in{\cal{B}}(Q_T)$ the strong solution to \eqref{Psemilineare}.}
 Let $\{u_k\}_{k\in\N}\subseteq {\cal{H}}(Q_T)$ be a approximating sequence of 
%\textcolor{red}{
the strong solution %}
 $u.$ %strict solutions to \eqref{Psemilineare}. 
For every
%\textcolor{red}{
fixed %}
 $k\in\N,$ let $v_k\in{\cal{H}}(Q_T)$ be the solution to
%\textcolor{red}{
\eqref{PL}
%}
with initial state $u_k(0,x)$ %such that $u_k^0\longrightarrow u_0,$ as $k\rightarrow \infty$. For every $k\in\N,$ let us consider the difference between the strict solution $u_k\in {\cal{H}(Q_T)}$ of \eqref{Psemilineare}
 %e la soluzione
% and the weak solution $v_k\in {\cal{H}(Q_T)}$ of \eqref{PL},
 %con dato inizale
% 
%with the same %coefficient $\alpha$ and 
%initial state $u^0_k$ 
(\footnote{ For existence, uniqueness and regularity of solutions of linear problem (\ref{PL}) see Section 3.3.}). %in the $\mathcal{B}(Q_T)$ norm.
%Posto
%\textcolor{red}{
Setting, for simplicity of notation,%}
$$w(t,x):=u_k(t,x)-v_k(t,x)\;\; \mbox{ in }\; Q_T,$$
%si ha il seguente sistema
%\textcolor{red}{
we have that %}
$w\in{\cal{H}(Q_T)}$ is strict solution of the following system
\begin{equation}\label{LA4w}
\begin{cases}
w_t-(aw_x)_x=\alpha w +f(t,x,u_k)  \;\;\; \mbox{ in }\; Q_T
\\
a(x) w_x(t,x)|_{x=\pm 1}=0
\\
w(0,x)=0\qquad\qquad\qquad  .
\end{cases}
\end{equation}
%Moltiplicando per ambo i membri dell'equazione presente nel sistema
Multiplying by $w$ both members of the equation in \eqref{LA4w}
%si ottiene:
 we obtain
$$
w_t w-(a(x) w_x)_x w = \alpha w^2 +f(t,x,u_k)w
$$
%quindi integrando su  si deduce
and therefore, integrating on $(-1,1),$ we deduce that
\begin{multline*}\label{LA5w}
%\frac{1}{2}\frac{d}{dt}\int^1_{-1}w^2 dx\\\leq
\frac{1}{2}\frac{d}{dt}\int^1_{-1}w^2 dx+\int^1_{-1}a w^2_x dx
=\int^1_{-1}\alpha w^2 dx +\int^1_{-1}f(t,x,u_k)w
dx\\%\|\alpha\|_\infty \int^1_{-1}w^2 dx+
\leq \int^1_{-1}\alpha^+ w^2 dx +\int^1_{-1}|f(t,x,u_k)||w| dx\,
\leq \|\alpha^+\|_\infty \int^1_{-1}w^2 dx +\int^1_{-1}|f(t,x,u_k)||w| dx\,.
\end{multline*}
%Dalla
Fixing $t\in(0,T)$ and integrating on $(0,t)$, we obtain
$$%\begin{multline*}%\label{LA5t}
\|w(t,\cdot)\|^2_{L^2(-1,1)}+2\int_0^t\,ds\,\int^1_{-1}a w^2_x\,dx%=
\leq 2\|\alpha^+\|_\infty \int_0^t\,\|w(s,\cdot)\|^2_{L^2(-1,1)}\,ds\,%\int^1_{-1} w^2 dx
+2\int_0^t\,ds\,\int^1_{-1}|f(s,x,u_k)||w| dx\, .
$$%\end{multline*}
%Let $p\in\left[\frac{1}{2},2\right),$ ($p$ will be fixed appropriately below)
\noindent Since $u_k,\,v_k\in{\cal{H}}(Q_T)$ and therefore $w=u_k-v_k%$ belongs to $
\in{\cal{H}}(Q_T)\subseteq{\cal{B}}(Q_T),$ by %\colorbox{red}{
$(\ref{Superlinearit})$
%}
 and
%applicando in un primo tempo la
%applying
%\textit{
%Disuguaglianza di H\"older
H\"older's inequality, % },
%e dopo quella di
%e tenendo conto della  si ha:
%keeping in mind \eqref{LA}
we have
\begin{equation*}%\label{fv}
\int_0^t\,ds\,\int^1_{-1}|f(s,x,u_k)||w|dx \leq \gamma_0\int_0^t\,ds\,\int^1_{-1}|u_k|^\vartheta|w|dx\leq \gamma_0 \|u_k\|^\vartheta_{L^{\vartheta+1}(Q_t)}\|w\|_{L^{\vartheta+1}(Q_t)}.%\leq %c(p)\|u\|^{2p}_{L^{2p}(Q_t)}+c(p)\|w\|^{\frac{2p}{2p-1}}_{L^{\frac{2p}{2p-1}}(Q_t)}\leq\\
\end{equation*}
%Keeping in mind that $p=\frac{1+\vartheta}{2}< 2$ if and only if $\vartheta<3$ (\footnote{So, this is the motivation of the requirement $\vartheta <3$ in the assumption (\ref{superl}).}),
Thanks to the assumption (A.4) %i.e
 $\xi_a\in L^{\frac{1+\vartheta}{3-\vartheta}}(-1,1)%\subseteq L^{\vartheta}(-1,1)
,$ then we can apply the Corollary \ref{sob2cor}, %, but being $\frac{q_\vartheta}{2-q_\vartheta}=\frac{p}{3p-2\vartheta},$ we fix $p$ such that $\frac{p}{2-p}=\frac{p}{3p-2\vartheta},$ then $p=\frac{1+\vartheta}{2}$. So $\frac{q_\vartheta}{2-q_\vartheta}=\frac{1+\vartheta}{3-\vartheta},$ then it is necessary that $\xi_a\in L^{\frac{1+\vartheta}{3-\vartheta}}(-1,1),$ which is required in the initial hypotheses.
%$$L^{2}(Q_t)\subseteq L^{\frac{2p}{2p-\vartheta}}(Q_t)$$
so, %we have
%\begin{equation}\label{wqtheta}
%\|w\|_{L^{\vartheta+1}(Q_t)}
%%=\|w\|_{L^{2q_\vartheta}(Q_t)}\leq c(a,q_\vartheta)t^{\frac{1}{2q_\vartheta}\left(1-\frac{q_\vartheta}{2}\right)}\,e^{\|\alpha^+\|_\infty\,t}\,\|w\|_{B(Q_t)}\\
%\leq c(a) t^{\frac{1}{4}}\,%e^{\|\alpha^+\|_\infty\,t}\,
%\|w\|_{B(Q_t)}.
%\end{equation}
%Thus, being $\xi_a\in L^{2\vartheta-1}
%%\frac{1+\vartheta}{3-\vartheta}}
%(-1,1),$ %taking into account that
%%by Corollary 4.1 %=L^{\frac{p}{2-p}}(-1,1)$
%by Corollary \ref{sob3}
%we have
% \begin{equation}\label{utheta}
% \|u\|^{\vartheta}_{L^{2\vartheta}(Q_t)}%\\
% %\leq c(a,p,\vartheta)\,t^{\frac{\vartheta}{p}\left(1-\frac{p}{2}\right)}\|u\|^{\vartheta%\frac{1}{2}
% %}_{L^2\left(0,T;H^1_a(-1,1)\right)}\|u\|^{\vartheta%\frac{1}{2}
% %}_{L^\infty\left(0,T;L^2(-1,1)\right)}\,\\
% %\leq c(a,p,\vartheta)\,t^{\frac{\vartheta}{2p}\left(1-\frac{p}{2}\right)}\|u\|^{\vartheta%\frac{1}{2}
% %}_{B(Q_T)}\\
%  \leq %\nu_T
%  c(\vartheta,a)\,t^{\frac{1}{2}}\,%e^{\vartheta\|\alpha^+\|_\infty\,t}\,\|u_0\|^{\vartheta}_{L^{2}(-1,1)}.
% \|u\|^{\vartheta}_{{\cal{H}}(Q_t)}.
% \end{equation}
%Then, by (\ref{wqtheta}) and (\ref{utheta}), 
applying also \textit{Young's} inequality, we obtain
\begin{multline*}%\label{fvsec}
\int_0^t\,ds\,\int^1_{-1}|f(t,x,u_k)||w|dx
\leq
\gamma_0 \|u_k\|^\vartheta_{L^{\vartheta+1}(Q_t)}\|w\|_{L^{\vartheta+1}(Q_t)}\\%\gamma_0\|u\|^\vartheta_{L^{2\vartheta}(Q_t)}\|w\|_{L^{2}(Q_t)}\\%\leq %c(p)\|u\|^{2p}_{L^{2p}(Q_t)}+c(p)\|w\|^{\frac{2p}{2p-1}}_{L^{\frac{2p}{2p-1}}(Q_t)}\leq\\
\leq c\,\,t^{
%\textcolor{red}{
\frac{3-\vartheta}{4}%}
}
%\textcolor{blue}{{^\frac{(1+\vartheta)^2}{4(3-\vartheta)}}}
%\textcolor{blue}{\;\;\frac{(1+\vartheta)^2}{4(3-\vartheta)}}
%{\frac{(1+\vartheta)^2}{4(3-\vartheta)}}
\|u_k\|^{\vartheta}_{{\cal{B}}(Q_t)}\,\|w\|_{{\cal{B}}(Q_t)}
%\leq \gamma_0c(\vartheta,a) \,T^{\frac{3}{4}}\,\|u\|^{\vartheta}_{{\cal{H}}(Q_T)}\,\|w\|_{{\cal{B}}(Q_t)}\\
%=c(\nu_T,\vartheta,a) \,t^{\frac{3-\vartheta}{4}}e^{\left(1+\vartheta\right)\|\alpha^+\|_\infty\,t}\|u_0\|^{\vartheta}_{L^{2}(-1,1)}\,\|w\|_{B(Q_t)}\\
\leq c\,\,t^{
%\textcolor{red}{
\frac{3-\vartheta}{2}%}
}
%\textcolor{blue}{{^\frac{(1+\vartheta)^2}{2(3-\vartheta)}}}
\,\|u_k\|^{2\vartheta}_{{\cal{B}}(Q_t)}\,+
%\textcolor{red}{
\frac{1}{4}%}
%\textcolor{blue}{\;\;\frac{1}{2}}
\|w\|^2_{{\cal{B}}(Q_t)}.
%\leq c(\gamma_0,\vartheta,a) \,T^{\frac{3}{2}}\,\|u\|^{2\vartheta}_{{\cal{H}}(Q_T)}\,+\frac{1}{4}\|w\|^2_{{\cal{B}}(Q_t)}.\\
%c(\nu_T,\vartheta,a) \,t^{\frac{3-\vartheta}{2}}e^{2\left(1+\vartheta\right)\|\alpha^+\|_\infty\,t}
%\|u_0\|^{2\vartheta}_{L^{2}(-1,1)}\,+\frac{1}{4}\|w\|^2_{B(Q_t)}.\\
\end{multline*}
\noindent So, for every $t\in(0,T),$ we obtain
\begin{multline*}%\label{LA5t1}
\|w(t,\cdot)\|^2_{L^2(-1,1)}+2\!\!\int_0^t\,ds\,\int^1_{-1}a w^2_x\,dx%=
\leq 2\|\alpha^+\|_\infty \!\!\int_0^t\,\|w(s,\cdot)\|^2_{L^2(-1,1)}\,ds\,\,+c\,\,t^
{
%\textcolor{red}{
\frac{3-\vartheta}{2}%}
}
%\textcolor{blue}{{^\frac{(1+\vartheta)^2}{2(3-\vartheta)}}}
%{\frac{(1+\vartheta)^2}{2(3-\vartheta)}}
\,\|u_k\|^{2\vartheta}_{{\cal{B}}(Q_t)}\,+\frac{1}{2}\|w\|^2_{{\cal{B}}(Q_t)}\,\\
%c(\nu_T,\vartheta,a) \,t^{\frac{3-\vartheta}{2}}e^{2\left(1+\vartheta\right)\|\alpha^+\|_\infty\,t}
%\|u_0\|^{2\vartheta}_{L^{2}(-1,1)}\\
\leq 2\|\alpha^+\|_\infty \int_0^t\,\|w\|^2_{{\cal{B}}(Q_s)}\,ds\,\,+
c\,\,t^
{
%\textcolor{red}{
\frac{3-\vartheta}{2}%}
}
%\textcolor{blue}{{^\frac{(1+\vartheta)^2}{2(3-\vartheta)}}}
%{\frac{(1+\vartheta)^2}{2(3-\vartheta)}}
\,\|u_k\|^{2\vartheta}_{{\cal{B}}(Q_t)}\,+\frac{1}{2}\|w\|^2_{{\cal{B}}(Q_t)}.
%\frac{1}{2}\|w\|^2_{B(Q_t)}
%+c\,T^{\frac{3}{2}}\,\|u\|^{2\vartheta}_{{\cal{H}}(Q_T)}
%\,.\\
%c(\nu_T,\vartheta,a) \,t^{\frac{3-\vartheta}{2}}e^{2\left(1+\vartheta\right)\|\alpha^+\|_\infty\,t}
%\|u_0\|^{2\vartheta}_{L^{2}(-1,1)}.
\end{multline*}
%-----------------------------------------------------------\\
%-----------------------------------------------------------\\
%CONTI DA ELIMINARE UTILI PER LA TESI\\
%Now, fixing $\bar{t}\in(t,T),$ we have
%\begin{multline}\label{LA5t3}
%\sup_{t\in [0,\bar{t}]}\|w(t,\cdot)\|^2_{L^2(-1,1)}+2\sup_{t\in [0,\bar{t}]}\int_0^t\,ds\,\int^1_{-1}a w^2_x\,dx\\
%\leq 2\|\alpha^+\|_\infty \sup_{t\in [0,\bar{t}]}\int_0^t\,\|w\|^2_{B(Q_s)}\,ds\,\,+\frac{1}{2}\sup_{t\in %[0,\bar{t}]}\|w\|^2_{B(Q_t)}\\
%+c(\nu_T,\vartheta,a) \,\bar{t}^{\frac{3-\vartheta}{2}}e^{2\left(1+\vartheta\right)\|\alpha^+\|_\infty\,\bar{t}}
%\|u_0\|^{2\vartheta}_{L^{2}(-1,1)},
%\end{multline}
%------------------------------------------------------------\\
%FINE CONTI ELIMINATI
%------------------------------------------------------------\\
%\noindent Then, we have
%
%\begin{multline*}%\label{LA5t4}
%\|w\|^2_{B(Q_t)}
%\leq 2\|\alpha^+\|_\infty \int_0^t\,\|w\|^2_{B(Q_s)}\,ds\,\,+\frac{1}{2}\|w\|^2_{B(Q_t)}\\
%+c\,\,T^{\frac{1+\vartheta}{3-\vartheta}}\,\|u_k\|^{2\vartheta}_{{\cal{B}}(Q_t)},\qquad\qquad t\in(0,T).
%\end{multline*}
From which, by standard saturation argument, we deduce
$$%\begin{multline*}%\label{LA5t4}
\frac{1}{2}\|w\|^2_{{\cal{B}}(Q_t)}
\leq 2\|\alpha^+\|_\infty \int_0^t\,\|w\|^2_{{\cal{B}}(Q_s)}\,ds\,
+c\,\,T^
{
%\textcolor{red}{
\frac{3-\vartheta}{2}
%}
}
%\textcolor{blue}{{^\frac{(1+\vartheta)^2}{2(3-\vartheta)}}}
%{\frac{(1+\vartheta)^2}{2(3-\vartheta)}}
\,\|u_k\|^{2\vartheta}_{{\cal{B}}(Q_T)},\qquad\qquad t\in(0,T).
$$%\end{multline*}
%\textcolor{red}{
Keeping in mind that $w=u_k-v_k$ and %}
applying Gronwall's inequality, 
%\textcolor{red}{
for every $k\in\N,$ %}
we have
\begin{equation*}%\label{LA5t5}
\|u_k-v_k\|^2_{{\cal{B}}(Q_t)}\leq c \,T^
{
%\textcolor{red}{
\frac{3-\vartheta}{2}}
%}
%\textcolor{blue}{{^\frac{(1+\vartheta)^2}{2(3-\vartheta)}}}
%{\frac{(1+\vartheta)^2}{2(3-\vartheta)}}
\,e^{4\|\alpha^+\|_\infty\,T}\|u_k\|^{2\vartheta}_{{\cal{B}}(Q_T)},
%\|u_0\|^{2\vartheta}_{L^{2}(-1,1)}\\
%=c(\nu_T,\vartheta,a)\,t^{\frac{3-\vartheta}{2}}e^{2\left(3+\vartheta\right)\|\alpha^+\|_\infty\,t}
%\|u_0\|^{2\vartheta}_{L^{2}(-1,1)}
\qquad\qquad t\in(0,T),
\end{equation*}
%\textcolor{red}{
 where $c$ is a positive constant, independent of $k.$
Passing %}
 to the limit, as $k\rightarrow\infty,$ in the above inequality, and applying Corollary \ref{L2} we obtain
\begin{multline*}
\|u-v\|^2_{{\cal{B}}(Q_T)}
\leq c \,T^
{
%\textcolor{red}{
\frac{3-\vartheta}{2}
%}
}
%\textcolor{blue}{{^\frac{(1+\vartheta)^2}{2(3-\vartheta)}}}
%{\frac{(1+\vartheta)^2}{2(3-\vartheta)}}
\,e^{4\|\alpha^+\|_\infty\,T}\|u\|^{2\vartheta}_{{\cal{B}}(Q_T)}\\
\leq c \,\nu_T^{2\,\vartheta}\,e^{2(2+\vartheta)\|\alpha^+\|_\infty\,T}\,T^
{
%\textcolor{red}{
\frac{3-\vartheta}{2}%}
}
%\textcolor{blue}{{^\frac{(1+\vartheta)^2}{2(3-\vartheta)}}}
%{\frac{(1+\vartheta)^2}{2(3-\vartheta)}}
\,\|u_0\|^{2\vartheta}_{L^2(-1,1)}
%\textcolor{blue}{\leq c\,T^
%%{\textcolor{red}{\frac{3-\vartheta}{2}}}
%{\frac{(1+\vartheta)^2}{2(3-\vartheta)}}
%%{\frac{(1+\vartheta)^2}{2(3-\vartheta)}}
%\, e^{2[(2+\vartheta)\|\alpha^+\|_\infty\,T+\vartheta\,\int_0^T\nu
%%\textcolor{blue}{
%(t)%}
%\,dt]}\,\|u_0\|^{2\vartheta}_{L^2(-1,1)}}
= c\,T^
{
%\textcolor{red}{
\frac{3-\vartheta}{2}
%}
}
%\textcolor{blue}{{^\frac{(1+\vartheta)^2}{2(3-\vartheta)}}}
%{\frac{(1+\vartheta)^2}{2(3-\vartheta)}}
\, e^{2[(2+\vartheta)\|\alpha^+\|_\infty\,+\vartheta\,%\overline{
\nu
%}
]\,T}\,\|u_0\|^{2\vartheta}_{L^2(-1,1)}.
\end{multline*}
\end{proof}
\begin{lem}\label{NN}
%Sia
Let $T>0$,
%, sia
%$\alpha\in L^\infty(Q_T)$,
 let $u_0 \in L^2(-1,1),\,u_0(x)\geq 0 \,\mbox{ a.e. } x \in (-1,1)$ and % $\alpha(x)\geq 0,$
%sia
let $u\in{\cal{B}}(Q_T)$ %= C([0,T],\, L^2(-1,1))\, \cap \,L^2([0,T],H^1_a (-1,1))$
%soluzione del sistema lineare
 be the strong solution to the semilinear system $(\ref{Psemilineare}).$
% \colorbox{red}{
 Then%}\\
$$u(t,x)\geq 0,\,\,\,\,\mbox{ for a.e. } (t,x)\in Q_T\,.%>0
$$
%\begin{equation*}%\label{P2bis}
%$$\left\{\begin{array}{l}
%\displaystyle{u_t-(a(x) u_x)_x =\alpha (t,x)u+f(x,u)\,\,\qquad \mbox{in} \qquad Q_T \,=\,(0,T)\times(-1,1) }\\ %[2.5ex]
%\displaystyle{a(x)u_x(t,x)|_{x=\pm 1} = 0\,\,\qquad\qquad\qquad\,\, t\in (0,T) \qquad\qquad\qquad\qquad %(\ref{P2})
% }\\ [2.5ex]
%\displaystyle{u(0,x)=u_0 (x) \,\qquad\qquad\qquad\qquad\,\,\,\,\, x\in (-1,1)}~.
%\end{array}\right.$$
%%\left\{\begin{array}{l}
%%{ \displaystyle{ u_t-(a(x) u_x)_x - \alpha u=0 }}\\[2.5ex]
%%{ \displaystyle{ a(x)u_x(t,x)|_{x=\pm 1} = 0 }}\\[2.5ex]
%%{ \displaystyle{ u(0,x)=u_0 (x) }}~.
%%\end{array}\right.
%%%
%%\begin{cases}
%%u_t-(a(x) u_x)_x - \alpha u=0
%%\\
%%a(x)u_x(t,x)|_{x=\pm 1} = 0
%%\\
%%u(0,x)=u_0 (x)
%%\end{cases}
%%\end{equation*}
%$$
%\left\{\begin{array}{l}
%\displaystyle{u_t-(a(x) u_x)_x =\alpha(t,x)u+ f(x,u)\,\quad \mbox{ in } \; Q_T \,=\,(0,T)\times(-1,1) }\\ [2.5ex]
%\displaystyle{a(x)u_x(t,x)|_{x=\pm 1} = 0\,\,\qquad\qquad\qquad\qquad\qquad\;\;\,\, t\in(0,T) }\\ [2.5ex]
%\displaystyle{u(0,x)=u_0 (x) \,\qquad\qquad\qquad\qquad\quad\qquad\qquad\;\;\, x\in(-1,1)}~.
%\end{array}\right.
%$$
%Then\\
%$$u(t,x)\geq 0,\,\,\,\,\forall (t,x)\in Q_T\,.%>0
%$$
%%%%%%%%%%%%%%%
%Moreouer, if $u_0 \in L^\infty (-1,1),\,u_0(x)\geq 0 \,\mbox{ a.e. } x \in (-1,1),$ we have also the following %estimate
%$$0\leq u(t,x)\leq e^{\|\alpha\|_\infty
%t}\|u_0\|_{L^\infty (-1,1)},\,\,\,\,t\in [0,T]\,.%>0
%$$
\end{lem}
%%%%%%%%%%%%%%%%%%%%%%%%%%%%%%%%%%%%%%%%%%%%%%%%%%%%%%%%%%%%%%%%%%%%%%%%%%%%%%%%%%%%%%%%%%%%%%%%%%%%%%
\begin{proof}
%%\begin{itemize}
%%\item
%\underline{STEP 1} \textit{
%%Proviamo che
%Let us prove that $u(t,x)\geq 0 \,\, \forall(t,x)\in Q_T$.}\\
%%Sia
Since $u_0\in L^2(-1,1), u_0\geq0 \mbox{ a.e. } x\in(-1,1),$ there exists $\{u^0_k\}_{k\in\N}\subseteq C^\infty([-1,1]), u^0_k\geq0%, \mbox{ a.e. } x\in
\mbox{ on } (-1,1)$ for every $k\in\N,$ such that $u_k^0\longrightarrow u_0$ in $L^2(-1,1),$ as $k\rightarrow\infty.$
% every function $u_d\in %L^2
%H^1_a(-1,1), u_d\geq 0$ can be approximated by a sequence of strictly positive functions of class $C^\infty ([-1,1])$.\\
%
For every $k\in\N,$ we consider $u_k\in{\cal{H}}(Q_T)$ the strict solution to the semilinear system $(\ref{Psemilineare})$  with initial date $u_k^0$. Keeping in mind that $u(0,\cdot)=u_0$ and applying Proposition \ref{uni}, we can observe that $u_k\longrightarrow u \mbox{ in }{\cal{B}}(Q_T),$ as $k\rightarrow\infty.$
%\begin{equation}\label{Pk}
%\!\!\!\!\!\!\!\left\{\begin{array}{l}
%\displaystyle{u_{kt}-(a(x) u_{kx})_x =\alpha(t,x)u_k+ \phi(u_k)\,\quad \mbox{ a.e. \, in } \;Q_T:=\,(0,T)\times(-1,1) }\\ [2.5ex]
%\displaystyle{a(x)u_{kx}(t,x)|_{x=\pm 1} = 0\,\,\qquad\qquad\qquad\qquad\qquad\;\;\,\,\, t\in(0,T) }\\ [2.5ex]
%\displaystyle{u_k(0,x)=u^0_k (x) \,\qquad\qquad\qquad\qquad\quad\qquad\qquad\;\; \,x\in(-1,1)}~.
%\end{array}\right.
%\end{equation}
%consideramo le seguenti funzioni:
%and we consider the positive-part
%$$u^+\,=\,\max (u,0)\,=
%\left\{\begin{array}{l}
%\displaystyle{ u\qquad\qquad \mbox{ in }\{u>0\}} \\ [2.5ex]
%\displaystyle{ 0\qquad\qquad \mbox{ in }\{u\leq 0 \}}~.
%\end{array}\right.$$
%and the negative-part.
%%\,\begin{cases}
%% u\qquad\qquad \mbox{in}\{u>0\} \\
%%0\qquad\qquad \mbox{in}\{u\leq 0\}
%%\end{cases}
%$$u^-\,=\,\max (0,-u)\,=\,
%\left\{\begin{array}{l}
%\displaystyle{-u\qquad\,\,\,\, \mbox{ in }\{u<0\}}\\ [2.5ex]
%\displaystyle{ 0\qquad\qquad \mbox{ in }\{u\geq 0\} }~.
%\end{array}\right.$$
%%\begin{cases}
%% -u\qquad\,\,\,\, \mbox{in}\{u<0\}\\
%%0\qquad\qquad \mbox{in  }\{u\geq 0\}
%%\end{cases}$$
%%si ha, come richiamato nei preliminari
%%Therefore we haue, as recalled in the preliminaries (Section 2)
%%$
%Since $u=u^+-u^-\,,$
%Osservato che
%Having observed that
%Since $u^+,\,u^-\geq 0$,
%per provare la tesi ï¿½ sufficiente provare che
 %to prove the thesis,
%It is sufficient to prove that
\\
First, we prove that $u_k^-(t,x)\equiv 0  \mbox{ in }Q_T\,.$(\footnote{We denote with $u_k^+$, $u_k^-$ the positive and negative part of $u_k$, respectively 
%\textcolor{red}{
(see \ref{parti}).
%}
})\\
%Moltiplicando ambo i membri dell'equazione del problema per
Multiplying both members of the equation
$u_{kt}  -(a(x)u_{kx})_x =\alpha u_k+f(t,x,u_k)$ %in (\ref{P2}) 
by
$u_k^-$ %ed integrando in
and integrating on $(-1,1)$
%si ha:
we obtain
\begin{equation}
\label{4.2}
\int^1_{-1}\left[ u_{kt} u_k^- -(a(x)u_{kx})_x u_k^-\right]dx= \int^1_{-1}\left[\alpha u_k
u_k^-+f(t,x,u_k) u_k^-\right]dx.
\end{equation}
%Tenendo conto della definizione di si ha:
Recalling the definition of $u^+$ and $u^-$ 
%\textcolor{red}{
(see \ref{parti} ),
%} 
we
have
$$
\int^1_{-1}u_{kt} u_k^- dx = \int^1_{-1} (u_k^+ - u_k^-)_t u_k^- dx =
-\int^1_{-1} (u_k^-)_t u_k^- dx = -\frac{1}{2} \frac{d}{dt} \int
(u_k^-)^2 dx\,.
$$
%Inoltre, integrando per parti, e applicando il si ha
Integrating by parts and recalling that $u_k^-(t,\cdot)\in H^1_a (-1,1), \mbox{ for every } t\in (0,T),$ we
obtain
%e vale la seguente uguaglianza
the following equality %\textcolor{red}{
(see \ref{parti} )%}
$$
\int^1_{-1} (a(x)u_{kx})_x u_k^-\,dx =[a(x)u_{kx} u_k^-]^1_{-1} - \int^1_{-1}
a(x)u_{kx}(-u_k)_x\,dx = \int^1_{-1} a(x) u^2_{kx}\,dx\,.
$$
%Si ha anche
We also have
$$
\int^1_{-1}\alpha u_k u_k^- dx = -\int^1_{-1}\alpha(u_k^-)^2 dx.
$$
Moreover, using %\colorbox{red}{
(\ref{Remf}),
%}
we have
\begin{multline*}
\int^1_{-1} f(t,x,u_k) u_k^-\,dx=\int^1_{-1} f(t,x,u_k^+-u_k^-) u_k^-\,dx
=\int^1_{-1} f(t,x,-u_k^-) u_k^-\,dx\\=-\int^1_{-1} f(t,x,-u_k^-) \left(-u_k^-\right)\,dx
\geq -\int^1_{-1}  \nu%(t)
\left(-u_k^-\right)^2\,dx=-\int^1_{-1}  \nu%(t)
\left(u_k^-\right)^2\,dx
\end{multline*}
%Quindi la  diventa:
and therefore (\ref{4.2}) becomes
$$ -\frac{1}{2} \frac{d}{dt}
\int^1_{-1}(u_k^-)^2 dx + \int^1_{-1}\alpha (u_k^-)^2 dx+\int^1_{-1}  \nu%(t)
\,\left(u_k^-\right)^2\,dx
 \geq \int^1_{-1}
a(x) u^2_{kx}\,dx\:\geq 0,
$$
%da cui,
from which
$$
\frac{d}{dt}\int^1_{-1}(u_k^-)^2 dx\leq 2 \int^1_{-1}\left(\alpha(t,x)+\nu%(t)
\right) (u_k^-)^2
dx\leq 2\left(\|\alpha\|_\infty+\nu%(t)
\right)\int^1_{-1}(u_k^-)^2 dx.
$$
%da ciï¿½ applicando il
From the above inequality, applying Gronwall's inequality %( Lemma \ref{GW} with $\eta(t)=\int^1_{-1}(u^-)^2 dx,\,
%\phi(t)=2\|\alpha\|_\infty,\, \psi(t)\equiu 0$)
%si ha:
 we obtain
$$
\int^1_{-1}(u_k^-(t,x))^2 dx\leq
\nu_T^2e^{2\|\alpha\|_\infty t}\,\,\int^1_{-1}(u_k^-(0,x))^2dx,\;\;\,\;%\textcolor{red}{
\forall t\in(0,T).%}%=0\,.
$$
%in quanto essendo
Since
$u_k(0,x)=u^0_k(x)\geq 0\,,$
%segue
we have
$u_k^-(0,x)=0.$
%Quindi abbiamo provato che
Therefore,% we prove that
$$
u_k^-(t,x)=0, \qquad \,\,\,\quad\forall (t,x)\in Q_T.
$$
%Da ciï¿½ segue che
From this, %as we initially 
 for every $k\in \N,$ it follows that
\begin{equation}
\label{succpos}
u_k(t,x)=u_k^+(t,x)\geq 0, \quad\forall (t,x)\in Q_T.
\end{equation}
%\textcolor{blue}{Let $u\in{\cal{B}}(Q_T)$ %$u\in $$\cal{B}$$(0,T)= C([0,T],\, L^2(-1,1))\, \cap \,L^2([0,T],H^1_a (-1,1))$
%be the
%%soluzione del sistema
%strong solution to the system $(\ref{Psemilineare})$ , then\\}
%\textcolor{red}{
Since $u_k\longrightarrow u \mbox{ in }{\cal{B}}(Q_T),$ as $k\rightarrow\infty,$
there exists %}
$\{u_{k_h}\}_{h\in\N}\subseteq\{u_{k}\}_{k\in\N}$ such that, as $h\rightarrow\infty,$
\begin{equation}
\label{estratta}
u_{k_h}(t,x)\longrightarrow u(t,x),\quad \mbox{ a.e. } (t,x)\in Q_T.
\end{equation}
Applying \eqref{succpos} and \eqref{estratta}, we obtain
$$u(t,x)\geq 0, \quad\mbox{ a.e. } (t,x)\in Q_T.$$
\end{proof}
\subsection{Proofs of main results}
%%%%%%%%%%%%%%%%%%%%%%%%%%%%%%%%%%%%%%%%%%%%%%%%%%%%%%%%%%%%%%%%%%%%%%%%%%%%%%%%%%%%%%%%%%%%%%%%%%%%%%
\begin{proof} (of Theorem \ref{T1}).
To prove Theorem \ref{T1}
%ï¿½ sufficiente considerare un qualsiasi insieme di non-negativi
it is sufficient to consider the set of target states %$v_d$
 %%denso nell'isieme di tutti i non negativi elementi di
 %dense in the set of all the nonnegative elements of $L^2(-1,1)$.\\
 %%Per la densitï¿½ di
 %Because of the density of $C^{\infty}(-1,1)$ in $L^2(-1,1)$,
 %%per ogni non negativa funzione
 %for every nonnegative function
%\begin{equation*}\label{vd} \qquad 
$$u_d\in C^\infty ([-1,1]), %H^2(-1,1)\cap %H^1_0(-1,1) \mbox{  such that  } v_d (x)>0 \mbox{  in  } (-1,1) \mbox{  and  } \frac{(a(x)v_d(x))_x}{v_d(x)}\in L^\infty(-1,1),
  \quad\,
u_d>0  \mbox{  on  } [-1,1].$$
% \end{equation*}
Indeed, %regularizing by convolution,
every function $u_d\in L^2(-1,1), u_d\geq 0$ can be approximated by a sequence of strictly positive functions of class $C^\infty ([-1,1])$.\\
Then, let us consider any $u_0%, u_d
\in L^2(-1,1)$ and any %, 
$u_d\in C^\infty ([-1,1])$ such that
$u_0\geq0,\,u_d>0$ and $u_0\neq0.$ % $\langle u_0,u_d\rangle%_{1,a}
%>0.$  % and any target state $u_d$, with $u_d$ satisfying the
%assumptions  $ u_d\in H^2(-1,1)\cap
%L^\infty(-1,1), u_d\geq h,$ with $h>0.$
%\textcolor{red}{In the following, we denote with $u(t,x)$ the strong solution of problem $(\ref{Psemilineare})$.}
\\
\noindent\textbf{STEP. 1}
We denote with
$\{-\mu_k\}_{k\in\N} \mbox{  and  } \{P_k\}_{k\in\N},$
respectively, the eigenvalues and orthonormal eigenfunctions of the
spectral problem $A_0\omega=\mu \omega,$ with $A_0$ defined as in (\ref{D(A0)}) %with $A_0=(a(x)u_x)_x$
%( see Lemma \ref{spectrum} )
 (\footnote{In the case $a(x)=1-x^2$, that is, where the principal part of the operator is that the Budyko-Sellers model, the orthonormal eigenfunctions are reduced to Legendre polynomials, and the eigenvalues are $\mu_k=(k-1)k, k%\in\N
\geq 1$ (see also Remark \ref{Legendre} ).}) (see Lemma \ref{spectrum} ).
%%%%%%%%%%%%%%%%%%%%%%%%%%%%%%%%%%%%%%%%%%%%%%%%%%%%%%%%%%%%%%%%
 Set
$$z(t,x):=\sum^{\infty}_{k=1}\, e^{-\mu_k
t}
%\bigg(\int^1_{-1} u_0(r)P_k(r)dr\bigg)
\langle u_0, P_k \rangle%_{L^2(-1,1)}
 P_k (x).$$
Since $z\in{\cal{B}}(Q_T),$ we can observe that
%we have the following convergence
\begin{equation*}
z(t,x)=\sum^\infty_{k=1}\, (e^{-\mu_k
t}-1)%\bigg(\int^1_{-1} u_0(r)P_k(r)dr\bigg)
\langle u_0, P_k \rangle%_{1,a}
 P_k(x)+u_0(x)\stackrel{L^2}{\longrightarrow} u_0(x),\, \mbox{ as } t\rightarrow
0.
\end{equation*}
Fix any $s\in(0,1)$,
thus
\begin{equation}\label{t*}
\exists\,\, %t^*=
t^*(s)>0\,\mbox{ such that } \|z(t,\cdot)-u_0\|%_{1,a}
\leq\frac{s%\sigma
}{2}, \;\;\forall t\leq t^*(s).
\end{equation}
Moreover, %for every $\alpha_1\in \R$
\begin{equation}\label{bar t(s)}
\exists\,\, %\bar{t}=
\bar{t}(s)>0\,\,\mbox{ such that }\,\,t^{\rho} e^{K\,t}%t^{%\frac{3-\vartheta}{4}3}
 \leq\frac{s^2}{2C\|u_0\|%_{1,a}
 ^{\vartheta}%_{L^2(-1,1)}
},\quad\forall t\leq \bar{t}(s),%\,\forall \alpha_1\in\R,
\end{equation}
where $\rho, C, K$
%(depending not on $T$ and $\alpha$),
% $C=C(%\alpha_1,
% \gamma_0,\theta,\nu,a),$ and %$C_T\geq 1, \forall T\geq 0,$
%%$C_t$ and
%$K=K(\gamma_0,\theta,\nu,a)$
 are the positive constants of Lemma \ref{stima su w}.
\noindent Now, set
$$t_1(s)= \min \{t^*(s), \bar{t}(s), 1\},$$
%in particular, it can be choose as follows
%\begin{equation}\label{def t}
%t_1(s)=s^{\frac{1}{k}} \;\;\;\mbox{ for some } 0<k<1,
%\end{equation}
%we observe that
%with
we can observe that
$t_1(s)\longrightarrow 0, \mbox{ as } s\rightarrow0.$\\
%From (\ref{lambda}) %and \eqref{def t}
% we have the following choice
%$$\alpha (t,x)=\alpha_1(t_1(s),s)%=\alpha_1(s)=\frac{ ln s}{t_1(s)}
%\;\; \forall t\in[0,t_1(s)].$$
%%%%%%%%%%%%%%%%%%%%%%%%%%%%%%%%%%%%%%%%%%%%%%%%%%%%%%%%%%%%%%%%
% and $\sigma>0.$
%We will choose a $t>0,$ and we will select a constant on
%$(0,t)$ bilinear control
%$\alpha_1=\alpha_1(t,s)<0,$ such that
%%($\alpha_1$ will be chosen after)
\noindent We select the following negative constant %on
%$(0,t)$ static
bilinear control
%$\alpha_1=\alpha_1(t,s)<0,$
\begin{equation*}%\label{lambda}
%\alpha_1=
\alpha(t,x)=\alpha_1(s):=\frac{\ln s}{t_1(s)}<0,\;\;\forall t\in[0,t_1(s)], %\,s,t>0
\forall x\in (-1,1),
\end{equation*}
that is, $\alpha_1(s)$ is such that
$e^{\alpha_1(s) t_1(s)}=s.$
%??Let $t_1(s)>0.$
On the interval $\big(0,t_1(s)\big),$ we apply the negative constant control $\alpha(t,x)=\alpha_1(s),$ $\forall x\in (-1,1).$ %(its value will be chosen below).
Now, %let's
we consider the linear problem $(\ref{PL})$ with $\alpha(t,x)\equiv\alpha_1(s), $\,\;$\forall t\in [0,t_1(s)],$ %=\frac{\ln s}{t_1(s)}<0%\in\R %\;\alpha_1<0
 %($\alpha_1$ will be chosen after)
 %\textcolor{red}{
$\forall x\in(-1,1),$
 %}
and initial state $v_0 = u_0.$ For $t=t_1(s),$ the
%\textcolor{red}{
weak
%}
solution $v(t,x)$ of \eqref{PL} $(^{18})$ has the following representation in Fourier series
\begin{equation*}
v(t_1(s),x)= e^{\alpha_1(s) t_1(s)}\sum^\infty_{k=1}\, e^{-\mu_k
t_1(s)}\langle u_0, P_k \rangle%_{1,a}
P_k (x)
=s \,z(t_1(s),x), \,\forall x\in (-1,1).
\end{equation*}

%%%%%%%%%%%%%%%%%%%%%%%%%%%%%%

\noindent Therefore, by (\ref{t*}), we obtain %$v(t,x)=s\,z(t,x), \forall (t,x)\in Q_T,$ %T>0$
% and
%$$v(t,x)=e^{\alpha_1 t}\,z(t,x)=s z(t,x)\longrightarrow s u_0(x),\qquad \mbox{ as } t\rightarrow
%0,$$
%then
\begin{equation}\label{su_0}
\|v(t_1(s),\cdot)-s u_0\|%_{1,a}%_{L^2(-1,1)}=\|s\,z(t,\cdot)-s u_0(\cdot)\|_{L^2(-1,1)}\\
=s\,\|z(t_1(s),\cdot)-u_0\|%_{1,a}%_{L^2(-1,1)}
\leq\,\frac{s^2}{2}. %\qquad
%\qquad\forall t\leq t^*(s).
\end{equation}
%\textcolor{red}{
Let $u$ be the strong solution to $(\ref{Psemilineare})$ with bilinear control $\alpha(t,x)\equiv\alpha_1(s),\:t>0,\, x\in(-1,1),$ and initial state $u_0.$ %}
%%%%%%%%%%%%%%%%%%%%%%%%%%%%%%
%\textcolor{blue}{Moreover,\\}
By %from %\eqref{LBalfa} a priori
%the estimate % of the
Lemma \ref{stima su w}, the choice of $t_1(s)$ and (\ref{bar t(s)}) %(with $\alpha(x)\equiv \alpha_1<0 $)
we have
\begin{equation}\label{w1}
%\|w(t_1(s),\cdot)\|%_{1,a}%_{L^2(-1,1)}
%=
\|u(t_1(s),\cdot)-v(t_1(s),\cdot)\|%_{1,a}%_{L^2(-1,1)}
\leq %\sqrt{t_1(s)}C
C\,(t_1(s))^\rho e^{K\,t_1(s)}\|u_0\|^{\vartheta}%_{1,a}
\leq %\sqrt{t_1(s)}C\|u_0\|^{\vartheta}_{1,a}%{\frac{3-\vartheta}{4}} \|u_0\|^{\vartheta}%_{L^2(-1,1)}
%\leq
\frac{s^2}{2},
%\longrightarrow 0 ,  \qquad \mbox{ as } t\rightarrow 0.
\end{equation}
where $\rho, C, K$ are the positive constants of Lemma \ref{stima su w}.
%Then,
%Hence,
%\begin{equation}\label{w1}
%%\exists\, \bar{t}>0\,:
%\quad
%\|w(t,\cdot)\|%_{L^2(-1,1)}
%=\|u(t,\cdot)-v(t,\cdot)\|%_{L^2(-1,1)}
%\leq\frac{s^2}{2},\quad\forall t\leq \bar{t}.
%\end{equation}
%%%%%%%%%%%%%%%%%%%%%%%%%%%%%%
From (\ref{su_0}) and (\ref{w1}) we obtain%, if $\sigma$
\begin{equation}\label{su_o}
\|u(t_1(s),\cdot)-s u_0\|%_{1,a}
%\\%_{L^2(-1,1)}
\leq\|u(t_1(s),\cdot)-v(t_1(s),\cdot)\|%_{1,a}%_{L^2(-1,1)}
+\|v(t_1(s),\cdot)-s u_0\|%_{1,a}%_{L^2(-1,1)}
\leq\,s^2.
%<\,c(\gamma_0,
%\alpha_1)t_1(s)^{1-\frac{r}{2p}}\|u_0\|^{r}_{L^2(-1,1)}+%s\,
%\frac{s^2}{2}%=\,c(\gamma_0,
%%\lambda)s^{\frac{1}{k}(1-\frac{r}{2p})}\|u_0\|^{r}_{L^2(-1,1)}+%s\,
%%\frac{s^2}{2}
%<\, s,
\end{equation}
%We observe that, if $\frac{1}{k}(1-\frac{r}{2p})-1>0\Leftrightarrow 0<k<\frac{2p-r}{2p},$
Let us define
%\begin{equation*}
$$\delta_s(x):=u(t_1(s),x)-s u_0(x),  \;\;\;\forall x\in(-1,1),$$
%\end{equation*}
%we finally obtain
%\begin{equation*}
%    u(t_1(s),x)=s\, u_0(x) +  \delta_s(x), \qquad\qquad \forall x\in (-1,1)%\qquad\qquad \mbox{ as } s\rightarrow 0.
%\end{equation*}
%where
and we observe that, in view of (\ref{su_o}),
\begin{equation}\label{inf s}
\frac{\left\|\delta_s(\cdot)\right\|%_{1,a}%_{L^2(-1,1)}
}{s}%=c(\gamma_0,\alpha_1)\|u_0\|^{r}_{L^2(-1,1)}\left(\int_0^1 s^{2\left(\frac{1}{k}(1-\frac{r}{2p})-1\right)}ds\right)^\frac{1}{2}%+s\,\frac{\sigma}{2}<\sigma\,
\longrightarrow 0, \qquad \mbox{ as } s\rightarrow 0.
\end{equation}
In this way, we have steered the nonlinear system (\ref{Psemilineare}) from the initial state $u_0$ to the target state $s u_0+\delta_s,$ at time $t_1(s).$
%____________________________________________________________________________________________
%STEP 2
%____________________________________________________________________________________________

\noindent\textbf{STEP. 2} Let us fix $\eta\in(0,\vartheta-1).$ We will steer the system from the initial state %$s u_0+\delta_s$
 $ u(t_1(s),x)=s\,u_0(x)+\delta_s(x),\;x\in(-1,1),$
 to an arbitrarily small neighborhood of the target state
 $$s^{1+\eta}\,u_d,  \;\;  %\eta\in (0,\vartheta-1)\,.
 $$
  %$s^{1+\eta} u_d$
 at some time $t_2(s).$ For this purpose, define
%Let's select the bilinear control
$$\alpha_2(x):=\alpha_*(x)+\beta,\;\;\,\forall x\in(-1,1),$$
with $\alpha_*(x)=-\frac{(a(x)u_{dx}(x))_x}{u_d(x)},\; x\in
(-1,1),$
and
\begin{equation}\label{Beta2}
\beta=\min\Big\{-\|\alpha_*\|_{L^\infty(-1,1)},-\frac{\eta\,K}{\vartheta-1-\eta}\Big\}-1,
\end{equation}
where $K$ is the positive constant of Lemma \ref{stima su w}.
%We consider the initial state at time $t_1(s)$
%\begin{equation}\label{is}
%    u(t_1(s),x)=s\,u_0(x)+\delta_s(x),\;x\in(-1,1),
%\end{equation}
%and, as target state
%\begin{equation}\label{ts}
%s^{1+\eta}\,u_d ,  \;\;  \eta\in (0,\vartheta-1)\,.
%\end{equation}
We denote by
$\{-\lambda_k\}_{k\in\N} %\qquad\qquad\qquad
\mbox{  and  } %\qquad\qquad\qquad
\{\omega_k\}_{k\in\N},$
respectively, the eigenvalues and orthonormal eigenfunctions of the
spectral problem $A\omega=\lambda \omega,$ with $A=A_0+\alpha_*I$ and $D(A)=H^2_a(-1,1)$ ($A_0$ is the operator defined in \eqref{D(A0)}, see also Lemma \ref{spectrum}).
Applying Lemma \ref{Autof}, %(see Appendix B), %ref{Autof},
we
%\textcolor{blue}{deduce} 
%\textcolor{red}{
have
%}
that
\begin{equation}
\label{4autof}
%\,\exists\, k_* \in \N\,:\qquad\qquad
\lambda_1=0\quad \mbox{ and } \quad\omega_{1}(x)=\frac{u_d(x)}{\|u_d\|%_{1,a}%_{L^2(\Omega)}
}>0,\,\,\forall
x\in (-1,1)\,.
\end{equation}
%%%%%%%%%%%%%%%%%%%%%%%%%%%%%%%%%%%%%%%%%%%%%%%%%%%%%%%%%%%%%%%%%%%%%%%%%%%%%%%%%%%%%%%%%%%%%%%%%%%%%%%%%
%%%%%%%%%%%%%%%%%%%%%%%%%%%%%%%%%%%%%%%%%%%%%%%%%%%%%%%%%%%%%%%%%%%%%%%%%%%%%%%%%%%%%%%%%%%%%%%%%%%%%%%%%
Set $$u_k(s):=\langle u(t_1(s),\cdot),\omega_k%(\cdot)
\rangle%_{1,a}
%\int^1_{-1}u(t_1(s),x)\omega_k(x)dx
, \;\;\forall k\in\N.$$
Thus,
$$u_k(s)=s\,z_k(s),\;\mbox{ where } z_k(s):=%\int^1_{-1}%\left(u_0(x)+\frac{\delta_s(x)}{s}\right)
\Big\langle u_0+\frac{\delta_s}{s},\omega_k\Big\rangle%_{1,a}
,\; \forall k\in\N.$$
Then, by (\ref{inf s}) and (\ref{4autof}), we can observe that
\begin{equation}\label{z1a}
z_1(s)\longrightarrow \frac{1}{\|u_d\|}\langle u_0, u_d\rangle%_{1,a}
\,>0\,, \mbox{ as } s\rightarrow 0.
\end{equation}
The %\textcolor{red}{
weak %}
solution of linear problem (\ref{PL}),
%per questo particolare
%for this particular bilinear coefficient
with $\alpha(t,x)=\alpha_*(x)+\beta,\; t>t_1(s),\,
%\textcolor{red}{
x\in(-1,1), %}
$ and initial state $v(t_1(s),\cdot)=s\,u_0(\cdot)+\delta_s(\cdot), $
%ï¿½ data da:
%Now, select $$\beta=-\|\alpha_*\|_{L^\infty(-1,1)}-1,$$
%with some fixed constant $\rho=\rho(s)>0%=-\|\alpha_*\|_{L^\infty(-1,1)}-\frac{1}{t_2(s)-t_1(s)}\ln\bigg(
%\frac{s^{\eta}\,\|u_d\|^2%_{L^2(\Omega)}
%}{\int^1_{-1} \left(u_0+\frac{\delta_s}{s}\right)u_d dx
%%\int^1_{-1} \left(su_0+\delta_s\right) s^{1+\eta}\,u_d dx
%}\bigg)
%$
has the following representation in Fourier series
(\footnote{We observe that adding $\beta\in\R$ to the coefficient $\alpha_*(x)$ there is a shift of the eigenvalues corresponding to $\alpha_*$ from $\{-\lambda_k\}_{k\in\N}$ to $\{-\lambda_k+\beta\}_{k\in\N},$ but the eigenfunctions remain the same for $\alpha_*$ and $\alpha_*+\beta$. })
%\begin{multline*}
$$v(t,x)=\sum^\infty_{k=1}
e^{(-\lambda_k+\beta)(t-t_1(s))}u_k(s)\omega_k(x)\\
=e^{\beta (t-t_1(s))}u_1(s)\omega_1(x)+\sum_{k>1}
e^{(-\lambda_k+\beta)(t-t_1(s))}u_k(s)\omega_k(x)\,.
$$%\end{multline*}
%Poniamo
Let % us suppose that
$$
r_s(t,x)=\sum_{k>1}
e^{(-\lambda_k+\beta)(t-t_1(s))}u_k(s)\omega_k(x)
$$
%dove, tenendo conto che
where %recalling that $\lambda_k%=k(k+1)
%<\lambda_{k+1},$
%%si ha:
%we obtain
$-\lambda_k<-\lambda_1=0,\,\mbox{ for every }\,k\in\N,\,k>1$\,(see Lemma \ref{spectrum}).
%Adesso diamo la seguente stima, tenendo conto della
%Now we give the following estimate,
Owing to (\ref{4autof}),
%\begin{multline*}
$$
\|v(t,\cdot)-s^{1+\eta}\,u_d\|%_{1,a}
\leq\bigg\|e^{\beta
(t-t_1(s))}u_1(s)\omega_1
-\|s^{1+\eta}\,u_d\|%_{1,a}%_{L^2(\Omega)}
\omega_1\bigg\|%_{1,a}%_{L^2(\Omega)}
\!+\|r_s(t,x)\|%_{1,a}%_{L^2(\Omega)}
\!\!%\\
=\left|e^{\beta (t-t_1(s))}u_1(s)-s^{1+\eta}\,\|u_d\|%_{1,a}%_{L^2(\Omega)}
\right|+\|r_s(t,x)\|.%_{1,a}\,.%_{L^2(\Omega)}
$$%\end{multline*}
Since $-\lambda_k<-\lambda_2,$ for every $k\in\N,\,k>2$ (see Lemma \ref{spectrum}),
%ed applicando la disuguaglianza di Cauchy-Schwarz
 applying  Parseval's equality we have %can also give the following estimate
\begin{multline*}%\label{10}
\|r_s(t,x)\|%_{1,a}
^2%_{L^2(\Omega)}
\leq e^{2(-\lambda_2 +\beta)(t-t_1(s))}
\sum_{k>1}|%\int^1_{-1} \left(su_0+\delta_s\right) dr
u_k(s)|^2 \|\omega_k%(x)
\|%_{1,a}
^2%_{L^2(\Omega)}
\\
=e^{2(-\lambda_2 +\beta)(t-t_1(s))}\sum_{k>1}|\langle su_0+\delta_s,\omega_k\rangle%_{1,a}
|^2= e^{2(-\lambda_2
+\beta)(t-t_1(s))}\|su_0+\delta_s\|%_{1,a}
^2%_{L^2(\Omega)}
.
\end{multline*}
%%%%%%%%%%%%%%%%%%%%%%%%%%%%%%%%%%%%%%%%%%%%%%%%%%%%%%%%%%%%%%%%%%%%%%%%%%%%%%%%%%%%%%%%%%%%%%%%%%%%%%%%%%%%%
%\noindent Fixed $\varepsilon>0$,
%Now, we choose $t_2=t_2(s),\;t_2>t_1(s)$ such that
%\begin{equation}
%\label{T}
%e^{-\lambda_2 \left(t_2-t_1\right)}=\varepsilon\frac{\int^1_{-1} \left(su_0+\delta_s\right) s^{1+\eta}\,u_d %dx}{\|su_0+\delta_s\|\|s^{1+\eta}\,u_d\|^2}\,.
%\end{equation}
%exist $s^*\in(0,1)$ such that $su_0+\delta_s\in L^2(-1,1)\backslash\{0\}, \,su_0+\delta_s\geq 0$ for every $s<s^*,$
%\mbox{ %as } s\rightarrow 0,
%\mbox{ and } su_0+\delta_s\not\equiv0$ in $(-1,1)$
%and
\noindent By %$s\int^1_{-1} \left(u_0(x)+\frac{\delta_s(x)}{s}\right)\omega_{1}(x)dx\longrightarrow \frac{1}{\|u_d\|}\int^1_{-1} u_0(x)u_d(x)dx>0, \qquad (\mbox{ as } s\rightarrow 0),$
%by
(\ref{z1a})
%si ha
we obtain
\begin{equation}
\label{6}
%\langle su_0+\delta_s,\omega_{1}\rangle
\exists\; s^*\in (0,1):\, u_1(s)=\langle su_0+\delta_s,\omega_{1}\rangle%_{1,a}
>0,\,
%\int^1_{-1} \left(su_0(x)+\delta_s(x)\right)\omega_{1}(x)dx>0,
 \forall s\in(0,s^*).
\end{equation}
\noindent Then, we choose $t_2(s),\;t_2(s)>t_1(s)$ such that %it is possible choose $\beta=\beta_s=\beta(s,t(s))$ %and $T>0$
%in modo che
%so that
\begin{equation}\label{beta}
e^{\beta \left(t_2(s)-t_1(s)\right)} u_1(s) =s^{1+\eta}\,\|u_d\|%_{1,a}%_{L^2(\Omega)}
\,,
\end{equation}
that is, since $\omega_1=\frac{%s^{1+\eta}\,
u_d}{\|%s^{1+\eta}\,
u_d\|},$
\begin{equation}
\label{9} t_2(s)=t_1(s)+\frac{1}{\beta}\ln\bigg(
\frac{s^{\eta}\,\|u_d\|%_{1,a}
^2%_{L^2(\Omega)}
}{\langle u_0+\frac{\delta_s}{s},u_d\rangle%_{1,a}
%\int^1_{-1} \left(u_0+\frac{\delta_s}{s}\right)u_d dx
%\int^1_{-1} \left(su_0+\delta_s\right) s^{1+\eta}\,u_d dx
}\bigg).
\end{equation}
So, by (\ref{beta}) and the above estimates for $\|v(t_2(s),\cdot)-s^{1+\eta}\,u_d(\cdot)\|%_{1,a}
$ and $\|r_s(t_2(s),\cdot)\|%_{1,a}
$ we conclude that
%So, by $(\ref{v-vd})-(\ref{T})$ %(\ref{9})
%and (\ref{9}) % the last inequality, and the above estimate for $\|v(T,\cdot)-s^{1+\eta}\,u_d(\cdot)\|$
% we conclude that
%%segue
%%it follows
%$$
%\|v(t_2,\cdot)-s^{1+\eta}\,u_d(\cdot)\|\leq e^{(-\lambda_2
%+\beta_\varepsilon)\left(t_2-t_1\right)}\|su_0+\delta_s\|%_{L^2(\Omega)}
%%=
%%$$
%%$$
%=e^{-\lambda_2 \left(t_2-t_1\right)}%\ln\bigg(
%\frac{\|s^{1+\eta}\,u_d\|^2%_{L^2(\Omega)}
%}{%\int^1_{-1} v_0\omega_1 dx
%\int^1_{-1} \left(su_0+\delta_s\right) s^{1+\eta}\,u_d dx}%\bigg)
%\|su_0+\delta_s\|%_{L^2(\Omega)}
%=\varepsilon\,.
%%\stackrel{T\longrightarrow\infty}{\longrightarrow}0\,.
%$$
%%Da cui la tesi, in quanto
%From which %the thesis, being
%we have the conclusion.
%FIN QUI\\
%%%%%%%%%%%%%%%%%%%%%%%%%%%%%%%%%%%%%%%%%%%%%%%%%%%%%%%%%%%%%%%%%%%%%%%%%%%%%%%%%%%%%%%%%%%%%%%%%%%%%%%%
%%%%%%%%%%%%%%%%%%%%%%%%%%%%%%%%%%%%%%%%%%%%%%%%%%%%%%%%%%%%%%%%%%%%%%%%%%%%%%%%%%%%%%%%%%%%%%%%%%%%%%%%
%Moreover, considering the estimates obtained in \cite{CF}, we have the following
\begin{multline}\label{stima lin}
\|v(t_2(s),\cdot)-s^{1+\eta} u_d(\cdot)\|%_{1,a}%_{L^2(-1,1)}
\leq e^{(-\lambda_2
+\beta)(t_2(s)-t_1(s))}\|s u_0+\delta_s\|%_{1,a}%_{L^2(-1,1)}
\\=e^{-\lambda_2 (t_2(s)-t_1(s))}%\ln\bigg(
\frac{s^{1+\eta}\|u_d\|%_{1,a}%_{L^2(-1,1)}
}{%\int^1_{-1} \left(s u_0+\delta_s\right)\omega_1
u_1(s)}%\bigg)
\|s u_0+\delta_s\|%_{1,a}%_{L^2(-1,1)}%\stackrel{T\longrightarrow\infty}{\longrightarrow}0
=e^{-\lambda_2 (t_2(s)-t_1(s))}%\ln\bigg(
\frac{%s^{1+\eta}
\|u_d\|%_{1,a}%_{L^2(-1,1)}
}{%\int^1_{-1} \left(s u_0+\delta_s\right)\omega_1
z_1(s)}%\bigg)
\bigg\|u_0+\frac{\delta_s}{s}\bigg\|\,s^{1+\eta}%_{1,a}
\,.
\end{multline}
%Choosing $\beta$ so that
%$$
%e^{\beta (t_2-t_1)}\int^1_{-1} \left(s u_0+ \delta_s(x)\right)\omega_1 dx=s^{1+\eta}\|u_d\|%_{L^2(-1,1)}
%$$
\noindent Thus, by (\ref{beta}) and by (\ref{6}), we deduce that there exists $\,%\exists 
s_0\in(0,s^*) \mbox{ such that }$ %$\mbox{ as } s\rightarrow 0^+,$ we obtain
\begin{equation*}\label{e mag}
   e^{-\lambda_2 (t_2(s)-t_1(s))}\frac{%\bigg
  \|u_d\| \|u_0+\frac{\delta_s}{s}%\bigg
   \|%_{1,a}
   }{z_1(s)}=\left(\frac{s^{\eta}\|u_d\|%_{1,a}%_{L^2(-1,1)}
    }{ %\int^1_{-1} \left( u_0+\frac{\delta_s(x)}{s}\right)\omega_1dx
z_1(s)}\right)^{\frac{-\lambda_2}{\beta}}\frac{%\bigg
\|u_d\|\|u_0+\frac{\delta_s}{s}%\bigg
\|%_{1,a}
}{z_1(s)}\leq c s^{\frac{-\eta\lambda_2}{\beta}},\;\; \forall s\in(0,s_0).
\end{equation*}
\noindent From the above, %$\mbox{ as } s\rightarrow 0^+,$ %\eqref{e mag}
the inequality \eqref{stima lin} becomes
\begin{equation}\label{stima lin+ e mag}
\|v(t_2(s),\cdot)-s^{1+\eta} u_d(\cdot)\|%_{1,a}%_{L^2(-1,1)}
\leq %C s^{\frac{-\eta\lambda_2}{\beta}}s^{1+\eta}\frac{\|u_d\|%_{L^2(-1,1)}
%}{%\int^1_{-1} \left(s u_0+\delta_s\right)\omega_1
%z_1(s)}%\bigg)
%\bigg\|u_0+\frac{\delta_s}{s}\bigg\|\leq
c s^{\frac{-\eta\lambda_2}{\beta}}s^{1+\eta},\;\;\forall s\in(0,s_0),
%e^{-\lambda_2 (t_2(s)-t_1)}
%\frac{s^{1+\eta}\|u_d\|%_{L^2(-1,1)}
%}{\int^1_{-1} \left(s u_0+\delta_s(x)\right)\omega_1
%dx}
%\|s u_0+\delta_s(x)\|%_{L^2(-1,1)}\leq
%\\\leq
%C s^{\frac{-\eta\lambda_2}{\beta}}\left(s^{1+\eta}\|u_d\|%_{L^2(-1,1)}
%\right)\frac{\| u_0+\frac{\delta_s(x)}{s}\|%_{L^2(-1,1)}
%}{\int^1_{-1} \left( u_0+\frac{\delta_s(x)}{s}\right)\omega_1
%dx}
%\leq C s^{\frac{-\eta\lambda_2}{\beta}}\left(s^{1+\eta}\|u_d\|%_{L^2(-1,1)}
%\right)\\
%\;\mbox{ as } s\rightarrow 0^+.
\end{equation}
where $c$ is a positive constant.\\
Then, by \eqref{Beta2}, we observe that
$$\alpha_2(t,x)=\alpha_*(x)+\beta%=\alpha_*(x)-\|\alpha_*\|-1
<0,\;\forall t\in[t_1(s),t_2(s)],\;\forall x\in (-1,1).$$\\
%\textcolor{red}{
Let $u$ be the strong solution to $(\ref{Psemilineare})$ 
with $\alpha(t,x)=\alpha_*(x)+\beta,\; t>t_1(s),\, x\in(-1,1),$ and initial state $u(t_1(s),\cdot)=s\,u_0(\cdot)+\delta_s(\cdot).$
%}
Thus, by Lemma \ref{stima su w} we deduce the following estimate
\begin{equation}\label{uv}
\|u(t_2(s),\cdot)-v(t_2(s),\cdot)\|%_{1,a}%\leq\|u-v\|_{B(Q_T)}
\leq
%\biggl[(c+2\|\alpha\|_\infty)M+(cC^2e^{(c+2\|\alpha\|_\infty)T}+2\|\alpha\|_\infty+cC^2)
C\,\left(t_2(s)-t_1(s)\right)^\rho e^{K\left(t_2(s)-t_1(s)\right)}{%\frac{3-\vartheta}{4}
}%e^{\left(3+\vartheta\right)\|\alpha^+\|_\infty \left(t_2(s)-t_1\right)}
\|su_0+\delta_s\|%_{1,a}
^{\vartheta},
\end{equation}
where $\rho, C, K$ are the positive constants of Lemma \ref{stima su w}.
Then, by (\ref{9}), we deduce that %$\,\exists s_0\in(0,s^*) \mbox{ such that }$ %$\mbox{ as } s\rightarrow 0^+,$ we obtain
\begin{equation}\label{e mag}
   e^{K (t_2(s)-t_1(s))}=\left(\frac{s^{\eta}\|u_d\|
   %\textcolor{red}{
   %}
   %_{1,a}%_{L^2(-1,1)}
    }{%\int^1_{-1} \left( u_0+\frac{\delta_s(x)}{s}\right)\omega_1dx
z_1(s)}\right)^{\frac{K}{\beta}}%\frac{%\bigg
%\|u_0+\frac{\delta_s}{s}%\bigg
%\|_{1,a}}{z_1(s)}
\leq c^\prime s^{\frac{\eta K}{\beta}},\;\; \forall s\in(0,s_0).
\end{equation}
%\textcolor{red}{Let $u$ be the strong solution to $(\ref{Psemilineare})$ 
%with $\alpha(t,x)=\alpha_*(x)+\beta,\; t>t_1(s),\, x\in(-1,1),$ and initial state $u(t_1(s),\cdot)=s\,u_0(\cdot)+\delta_s(\cdot).$
%}
Then, by $(\ref{stima lin+ e mag})-(\ref{e mag}),$ we have the following estimate
\begin{multline}\label{stima semilin}
   \|u(t_2(s),\cdot)-s^{1+\eta} u_d(\cdot)\|%_{1,a}%_{L^2(-1,1)}
   \leq \|u(t_2(s),\cdot)-v(t_2(s),\cdot)\|%_{1,a}%_{L^2(-1,1)}
   +\|v(t_2(s),\cdot)-s^{1+\eta} u_d(\cdot)\|\\%_{1,a}%_{L^2(-1,1)}
  \leq C%\rho,a
  \left(t_2(s)-t_1(s)\right)^{\rho%\frac{1}{2}
}e^{K(t_2(s)-t_1(s))}\|s u_0 +\delta_s\|%_{1,a}
^{\vartheta}%_{L^2(-1,1)}
 +c s^{\frac{-\eta\lambda_2}{\beta}}s^{1+\eta}%\|u_d\|%_{L^2(-1,1)}
\\\leq C%\rho,a
  \left(t_2(s)-t_1(s)\right)^{\rho%\frac{1}{2}
} c^\prime s^{\frac{\eta K}{\beta}}s^\vartheta\left\| u_0 +\frac{\delta_s}{s}\right\|%_{1,a}
^{\vartheta}%_{L^2(-1,1)}
 +c s^{\frac{-\eta\lambda_2}{\beta}}s^{1+\eta}%\|u_d\|%_{L^2(-1,1)}
 \\
\leq k\left( %\rho,a
  \left(t_2(s)-t_1(s)\right)^{\rho%\frac{1}{2}
} s^{\frac{\eta K}{\beta}}s^{\vartheta-1-\eta}\left\| u_0 +\frac{\delta_s}{s}\right\|%_{1,a}
^{\vartheta}%_{L^2(-1,1)}
 + s^{\frac{-\eta\lambda_2}{\beta}}%\|u_d\|%_{1,a}%_{L^2(-1,1)}
 \right)s^{1+\eta}\\
\leq k\left(%\rho,a
  \left(t_2(s)-t_1(s)\right)^{\rho%\frac{1}{2}
} s^{\frac{\eta K}{\beta}+\vartheta-1-\eta}
%\norm{ u_0 +\frac{\delta_s(x)}{s} }^{r}_{L^2(-1,1)}
+ s^{\frac{-\eta\lambda_2}{\beta}}%\|u_d\|_{L^2(-1,1)}
\right)s^{1+\eta},
\quad%\mbox{ for every }
\forall s\in(0,s_0),
\end{multline}
where $k$ is a positive constant.
Now, we have
\begin{equation*}\label{Tdiv}
t_2(s)-t_1(s)=%
\frac{1}{\beta}\ln\bigg(
\frac{s^{\eta}\,\|u_d\|%_{1,a}
^2%_{L^2(\Omega)}
}{\langle u_0+\frac{\delta_s}{s}, u_d\rangle%_{1,a}
%\int^1_{-1} \left(u_0+\frac{\delta_s}{s}\right)u_d dx
%\int^1_{-1} \left(su_0+\delta_s\right) s^{1+\eta}\,u_d dx
}\bigg)
%\frac{1}{\beta}\ln\left(\frac{s^{1+\eta}\|u_d\|%_{L^2(-1,1)}
%}{\int^1_{-1} \left(s u_0+\delta_s(x)\right)\omega_1
%dx}\right)=
%\frac{1}{\beta}ln\left(\frac{s^{\eta}\|u_d\|%_{L^2(-1,1)}
%}{\int^1_{-1} \left( u_0+\frac{\delta_s(x)}{s}\right)\omega_1
%dx}\right)\\
%=\frac{1}{\beta}ln\left(s^{\eta}\|u_d\|%_{L^2(-1,1)}
%\right)-\frac{1}{\beta}ln\left(\int^1_{-1} \left( u_0+\frac{\delta_s(x)}{s}\right)\omega_1
%dx\right)
\longrightarrow +\infty,\; \mbox{ as } s\rightarrow 0^+.
\end{equation*}
Since $\frac{\eta K}{\beta}+\vartheta-1-\eta>0,$ by the choice of $\beta$ (see (\ref{Beta2})), we have
 \begin{equation*}
 \left(t_2(s)-t_1(s)\right)^{\rho%\frac{1}{2}
 }s^{\frac{\eta K}{\beta}+\vartheta-1-\eta}=\left(\frac{1}{\beta}\ln\left(\frac{s^{\eta}\|u_d\|^2%_{1,a}%_{L^2(-1,1)}
 }{\langle  u_0+\frac{\delta_s}{s},u_d\rangle
 %\int^1_{-1} \left(s u_0+\delta_s(x)\right)\omega_1dx
 }\right)\right)^{\rho%\frac{1}{2}
} s^{\frac{\eta K}{\beta}+\vartheta-1-\eta} \longrightarrow 0,
\end{equation*}
$\mbox{ as } s\rightarrow 0^+.$
Defining
$$\delta_{s^{1+\eta}}(x):=u(t_2(s),\cdot)-s^{1+\eta} u_d(\cdot) \qquad x\in (-1,1),$$
%then
%$$u(t_2,\cdot)=s^{1+\eta}u_d+\delta_{s^{1+\eta}}(x)\qquad\qquad \mbox{ as } s\rightarrow 0^+.$$
%with
estimate (\ref{stima semilin}) yields 
%$$
\begin{equation}\label{deltaeta}
\frac{\|\delta_{s^{1+\eta}}(\cdot)\|%_{1,a}
}{s^{1+\eta}}\rightarrow 0, \mbox{ as } s\rightarrow 0^+.
\end{equation}
%$$
\textbf{STEP. 3} Let $\tau>0%,\,T>t_2(s)
.$ On the interval $(t_2(s),T(s)),$ with $T(s)=t_2(s)+\tau,$ we
 apply a positive constant control $\alpha_3(x)\equiv\alpha_3$ (its value will be chosen
 below).\\
%Set $\tau=T-t_2,$
We can represent the
%\textcolor{red}{
weak
%}
 solution of the linear problem (\ref{PL}), with $\alpha(t,x)\equiv\alpha_3$ and initial state $v(t_2(s),\cdot)=
%\textcolor{red}{
u(t_2(s),\cdot)=%}
s^{1+\eta}u_d+\delta_{s^{1+\eta}},$ by Fourier series in the following way
\begin{equation*}%\label{Fourier rap}
v(t_2(s)+\tau,x)= e^{\alpha_3 \tau}\sum^\infty_{k=1}\, e^{-\mu_k
\tau}%\bigg(\int^1_{-1} u(t_2(s),r)P_k(r)dr\bigg)
\langle u(t_2(s),\cdot),P_k\rangle%_{1,a}
P_k (x).\\
%=e^{\alpha_3 \tau}\sum^\infty_{k=1}\, e^{-\mu_k
%\tau}\bigg(\int^1_{-1} (s^{1+\eta}u_d+\delta_{s^{1+\eta}}(x))P_k(r)dr\bigg)P_k (x)
\end{equation*}
Let us consider
$$z(\tau,x):=\sum^\infty_{k=1}\, e^{-\mu_k
\tau}%\bigg(\int^1_{-1} u(t_2(s),r)P_k(r)dr\bigg)
\langle u(t_2(s),\cdot),P_k\rangle%_{1,a} 
P_k (x),$$
then,
\begin{equation}\label{z(T,x)}
    z(\tau,x)=
    \sum^\infty_{k=1}\, \left(e^{-\mu_k
\tau}-1\right)%\bigg(\int^1_{-1} u(t_2(s),r) %(s^{1+\eta}u_d+\delta_{s^{1+\eta}}(x))
%P_k(r)dr\bigg)
\langle u(t_2(s),\cdot),P_k\rangle%_{1,a}
 P_k (x)+s^{1+\eta}u_d(x)+\delta_{s^{1+\eta}}(x)
\stackrel{L^2}{\longrightarrow} s^{1+\eta}u_d+\delta_{s^{1+\eta}},\, \mbox{ as } \tau\rightarrow 0^+\,.
%\leq \left(e^{-\mu_1
%\tau}-1\right)\sum^\infty_{k=1}\, \bigg(\int^1_{-1} (s^{1+\eta}u_d+\delta_{s^{1+\eta}}(x))P_k(r)dr\bigg)P_k %(x)+s^{1+\eta}u_d+\delta_{s^{1+\eta}}(x) \\
% =h(\tau)(s^{1+\eta}u_d+\delta_{s^{1+\eta}}(x))+s^{1+\eta}u_d+\delta_{s^{1+\eta}}(x).\, %,
\end{equation}
\vspace{0.5cm}
\noindent Now, for every $0<
\varepsilon<1,$ by \eqref{deltaeta}, we have
\begin{itemize}
  \item $\exists\, s_\varepsilon\in(0,s_0)$ such that %$$
   \begin{equation}\label{deltaeta2}
  \frac{\|\delta_{s_\varepsilon^{1+\eta}}\|%_{1,a}
  }{s_\varepsilon^{1+\eta}}\leq%<
  \frac{\varepsilon}{4}.
  \end{equation}
  %$$
  \noindent So, by \eqref{z(T,x)}, %in correspondence to $s_\varepsilon,$
\begin{description}
  \item $\exists %\tau_\varepsilon=
  \,\tau_\varepsilon=\tau(s_\varepsilon)>0$ such that\\
  %\begin{itemize}
    %\item 
    %C\tau_\varepsilon^{\frac{1}{2}}s_\varepsilon^{-K(1+\eta)}e^{K\tau_\varepsilon}\left(\|u_d\|_{1,a}^\vartheta+1\right)<\frac{\varepsilon}{2}
\begin{equation}\label{tausist}
        C\tau_\varepsilon^{\rho}e^{{\nu}\vartheta\tau_\varepsilon}s_\varepsilon^{-2\left(1+\eta\right)}
        %\textcolor{red}{
        \left(\|u_d\|%_{1,a}
        %^\vartheta
        +1\right)^\vartheta
        %}
        \leq%<
        \frac{\varepsilon}{2}\quad
 \mbox{ and }
 \quad%\|z(\tau_\varepsilon,\cdot)-s_\varepsilon^{1+\eta}\,u_d\|%_{1,a}
        %\leq\|\delta_{s_\varepsilon^{1+\eta}}\|%_{1,a}
         %+\frac{\varepsilon}{4}\,s_\varepsilon^{1+\eta},
\|z(\tau_\varepsilon,\cdot)-(s_\varepsilon^{1+\eta}\,u_d+\delta_{s_\varepsilon^{1+\eta}})\|%_{1,a}
        \leq %\|\delta_{s_\varepsilon^{1+\eta}}\|%_{1,a}
         %+
         \frac{\varepsilon}{4}\,s_\varepsilon^{1+\eta},
\end{equation}
where $\rho, C$
 are the positive constants of Lemma \ref{stima su w}.
%  \end{itemize}
\end{description}
\end{itemize}
Set $T_\varepsilon=T(s_\varepsilon)=t_2(s_\varepsilon)+\tau_\varepsilon.$ Let us define
\begin{equation}\label{alfa3}
    \alpha(t,x)=\alpha_3(s_\varepsilon):=%-\frac{\ln(s_\varepsilon^{1+\eta})}{\tau_\varepsilon}=
    -\frac{1+\eta}{\tau_\varepsilon}\ln s_\varepsilon,%>0,
    \quad \forall t \in[t_2(s_\varepsilon),T_\varepsilon],\;\forall x\in (-1,1).
\end{equation}
%\textcolor{red}{
Let $u$ be the strong solution to $(\ref{Psemilineare})$ with bilinear control $\alpha(t,x)\equiv\alpha_3%(s_\varepsilon)
,$ \:$t>t_2(s_\varepsilon),\,x\in(-1,1),$ and initial state $u(t_2(s_\varepsilon),\cdot)=s_\varepsilon^{1+\eta}u_d+\delta_{s_\varepsilon^{1+\eta}}.$
%}
%%%%%%%%%%%%%%%%%%%%%%%%%%%%%%%%%%%%%%%%%%%%%%%%%%%%%%%%%%%%%%%%%%
\noindent By Lemma \ref{stima su w}, taking in mind that in our case the positive constant $K$ of Lemma \ref{stima su w} is $K=(2+\vartheta)%\|
\alpha_3(s_\varepsilon)%\|_\infty
+\vartheta\nu,$ and by \eqref{deltaeta2} and \eqref{alfa3}, since $\varepsilon<1,$
   we obtain 
\begin{multline*}\label{w stima}
%\|w(T,x)\|_{L^2(-1,1)}=
\|u(t_2(s_\varepsilon)+\tau_\varepsilon,\cdot)-v(t_2(s_\varepsilon)+\tau_\varepsilon,\cdot)\|%_{1,a}%_{L^2(-1,1)}
%\leq\|u-v\|_{B(Q_T)}\\
%\leq c(\gamma_0,\vartheta)\tau_\varepsilon^{\frac{3-\vartheta}{4}}e^{(3+\vartheta)\alpha_3\tau_\varepsilon}%\cdot\\ \cdot
%%\left(
%\|s^{1+\eta}u_d\|^{\vartheta}%_{L^2(-1,1)}%+\|s^{1+\eta}u_d\|^{\frac{r(r+1)}{2}}_{L^2(-1,1)}\right)
%\\
\leq C\tau_\varepsilon^{\rho%\frac{1}{2}
}e^{K%(1+\alpha_3(%\tau_\varepsilon,
%s_\varepsilon))
\tau_\varepsilon}%\cdot\\ \cdot
s_\varepsilon^{\left(1+\eta\right)\vartheta}\Big\|u_d+\frac{\delta_{s_\varepsilon^{1+\eta}}}{s_\varepsilon^{1+\eta}}\Big\|%_{1,a}
^\vartheta
\\
=C\tau_\varepsilon^{\rho}e^{\nu\vartheta\,\tau_\varepsilon}e^{(2+\vartheta)\alpha_3(%\tau_\varepsilon,
s_\varepsilon)\tau_\varepsilon}
s_\varepsilon^{\left(1+\eta\right)\vartheta}\Big\|u_d+\frac{\delta_{s_\varepsilon^{1+\eta}}}{s_\varepsilon^{1+\eta}}\Big\|%_{1,a}
^\vartheta
\leq C\tau_\varepsilon^{\rho}e^{\nu\vartheta\,\tau_\varepsilon}s_\varepsilon^{\left(1+\eta\right)\vartheta}s_\varepsilon^{-\left(1+\eta\right)\left(2+\vartheta\right)}
%\textcolor{red}{
\left(\|u_d\|%_{1,a}
+1\right)^\vartheta
%}
\\
\leq C\tau_\varepsilon^{\rho}e^{\nu\vartheta\,\tau_\varepsilon}s_\varepsilon^{-2\left(1+\eta\right)}
%\textcolor{red}{
\left(\|u_d\|%_{1,a}
+1\right)^\vartheta
%}
%\left(\|u_d\|%_{1,a}
%^\vartheta+1\right)%<
\leq\frac{\varepsilon}{2}.
%=C\tau_\varepsilon^{\frac{1}{2}}e^{K(1+\alpha_3(\tau_\varepsilon,s_\varepsilon))\tau_\varepsilon}%\cdot\\ \cdot
%s_\varepsilon^{\left(1+\eta\right)\vartheta}%\cdot\\ \cdot
%s_\varepsilon^{\left(1+\eta\right)\vartheta}\left(\|u_d\|_{1,a}^\vartheta+1\right)\\
%=c(\gamma_0,\vartheta)\tau_\varepsilon^{%\frac{3-\vartheta}{4}
%3}s_\varepsilon^{-3(1+\eta)}\left(\|u_d\|_{1,a}^\vartheta+1\right) %\cdot\\ \cdot
%%s^{\left(1+\eta\right)\vartheta}
%.\\
%\leq %c(\alpha_3,\gamma_0,p,r)e^{\tau_\varepsilon}\tau_\varepsilon^{\frac{r(r+1)(2-p)}{8p}}s^{-\left(1+\eta\right)(r+1)}s^{\left(1+\eta\right)r}\leq\\
%\leq c(\alpha_3,\gamma_0,p,r)e^{(3+\vartheta)\alpha_3\tau_\varepsilon}\tau_\varepsilon^{\frac{r(r+1)(2-p)}{8p}}s^{-\left(1+\eta\right)}\;.
\end{multline*}
%where $C$ and $K=(2+\vartheta)\|\alpha_3(s_\varepsilon)\|_\infty+\vartheta\nu$ are the constants of Lemma \ref{stima su w}.
%$$\exists \tilde{\tau_\varepsilon}_2(s)\,\mbox{ such that }\, %\tau_\varepsilon^{\frac{3-\vartheta}{4}}s^{-3(1+\eta)}<\frac{s}{2c(\gamma_0,\vartheta)}$$
Moreover, by $\eqref{deltaeta2}-\eqref{alfa3},$ we deduce that
\begin{multline*}
\|v(T_\varepsilon,x)-u_d\|=\|e^{\alpha_3(%\tau_\varepsilon,
s_\varepsilon)\tau_\varepsilon} z(\tau_\varepsilon,\cdot)-u_d\|= s_\varepsilon^{-(1+\eta)}\|z(\tau_\varepsilon,\cdot)-\,s_\varepsilon^{1+\eta}u_d\|
\\
\leq s_\varepsilon^{-(1+\eta)}\Big(\|z(\tau_\varepsilon,\cdot)-\,(s_\varepsilon^{1+\eta}u_d+\delta_{s_\varepsilon^{1+\eta}})\|+\|\delta_{s_\varepsilon^{1+\eta}}\|\Big)
\leq s_\varepsilon^{-(1+\eta)}\left(\frac{\varepsilon}{4}\,s_\varepsilon^{1+\eta}+\|\delta_{s_\varepsilon^{1+\eta}}\|\right)%<
\leq\frac{\varepsilon}{2}\,. 
\end{multline*}
Therefore, by the last two inequalities we have
\begin{equation*}%\label{u stima}
   \|u(T_\varepsilon,x)-u_d\|%_{1,a}%_{L^2(-1,1)}
   \leq \|u(T_\varepsilon,x)-v(T_\varepsilon,x)\|%_{1,a}%_{L^2(-1,1)}
   +\|v(T_\varepsilon,x)-u_d\|%_{1,a}
   %<
   \leq\varepsilon, %_{L^2(-1,1)}
   %\leq\\
%\leq s %\|k(\cdot,s,\tau)\|%_{L^2(-1,1)}
%+c(\gamma_0,\vartheta)\tau^{\frac{3-\vartheta}{4}}s^{-3(1+\eta)}%\cdot\\ \cdot
%s^{\left(1+\eta\right)\vartheta}
%\longrightarrow 0, \qquad\qquad\qquad \mbox{ as } s\rightarrow 0^+ \mbox{ and } %\tau^{\frac{3-\vartheta}{4}}s^{-3\left(1+\eta\right)}\longrightarrow 0.
\end{equation*}
%Then
%$$\|u(T,x)-u_d\|%_{L^2(-1,1)}
%\longrightarrow 0.\qquad \blacksquare$$
%
from which the conclusion, keeping also in mind the Lemma \ref{NN}.
\end{proof}
\begin{proof}(of Theorem \ref{T1SL}).
The proof of Theorem \ref{T1}
can be adapted to Theorem \ref{T1SL}, keeping in mind that in STEP.2 of the previous proof, %$\omega_{k_*}(x)=\frac{v_d(x)}{\|v_d(x)\|} > 0$ and
the inequality in (\ref{z1a})
%continua a valere in quanto:
continues to hold in this new setting. %also in the Theorem \ref{C1}.
In fact we have
$$
%\textcolor{red}{
\int^1_{-1} u_0(x)\omega_1 (x) dx=\int^1_{-1}
u_0(x)\frac{u_d(x)}{\|u_d\|}dx
=\frac{1}{\|u_d\|}\int^1_{-1} u_0u_d dx>0, \mbox{  by assumption (\ref{H2})}.
%}
$$
From this point on, 
one can proceed as in the proof of Theorem \ref{T1}. 
\end{proof}
%%%%%%%%%%%APPENDIX%
%\appendix
\appendix{
%\textcolor{red}{ 
\section{}
\subsection{Proof of a singular Sturm-Liouville result}\label{PSSL}
%}
\noindent In this section, we recall the proof of Lemma \ref{Autof} (see also \cite{CFproceedings1} and \cite{CF3}).
\begin{proof}{(of Lemma \ref{Autof}).}
%\underline{STEP.1}
We denote by
$\{-\lambda_k\}_{k\in\N}\mbox{  and  }\{\omega_k\}_{k\in\N},$
respectively, the eigenvalues and
%the
orthonormal eigenfunctions of the
%spectral problem $A\omega=\lambda \omega,$ with $A=A_0+\alpha_*I$
operator (\ref{operalfastella}) (see Lemma \ref{spectrum}).
Therefore,
\begin{equation*}\label{3difet}
\langle\omega_k,\omega_h\rangle%_{L^2(-1,1)}
=\int^1_{-1}
\omega_k(x)\omega_h(x)dx=0, \qquad \mbox{ if }h\neq k\,.
\end{equation*}
We can see, by easy calculations, that an eigenfunction of the operator defined in (\ref{operalfastella})
%è la funzione:
is the function
$
\frac{v(x)}{\|v\|%_{L^2(\Omega)}
},
$
associated with the eigenvalue $\lambda=0$.
%%%%%%%%%%%%%%%%%%%%%%%%%%%%%%%%%%%%%%%%%%%%%%%%%%%%%%%%%%%%%%%%%%%%%%%%%%%%%%
Taking into account the above and considering that
$v(x)>0,\,\forall x \in (-1,1)$
\begin{equation*}
\label{4}
\,\exists\, k_* \in \N\,\,:\omega_{k_*}(x)=\frac{v(x)}{\|v\|%_{L^2(\Omega)}
}>0\,  \mbox{ or }  \,\omega_{k_*}(x)=-\frac{v(x)}{\|v\|%_{L^2(\Omega)}
}<0,\,\,\forall x\in (-1,1)\,.
\end{equation*}
%Scritta la
%Writing \eqref{3difet} with $k=k_*$
%si ha:
% we obtain
%\begin{equation}
%\label{5} \langle\omega_{k_*},\omega_h\rangle_{L^2(-1,1)}=\int^1_{-1}
%\omega_{k_*}(x)\omega_h(x)dx=0, \qquad \forall h\neq k_*\, .
%\end{equation}
%%Quindi considerando la
%Therefore, considering (\ref{5})
%%e tenuto conto che
%and 
Keeping in mind that
$\int^1_{-1}
\omega_{k_*}(x)\omega_h(x)dx=0,\mbox{ if }h\neq {k_*}$
and
 $\omega_{k_*}>0$ or $\omega_{k_*}<0$ in $(-1,1)$,
%si ha che
we observe that $\omega_{k_*}$
%è l'unica autofunzione ortonormale di
is the only eigenfunction of the operator defined in (\ref{operalfastella})
%che non cambia segno in
that doesn't change sign in $%\Omega=
(-1,1)$.\\
%\noindent
%\underline{STEP.2}
%Proviamo adesso che
Let us now prove that
%\begin{equation}\label{6bis} 
$k_*=1\,$,
%\lambda_1=0
%\end{equation}
%%e quindi di conseguenza
%and consequently
that is, $\lambda_1=0$.\\
By a well-known variational characterization of the first eigenvalue, we have %the following representation
$$\lambda_1=\inf_{u\in H^1_a (-1,1)}\frac{\int^1_{-1} \left(a\,u_x^2\,-\alpha_*\,u^2\right)\, dx}{\int^1_{-1} u^2\, dx}\,\,.$$
By Lemma \ref{spectrum}, since $\lambda_{k_*}= 0,$ it is sufficient to prove that $\lambda_1\geq 0$, or
\begin{equation*}\label{VARIAT}
  \int^1_{-1}\alpha_*\,u^2\, dx\leq   \int^1_{-1} a\,u_x^2\,dx,\qquad \forall\,u\in H^1_a(-1,1).
\end{equation*}
Integrating by parts, we obtain the desired inequality
\begin{multline*}
\int^1_{-1}\alpha_*\,u^2\, dx=-\int^1_{-1}\frac{(a\,v_x)_x}{v}\,u^2\, dx = \int^1_{-1}a\,v_x\left(\frac{u^2}{v}\right)_x\,dx\\
=\int^1_{-1}a\,v_x\frac{2 u u_x}{v}\,dx-\int^1_{-1}a\,v^2_x\left(\frac{u^2}{v^2}\right)\,dx
=2\int^1_{-1}\sqrt{a}\,\frac{v_x}{v}u\sqrt{a}u_x\,\, dx-\int^1_{-1}a\,v^2_x\left(\frac{u^2}{v^2}\right)\,\, dx\\
\leq\int^1_{-1}\,a\,\left(\frac{ v_x u}{v}\right)^2\,dx+\int^1_{-1} a u^2_x\,dx-\int^1_{-1}a\,v^2_x\left(\frac{u^2}{v^2}\right)\,dx = \int^1_{-1} a u^2_x\,\, dx\,.
\end{multline*}
%from which (\ref{VARIAT}).
\end{proof}}
%\textcolor{red}{
%\label{cap1}
\subsection{Positive and negative part}\label{parti}
In this section, we recall a useful regularity property of positive and negative part of a given function.\\
Given $\Omega\subseteq\R^n$, $v:\Omega\longrightarrow\R$ we consider the positive-part function
$$
v^+(x) := \max\left\{v(x),0\right\},\qquad\qquad\qquad\forall x\in \Omega\,,
$$
and the negative-part function
$$
\!\,v^-(x) := \max\left\{0,-v(x)\right\},\qquad\qquad\qquad\forall x\in \Omega\,.%(\footnote{Calls respectively positive and negative part to the function $v$.})
$$
Then we have the following equality
$$
v=v^+ -v^- \qquad\quad\quad \mbox{  in    }\,\Omega\,.
$$
For the functions $v^+$ and $v^-$ the following result of
regularity in Sobolev's spaces will be useful (see \cite{KS},
Appendix $A$ %, and \cite{M}, Chapter 3
).\\
\begin{prop}
\label{A.1}
Let $\Omega\subset\R^n, \, u:\Omega\longrightarrow\R,\,  u\in H^{1,s}(\Omega) ,
\,1\leq s\leq\infty$ (\footnote{
%\textcolor{red}{
By $H^{1,s}(\Omega)$ we denote the usual Sobolev spaces.%}
}). Then
$%\max\,\{u,0\}
u^+,\,u^-\in H^{1,s}(\Omega)$
and, for $1\leq i\leq n,$
\begin{equation*}
%[\,\max(u,0)\,]
(u^+)_{x_i}=
\left\{\begin{array}{l}
{ u_{x_i} \qquad\qquad\qquad\, \mbox{ in }\{x\in\Omega : u(x)>0\}}\\ [2.5ex]
{ 0 \qquad\qquad\qquad\,\,\,\,\mbox{ in }\{x\in\Omega : u(x)\leq0\} }~,
\end{array}\right.
%(u^+)_{x_i}=
%\begin{cases}
%u_{x_i} \qquad\qquad\qquad\, \mbox{ in }\{x\in\Omega : u(x)>0\}
%\\
%0 \qquad\qquad\qquad\quad \mbox{ in }\{x\in\Omega : u(x)\leq0\}
%\end{cases}
\end{equation*}
and
\begin{equation*}
%%[\,\max(u,0)\,]
(u^-)_{x_i}=\left\{\begin{array}{l}
{ -u_{x_i} \qquad\qquad\qquad\!\! \mbox{ in }\{x\in\Omega : u(x)<0\}}\\ [2.5ex]
{ 0 \,\,\qquad\qquad\qquad\,\,\, \mbox{ in }\{x\in\Omega : u(x)\geq0\} }~.
\end{array}\right.
\end{equation*}
%\begin{cases}
%-u_{x_i} \qquad\qquad\qquad\, \mbox{ in }\{x\in\Omega : u(x)<0\}
%\\
%0 \qquad\qquad\qquad\qquad\, \mbox{ in }\{x\in\Omega : u(x)\geq0\}
%\end{cases}
\end{prop}
\begin{rem}\label{partirem}
The previous result holds true replacing the Sobolev spaces $H^{1,s}(\Omega)$ by the weighted Sobolev space $H^{1}_a(-1,1),$ in the case $n=1$ and $\Omega=(-1,1).$
 \end{rem}
%}%PRG

\section{%Results and proofs for the e
 Existence and uniqueness of strict solutions %in ${\cal{H}}(Q_T)$
}\label{AppA}
This appendix contains the proof of Theorem \ref{exB}, obtained in the Ph.D. Thesis \cite{CF3}, that is,
we prove that there exists a unique strict solution $u\in{\cal{H}(Q_T)}$ to $(\ref{Psemilineare})$, for all initial datum $u_0\in H^1_a(-1,1)$.
\\
We prove this theorem %in \cite{CF3} 
under the following assumptions (H.1)-(H.4):
 \begin{enumerate}%\renewcommand{\labelenumi}{\Roman{enumi})}
  \item[(H.1)] $u_0 \in H^1_a(-1,1);$ %H^1_a(-1,1):=\{u\in L^2(-1,1): u \mbox{ locally absolutely continuous in } (-1,1),$\\ $\qquad\qquad\qquad\qquad\qquad\qquad\qquad\hfill\sqrt{a} u_x \in L^2(-1,1)\};$ %L^2(-1,1)$
        \item[(H.2)] $\alpha \in L^\infty (Q_T);$ %is a bilinear control;
%  \item[(A.3)] $f:(-1,1)\times\R\rightarrow \R$ is a Carath\'{e}odory function %(i.e. $f$ is Lebesgue
%%measurable in $x$ for every $u\in \R,$ and continuous in $u$ for almost %all
%%every $x\in\,(-1,1)$)
%such that\\
%there exist $\vartheta\,>1$, $\gamma_0>0,$ and $\gamma_1>0$ such that
%\begin{equation}\label{Superlinearit}
%|f(x,u)|\leq\gamma_0\,|u|^\vartheta, \mbox{ for a.e. } x\in(-1,1), \forall u\in \R\,,
%\end{equation}
%and
%\begin{multline}\label{lipi}
%|f(x,u)-f(x,v)|\\
%\leq\gamma_1\left(1+|u|^{\vartheta-1}+|v|^{\vartheta-1}\right)|u-v|, \mbox{ for a.e. } x\in(-1,1),\;\forall u,v\in \R;
%\end{multline}
%there exists a %nonnegative constant $\nu,$ 
%locally integrable function $\nu:(0,+\infty)\rightarrow\R,\; \nu(t)\geq 0\; \forall t\in (0,+\infty),$
%such that
%\begin{equation}\label{fsigni}
%f(x,u)\,u \leq \nu(t)
%\,u^2,\qquad \mbox{ for a.e. } x\in(-1,1),\;\;\;  \forall u\in \R\,;
%\end{equation}
%Below we will put $\nu_T=e^{\nu T};%e^{\int_0^T \nu(t)\,dt}
\item[(H.3)] $f:Q_T%(-1,1)
\times\R\rightarrow \R$ is a Carath\'{e}odory function (i.e. $f$ is Lebesgue
measurable in $(t,x)$ for every $u\in \R,$ and continuous in $u$ for a.%all
e.$(t,x)\in Q_T%\in\,(-1,1)
$);\\
$t\longmapsto f(t,x,u)$ is locally absolutely continuous for a.e. $x\in (-1,1), \forall u\in\R.$\\%}
% there exist constants $\gamma_0\geq 0, \vartheta\in(1,3)$ and $\nu%\overline{\nu}
%\geq0$
Moreover, %such that
\begin{itemize}
\item  there exist $\vartheta\,\geq1$, $\gamma_0\geq0$ and $\gamma_1\geq0$ such that
\begin{equation}\label{hsuperl}
|f(t,x,u)|\leq\gamma_0\,|u|^\vartheta, \mbox{ for a.e. } (t,x)\in Q_T,\: \forall u\in \R\,,
\end{equation}
%where $\gamma_0>0,$
%\item
\begin{equation}\label{lip}
|f(t,x,u)-f(t,x,v)|
\leq\gamma_1\left(1+|u|^{\vartheta-1}+|v|^{\vartheta-1}\right)|u-v|, \mbox{ for a.e. } (t,x)\in Q_T,\;\forall u,v\in \R;
\end{equation}
\item there exists a constant $\nu\geq0$ %locally integrable function $\nu:(0,+\infty)\rightarrow\R,\; \nu(t)\geq0,\; \forall t\in (0,+\infty),$
     such that
\begin{equation}\label{hfsign}
f(t,x,u)\,u \leq \nu%(t)
\,u^2,\qquad \mbox{ for a.e. } (t,x)\in Q_T,\;\;\;  \forall u\in \R,\,
\end{equation}
\begin{equation}\label{hderfsign}
\;\;\;f_t(t,x,u)\,u \geq -\nu%(t)
\,u^2,\qquad \mbox{ for a.e. } (t,x)\in Q_T,\; \forall u\in \R,\,
\end{equation}
%where $c(t):(0,+\infty)\rightarrow\R,\; c(t)\geq 0\; \forall t\in (0,+\infty),$ is a local integrable %function on $(0,+\infty).$
below we will put $\nu_T=e^{%\int_0^T
\nu%(t)\,dt
T};%e^{\int_0^T \nu(t)\,dt}
$
\end{itemize}

  \item[(H.4)] $a \in C^1([-1,1])$ is such that
    $$a(x)>0, \,\, \forall \, x \in (-1,1),\quad a(-1)=a(1)=0,$$
    and, the function $\xi_a(x)=\int_0^x \frac{ds}{a(s)}$ satisfies the following
    \begin{equation*}%\label{Lintrod}
    \xi_a\in L^{2\vartheta-1}(-1,1).
  \end{equation*}
\end{enumerate}

%\begin{enumerate}%\renewcommand{\labelenumi}{\Roman{enumi})}
  \begin{rem}
We observe that assumptions $(H.3),\,(H.4)$ are more general than the assumptions $(A.3),\,(A.4)$ (see also Remark \ref{rem1} and Remark \ref{remmon}).
\end{rem}
The proof  of Theorem \ref{exB} follows from the next two lemmas. 
Firstly, the following Lemma \ref{AppL1} assures the local existence and uniqueness of the strict solution to $(\ref{Psemilineare}).$%  (see also \cite{CF3}).
\begin{lem}\label{AppL1}
For every $R>0,$ there is $T_R>0$ %and a positive constant $c_1(T)$ (nondecreasing in $T$)
such that for all $\alpha\in L^\infty(-1,1)$ and all $u_0\in H^1_a(-1,1)$ with $\|u_0\|_{1,a}\leq R$ there is a unique strict solution $u\in{\cal{H}}(Q_{T_R})$ to $(\ref{Psemilineare}).$
\end{lem}

\begin{proof}
Let us fix $R>0,$ $u_0\in H^1_a(-1,1)$ such that $\|u_0\|_{1,a}\leq R.$
Let $0<T\leq 1$ %$\,T_R=\min\{\frac{1}{2^\vartheta C_0^\vartheta(1) c^2(\gamma_0,\vartheta,a)R^{2(\vartheta-1)}},1\}>0\,,$
%where $c(\gamma_0,\vartheta,a)$ is a positive constant. %
(further constraints on $T$ will be imposed below).
We define
 $${\cal{H}}_R(Q_{T}):=\{u\in {\cal{H}}(Q_{T})\,:\,\|u\|_{{\cal{H}}(Q_{T})}\leq 2C_0(1) R\},$$
 where $C_0(1)$ is the constant $C_0(T)$ (nondecreasing in $T$) defined in Proposition \ref{MaxReg} and valued in $1$.   Then, let us define the following map
$$\Lambda :{\cal{H}}_R(Q_{T})\longrightarrow{\cal{H}}_R(Q_{T}),$$
such that
\begin{equation*}\label{contr}
   \Lambda(u)(t):=e^{tA}u_0+\int_0^t e^{(t-s)A} %f(s,x,u(s,x))
   %f(\cdot,u(s))
   \phi(s,u(s))\,ds\,,\;\forall t\in [0,T].
\end{equation*}
%where $\phi(u(s))=f(\cdot,u(s)).$\\
STEP. 1 \;\;We prove that %exists $\bar{T}_R$ such that
the map $\Lambda$ is well defined for some $T$.\\
Fix $u\in {\cal{H}}_R(Q_{T}).$
Let us consider
$U(t,x):=\Lambda\left(u\right)(t,x),$ %-v(t,x)\;\; \mbox{ in }\; Q_T,$$
%si ha il seguente sistema
then $U$ is solution of the following linear problem
\begin{equation}\label{LA4U}
\begin{cases}
U_t-(a\,U_x)_x=\alpha U +f(t,x,u) \;\;\; \mbox{ in }\; Q_{T}
\\
a(x) U_x(t,x)|_{x=\pm 1}=0
\\
U(0,x)=u_0\qquad\qquad\qquad  .
\end{cases}
\end{equation}
By Lemma \ref{f in L2}, $f(\cdot,\cdot,u)\in L^2(Q_{T})=L^2(0,T;L^2(-1,1)),$ then applying Proposition \ref{MaxReg} we deduce that a unique solution $U\in {\cal{H}}(Q_{T})$ of (\ref{LA4U}) exists and we have
$$\|U\|_{{\cal{H}}(Q_{T})}\leq C_0(T) \left(\|f(\cdot,\cdot,u)\|_{L^2(Q_{T})}+\|u_0\|_{1,a}\right).$$
Thus, keeping in mind that $C_0(T)\leq C_0(1),$ by our choice of $T,$ and applying Corollary \ref{f in L2 cor} we obtain
\begin{multline*}%\label{}
\|U\|_{{\cal{H}}(Q_{T})}\leq
 C_0(1) \left(\|f(\cdot,\cdot,u)\|_{L^2(Q_{T})}+\|u_0\|_{1,a}\right)\\
 \leq C_0(1)\left(\gamma_0\|u\|^\vartheta_{L^{2\vartheta}(Q_{T})}+\|u_0\|_{1,a}\right)
 \leq C_0(1)\left(c\,T^\frac{1}{2}\|u\|^\vartheta_{{\cal{H}}(Q_{T})}+\|u_0\|_{1,a}\right)\\
 \leq
 C_0(1)\left(c\,T^\frac{1}{2}(2C_0(1)R)^\vartheta+R\right)
 \leq C_0(1)\left(%2^{\vartheta}
 c\,C_0^\vartheta(1) R^\vartheta T^\frac{1}{2}+R\right).
\end{multline*}
Now, we fix
$T_0(R)=\min\left\{\frac{1}{C_0^{2\vartheta}(1) c^2R^{2(\vartheta-1)}},1\right\}.$ Then we have
 \begin{equation*}
    \|\Lambda(u)\|_{{\cal{H}}(Q_{T})}%=\|U\|_{{\cal{H}}(Q_{T})}
     \leq C_0(1)\left(c\,C_0^\vartheta(1)\,R^\vartheta T^\frac{1}{2}+R\right)
\leq 2C_0(1)R, \; \forall T\in[0,T_0(R)].
\end{equation*}
Thus,
$\Lambda u\in{\cal{H}}_R(Q_{T}),\; \forall T\in[0,T_0(R)].$ 
\\
STEP. 2  \;\; We prove that exists
$T_R\leq T_0(R)
$ such that the map $\Lambda$ is a contraction.\\
Let $T, \,
0<T\leq T_0(R)$ (T will be fix below).  Fix $u,v\in {\cal{H}}_R(Q_T)$ and
set
$W:=\Lambda(u)-\Lambda(v),$%-v(t,x)\;\; \mbox{ in }\; Q_T,$$
%si ha il seguente sistema
\, $W$ is solution of the following problem
\begin{equation}\label{LA4W}
\begin{cases}
W_t-(a\,W_x)_x=\alpha W +f(t,x,u)-f(t,x,v)  \;\;\; \mbox{ in }\; Q_T
\\
a(x) W_x(t,x)|_{x=\pm 1}=0
\\
W(0,x)=0\qquad\qquad\qquad  .
\end{cases}
\end{equation}
%By Lemma \ref{f in L2}
By Lemma \ref{f in L2} $f(\cdot,\cdot,u)\in L^2(Q_T)%=L^2(0,T;L^2(-1,1))
$ and applying Proposition \ref{MaxReg} we deduce that
a unique solution $W\in {\cal{H}}(Q_T)$ of (\ref{LA4W}) exists and we have
\begin{equation}\label{Wcontr}
\|W\|_{{\cal{H}}(Q_T)}\leq
 C_0(T) %\left(
 \|f(\cdot,\cdot,u)-f(\cdot,\cdot,v)\|_{L^2(Q_T)}.%+\|u_0\|_{1,a}\right)\,.
\end{equation}
Moreover, applying the inequality (\ref{lip}) (see assumptions $(H.3)$) and H\"older inequality we obtain
\begin{multline}\label{W H0}
\int_{Q_T}|f(t,x,u)-f(t,x,v)|^2\,dx\,dt \leq\gamma_1^2\int_{Q_T}\left(1+|u|^{\vartheta-1}+|v|^{\vartheta-1}\right)^2|u-v|^2\,dx\,dt\\
\leq c\,\left(\int_{Q_T}\left(1+|u|^{2(\vartheta-1)}+|v|^{2(\vartheta-1)}\right)^{\frac{\vartheta}{\vartheta-1}}dx dt\right)^{\!\!\frac{\vartheta-1}{\vartheta}}\!\!\!\left(\int_{Q_T}|u-v|^{2\vartheta}\,dx dt\right)^{\!\!\frac{1}{\vartheta}}\\
\leq c  \left(T^{1-\frac{1}{\vartheta}}+\|u\|^{2(\vartheta-1)}_{L^{2\vartheta}(Q_T)}+\|v\|_{L^{2\vartheta}(Q_T)}
^{2(\vartheta-1)}\right)
%^{2(\vartheta-1)}
\!\!\|u-v\|^2_{L^{2\vartheta}(Q_T)}.
\end{multline}
%where $M_T=\max\{T,1\}.$\\
Then, by (\ref{Wcontr}) and (\ref{W H0}), applying Corollary \ref{sob3} we have
\begin{multline*}\label{W H}
   \|\Lambda(u)-\Lambda(v)\|^2_{{\cal{H}}(Q_T)}
    \leq c\,%C^2_0(1) T^{1-\frac{1}{\vartheta}}
    \left(T^{1-\frac{1}{\vartheta}}+\|u\|^{2(\vartheta-1)}_{{\cal{H}}(Q_T)}+\|v\|^{2(\vartheta-1)}_{{\cal{H}}(Q_T)}\right)
T^\frac{1}{\vartheta}\|u-v\|^2_{{\cal{H}}(Q_T)}\\
\leq\,c\left(1+\|u\|^{2(\vartheta-1)}_{{\cal{H}}(Q_T)}+\|v\|^{2(\vartheta-1)}_{{\cal{H}}(Q_T)}\right)
T^\frac{1}{\vartheta}\|u-v\|^2_{{\cal{H}}(Q_T)}
    \leq c\left[1+2(2C_0(1)R)^{2(\vartheta-1)}\right]\,T^{\frac{1}{\vartheta}}
    %c \left(1+2R^{2(\vartheta-1)}\right)T
    \|u-v\|^2_{{\cal{H}}(Q_T)}.
\end{multline*}
Let $T_1(R)=\left(\frac{1}{2c\left[1+2(2C_0(1)R)^{2(\vartheta-1)}\right]}\right)^\vartheta,$
and we define $T_R=\min\{T_0(R),T_1(R)\}.$
Then, %we have that
$\Lambda$ is a contraction map.
Therefore, $\Lambda$ has a unique fix point in ${\cal{H}}_R(Q_{T_R}),$ from which the conclusion follows.\\
%$$\qquad\qquad\qquad\qquad\qquad\qquad\qquad\qquad\qquad\qquad\qquad\qquad\hfill\blacksquare$$
\end{proof}

%%%%%%%%%%%%%%%%%%%%%%%%%%%%%%%%%%%%%%%%%%%%%%%%%%%%%%%%%%%%%%%%%%%%%%%%%%%%%%%%
%%%%%%%%%%%%%%%%%%%%%%%%%%%%%%%%%%%%%%%%%%%%%%%%%%%%%%%%%%%%%%%%%%%%%%%%%%%%%%%%
Now, thanks to a classical result (see, e.g., \cite{Lunardi} and \cite{Pazy}), the following Lemma \ref{esist glob} assures the global existence of the strict solution to $(\ref{Psemilineare}),$ so we obtain the complete proof of Theorem \ref{exB}.
\begin{lem}\label{esist glob}
Let $T>0,\; u_0\in H^1_a(-1,1)$ and let $\alpha\in L^\infty(-1,1).$ %\colorbox{red}{be piecewise static.}
%la soluzione di
The strict solution $u\in {\cal{H}}(Q_T)$ of system \eqref{Psemilineare}
%soddisfa la seguente disuguaglianza:
satisfies the following estimate
$$\|u\|_{{\cal{H}}(Q_{T})}\leq C(\|u_0\|_{1,a})e^{k T}\|u_0\|_{1,a},$$
where
$C(\|u_0\|_{1,a})=h\,
%\sqrt{
\left(1+\|u_0\|^{\vartheta-1}_{1,a}\right)^{1+\frac{\vartheta}{2}}%\left(1+\|u_0\|^{\vartheta}_{1,a}\right)%}
,$\; $h$ and $k$ are positive constants.%,  and $k_2=1%\frac{1}{2}
%+(\nu+\|\alpha^+\|_\infty)\vartheta.$ %$\chi_T=e^{2(\nu+\|\alpha^+\|_\infty) T}.$
\end{lem}
\begin{proof}
Multiplying by $u_t$ both members of the equation in $(\ref{Psemilineare})$  and integrating %the following
%inequality
on $(-1,1)$ %and applying assumption (\ref{fsign})
we obtain
%$$
%u_t-(a(x)u_x)_x=\alpha u + f(x,u)
%$$
%si ha
%we have
%\begin{multline*}
%\label{1glob}
$$\int^1_{-1}
u_t^2(t,x)dx-\int^1_{-1}\left(a(x)u_x(t,x)\right)_x u_{t}(t,x)dx
=\int^1_{-1}\alpha(x) u(t,x)\,u_t(t,x)
dx+\int^1_{-1}f(t,x,u)\,u_t(t,x)\,dx,%\leq \int^1_{-1}\alpha^+(t,x) u^2+ \int^1_{-1}\nu(t) u^2 dx .
$$
%\end{multline*}
thus,
%\begin{multline*}
%\label{anteutl2}
$$\int^1_{-1} u_t^2(t,x)dx+\frac{1}{2}\frac{d}{dt}\int^1_{-1}a(x)u^2_x(t,x)\,dx
=\frac{1}{2}\frac{d}{dt}\int^1_{-1}\alpha(x) u^2(t,x)
dx+\int^1_{-1}f(t,x,u)\,u_t(t,x)\,dx.
$$%\end{multline*}
Now, let us consider the following function $F:Q_T\times\R\longrightarrow\R,$
\begin{equation*}\label{F}
    F(t,x,u):=\int_0^{u}f(t,x,\zeta)\,d\zeta\,, \;\forall (t,x,u)\in Q_T\times\R.
\end{equation*}
Then, we observe that %the following properties
 \begin{equation}\label{dF}
\frac{\partial F(t,x,u(t,x))}{\partial t}=f(t,x,u(t,x))u_t(t,x)+\int_0^uf_t(t,x,\zeta)\,d\zeta, \;\forall (t,x)\in Q_T.
\end{equation}
Moreover, by (\ref{hsuperl}) (see assumptions $(H.3)$), we have
$$F(0,x,u_0(x))=\int_0^{u_0}\,f(0,x,\zeta)d\zeta\leq \gamma_0\int_0^{u_0} |\zeta|^\vartheta\,d\zeta=\frac{\gamma_0}{\vartheta+1}|u_0|^{\vartheta+1}, \;\forall x\in (-1,1).$$
Then, by Lemma \ref{sob1}, we deduce that
\begin{equation}\label{F1}
    \int_{-1}^1\,|F(0,x,u_0(x))|\,dx\leq \frac{\gamma_0}{\vartheta+1} \|u_0\|_{L^{\vartheta+1}(-1,1)}^{\vartheta+1}\leq c \|u_0\|_{1,a}^{\vartheta+1}.
\end{equation}
Now, we observe the following property of the function $F$:\\
keeping in mind that, by (\ref{hfsign}), for almost every $(t,x)\in Q_T,$ we obtain
\begin{itemize}
  \item $f(t,x,\zeta)\leq\nu%(t) 
  \zeta,$ for every $\zeta\in\R, \zeta\geq0$
  \item $f(t,x,\zeta)\geq\nu%(t)
   \zeta,$ for every $\zeta\in\R, \zeta<0,$
\end{itemize}
then, for almost every $(t,x)\in Q_T,$ we  have
\begin{itemize}
  \item  for every $u\in\R, u\geq0, \,F(t,x,u)=\int_0^{u}\,f(t,x,\zeta)d\zeta\leq\nu%(t)
  \int_0^{u}\,\zeta d\zeta= \frac{\nu%(t)
  }{2}u^2$
  \item  for every $u\in\R, u<0, \,F(t,x,u)=-\int^0_{u}\,f(t,x,\zeta)d\zeta\leq-\nu%(t)
  \int_u^{0}\,\zeta d\zeta=\frac{\nu%(t)
  }{2}u^2.$
\end{itemize}
%So,
%\begin{equation}\label{F2}
%F(x,u)=\int_0^{u}\,f(x,\zeta)d\zeta\leq \nu\int_0^{u} |\zeta|^2\,d\zeta
%    \leq \frac{\nu}{2} u^2\,.
%\end{equation}
Then,
\begin{equation}\label{F3}
   F(t,x,u)%=\int_0^{u}\,f(x,\zeta)d\zeta
   \leq \frac{\nu%(t)
   }{2} u^2\,, \quad\forall (t,x,u)\in Q_T\times\R.
\end{equation}
Now, by \eqref{hderfsign}, proceeding similarly to \eqref{F3}, we obtain
\begin{equation}\label{F3der}
   %F(t,x,u)%=\int_0^{u}\,f(x,\zeta)d\zeta
   \int_0^uf_t(t,x,\zeta)\,d\zeta\geq -\frac{\nu%(t)
   }{2} u^2\,, \quad\forall (t,x,u)\in Q_T\times\R.
\end{equation}
In effect, 
%%%%%%%%%%%%%%%%%%%%%%%%
%keeping in mind that, 
by (\ref{hderfsign}), for almost every $(t,x)\in Q_T,$ we deduce that
\begin{itemize}
  \item $f_t(t,x,\zeta)\geq -\nu%(t)
   \zeta,$ for every $\zeta\in\R, \zeta\geq0$
  \item $f_t(t,x,\zeta)\leq-\nu%(t)
   \zeta,$ for every $\zeta\in\R, \zeta<0,$
\end{itemize}
then, for almost every $(t,x)\in Q_T,$ we obtain
\begin{itemize}
  \item  for every $u\in\R, u\geq0, \,\int_0^{u}\,f_t(t,x,\zeta)d\zeta\geq-\nu%(t)
  \int_0^{u}\,\zeta d\zeta= -\frac{\nu%(t)
  }{2}u^2$
  \item  for every $u\in\R, u<0, \,\int^u_{0}\,f_t(t,x,\zeta)d\zeta=-\int^0_{u}\,f_t(t,x,\zeta)d\zeta\geq\nu%(t)
  \int_u^{0}\,\zeta d\zeta= -\frac{\nu%(t)
  }{2}u^2.$
\end{itemize}
%So,
%\begin{equation}\label{F2}
%F(x,u)=\int_0^{u}\,f(x,\zeta)d\zeta\leq \nu\int_0^{u} |\zeta|^2\,d\zeta
%    \leq \frac{\nu}{2} u^2\,.
%\end{equation}
%Then
%\begin{equation}\label{F3}
%   F(t,x,u)%=\int_0^{u}\,f(x,\zeta)d\zeta
%   \leq \frac{\nu(t)}{2} u^2\,, \quad\forall (t,x,u)\in Q_T\times\R.
%\end{equation}
%%%%%%%%%%%%%%%%%%%%%%%%
By %(\ref{ante}) and
 (\ref{dF}), we deduce
%\begin{multline*}%\label{past F}
$$\int^1_{-1}
u_t^2(t,x)dx+\frac{1}{2}\frac{d}{dt}\int^1_{-1}\left\{a(x)u^2_x(t,x)\,-\alpha(x) u^2(t,x)-2F(t,x,u)\right\}\,dx\\
+\int_{-1}^1\int_0^uf_t(t,x,\zeta)\,d\zeta\,dx=0.
$$%\end{multline*}
Fix $t\in(0,T)$ and integrate on $(0,t),$ we have
\begin{multline*}\label{Fint1}
\int^t_{0}\int^1_{-1}
u_t^2(s,x)dx\,ds+\frac{1}{2}\int^1_{-1}\left\{a(x)u^2_x(t,x)\,-\alpha(x)u^2(t,x)\right\}\,dx\\
=\int^1_{-1}F(t,x,u(t,x))\,dx+\frac{1}{2}\int^1_{-1}\left\{a(x)u^2_{0x}(x)\,
-\alpha(x)u_0^2(x)\right\}\,dx\\
-\int^1_{-1}F(0,x,u_0(x))\,dx\,
- \int_{0}^t\int_{-1}^1\int_0^uf_t(t,x,\zeta)\,d\zeta\,dx\,dt.
\end{multline*}
Thus, by $\eqref{F1}-\eqref{F3der},$ we obtain
\begin{multline*}%\label{Fint1}
\int^t_{0}\|u_t(s,\cdot)\|^2\,ds+\|\sqrt{a}u_x(t,\cdot)\|^2\,%_{L^2(Q(t))}
\\\leq\left(\|\alpha^+\|_\infty %\|u(t,\cdot)\|^2\,
+\nu%(t)
\right)\|u(t,\cdot)\|^2
+\|\sqrt{a}u_{0x}\|^2\,+\|\alpha^-\|_\infty \|u_0\|^2+2\int^1_{-1}|F(0,x,u_0(x))|\,dx+\nu\int_{0}^t%\nu(s)
\|u(s,\cdot)\|^2\,ds\\
\leq \left(\|\alpha^+\|_\infty+\nu%(t)
\right)\|u(t,\cdot)\|^2+|u_0|^2_{1,a}+\|\alpha^-\|_\infty \|u_0\|^2+c\,\|u_0\|_{1,a}^{\vartheta+1}+\nu\int_{0}^t%\nu(s)
\|u(s,\cdot)\|^2\,ds,
\end{multline*}
%\textcolor{red}{
where we denote with $\alpha^+$, $\alpha^-$ the positive and negative part of $\alpha$, respectively 
(see \ref{parti}).
%} 
\\
Let us consider for simplicity $\chi_T:=e^{(\nu+\|\alpha^+\|_\infty) T}.$
%\textcolor{red}{
By Corollary %PRG
 \ref{L2} (see also Remark \ref{dcstrict}), %}
we deduce
\begin{multline*}%\label{H}
\|u(t,\cdot)\|^2+\|\sqrt{a}u_x(t,\cdot)\|^2+\int^t_{0}\|u_t(s,\cdot)\|^2\,ds\\
\leq \left(\|\alpha^+\|_\infty+\nu%(t)
+1\right)\|u(t,\cdot)\|^2+|u_0|^2_{1,a}+\|\alpha^-\|_\infty \|u_0\|^2+c\,\|u_0\|_{1,a}^{\vartheta+1}+\nu\int_{0}^t%\nu(s)
\|u(s,\cdot)\|^2\,ds\\
\leq %e^{\|\alpha^+\|_\infty+\nu%(t)+1}
c\,\|u\|_{{\cal{B}}(Q_t)}^2+|u_0|^2_{1,a}+\|\alpha^-\|_\infty \|u_0\|^2+c\,\|u_0\|^{\vartheta+1}_{1,a}+\nu\,t\,\|u\|_{{\cal{B}}(Q_t)}^2%\int_{0}^t\nu(s)\,ds
\,\\
\leq \left[%e^{\|\alpha^+\|_\infty+\nu(t)+1} 
(c+\nu\,T)\,\nu_T^2\,e^{2\|\alpha^+\|_\infty T%+\nu T%\int_{0}^T\nu(s)\,ds
}+\|\alpha^-\|_\infty+1 \right]\,(\|u_0\|^2+|u_0|^2_{1,a}) +c\,%\left[\|u_0\|_{1,a}^2+
\|u_0\|^{\vartheta+1}_{1,a}%\right]
\\
\leq%\max\left\{%e^{\|\alpha^+\|_\infty+\nu(t)+1} 
c(1+T)\,\nu_T^2\,e^{2\|\alpha^+\|_\infty T}%+\|\alpha^-\|_\infty,c,1\right\}
\left[\|u_0\|^2_{1,a}+\|u_0\|^{\vartheta+1}_{1,a} \right]
\leq c\,(1+T)\chi_T^2 \left[1+\|u_0\|^{\vartheta-1}_{1,a} \right]\|u_0\|^2_{1,a}\,.
\end{multline*}
Moreover, by the equation in $(\ref{Psemilineare}),$ we have
$$\left(a(x)u_x(t,x)\right)_x=u_t(t,x)-\alpha(x)u(t,x)-f(t,x,u),$$
then, for every $t\in(0,T),$ we obtain %
\begin{multline*}%\label{aux}
    \int^t_{0}\|\left(a(\cdot)u_x(s,\cdot)\right)_x\|^2\,ds \leq2\int^t_{0}\|u_t(s,\cdot)\|^2\,ds+2\|\alpha^+\|^2_\infty\,\int^t_{0}\|u(s,\cdot)\|^2\,ds
    +2\int_{Q_t}|f(s,x,u)|^2\,dx\,ds\,\\
    \leq c\,%\max\{T,1\}
    (1+T)\chi_T^2 \left[1+\|u_0\|^{\vartheta-1}_{1,a} \right]\|u_0\|^2_{1,a}+2\int_{Q_t}|f(s,x,u)|^2\,dx\,ds.
\end{multline*}
 By Lemma \ref{f in L2}, we deduce
\begin{multline*}
\int_{Q_t}|f(s,x,u)|^2\,dx\,ds\,\leq \gamma_0^2\int_{Q_t}|u|^{2\vartheta}dx\,ds\,
\leq c\,t\|u\|_{H^1(0,t;L^2(-1,1))}\|u\|_{L^\infty(0,t;H^1_a(-1,1))}^{2\vartheta-1}\\
\leq c\,T\left(\int_0^t\|u_t(s,\cdot)\|^2\,ds\right)^\frac{1}{2}\,
\left(\sup_{t\in[0,T]}\|u(t,\cdot)\|_{1,a}\right)^{2\vartheta-1}\\
\leq c T \left[(1+T)\chi_T^2\left(1+\|u_0\|^{\vartheta-1}_{1,a} \right)\|u_0\|^2_{1,a}\right]^\frac{1}{2}\,\left[(1+T)^{\frac{1}{2}}\chi_T\left(1+\|u_0\|^{\vartheta-1}_{1,a} \right)^\frac{1}{2}\|u_0\|_{1,a}\right]^{2\vartheta-1}\\
\leq c T\,(1+T)^\vartheta \chi_T^{2\vartheta}\left(1+\|u_0\|^{\vartheta-1}_{1,a} \right)^\vartheta\,%\left[\left(1+\|u_0\|^{\vartheta-1}_{1,a}\right)^\frac{1}{2} \right]^{2\vartheta-1}
\|u_0\|^{2\vartheta}_{1,a}\leq\,
c \,e^{(1+\vartheta)T} \chi_T^{2\vartheta}\left(1+\|u_0\|^{\vartheta-1}_{1,a} \right)^\vartheta\,%\left[\left(1+\|u_0\|^{\vartheta-1}_{1,a}\right)^\frac{1}{2} \right]^{2\vartheta-1}
\|u_0\|^{2\vartheta}_{1,a}.
\end{multline*}
%where $\chi_T=e^{(\nu+\|\alpha^+\|_\infty) T}.$
From which, the conclusion
\begin{multline*}
    \|u\|^2_{{\cal{H}}(Q_T)}\leq c\,\left[e^T%\max\{T,1\}
    \chi_T^2 \left(1+\|u_0\|^{\vartheta-1}_{1,a} \right)\|u_0\|^2_{1,a}%\right.\\ \left.
    +e^{(1+\vartheta)T} \chi_T^{2\vartheta}\left(1+\|u_0\|^{\vartheta-1}_{1,a} \right)^\vartheta\,\|u_0\|^{2\vartheta}_{1,a}\right]\\
    \leq c\,e^{(1+\vartheta)T}%\max\{T,1\}
    \chi_T^{2\vartheta}\left[1+\|u_0\|^{\vartheta-1}_{1,a}+
    \left(1+\|u_0\|^{\vartheta-1}_{1,a} \right)^\vartheta \right]\left(\|u_0\|^{2}_{1,a}+\|u_0\|^{2\vartheta}_{1,a}\right)\\
    \leq c\,e^{(1+\vartheta)T}%e^T
    \chi_T^{2\vartheta}\left(1+\|u_0\|^{\vartheta-1}_{1,a} \right)^\vartheta\,\left(1+\|u_0\|^{2\vartheta-2}_{1,a}\right)\|u_0\|^{2}_{1,a}\,\\
    \leq c\,e^{(1+\vartheta)T}%e^T%\chi_T^{2\vartheta}
    e^{2(\nu+\|\alpha^+\|_\infty)\vartheta T}
    \left(1+\|u_0\|^{\vartheta-1}_{1,a} \right)^\vartheta\,\left(1+\|u_0\|^{\vartheta-1}_{1,a}\right)^2\|u_0\|^{2}_{1,a}\\
    \leq c\,%e^T%\chi_T^{2\vartheta}
    e^{2[1+\nu+\|\alpha^+\|_\infty]\vartheta T}
    \left(1+\|u_0\|^{\vartheta-1}_{1,a} \right)^{2+\vartheta}\,%\left(1+\|u_0\|^{\vartheta-1}_{1,a}\right)^2
    \|u_0\|^{2}_{1,a}.
\end{multline*}
\end{proof}
%%%%%%%%%%%%%%%%%%%%%%%%%%%%%%%%%%
%%%%%%%%%%%%%%%%%%%%%%%%%%%%%%%%%%
%}
%\section*{Acknowledgments}
\begin{center} \bf Acknowledgments\end{center}
\noindent I wish to express my thanks to professor Piermarco Cannarsa to
suggest the idea of my work. \\
%\textcolor{red}{
The author is indebted to the anonymous referees for the criticism and the useful suggestions which have made this paper easier to read and to understand.
%The authors want to thank the anonymous referees for their valuable comments which have made this paper easier to read and to understand.
%}
%\setlinespacing{1.44}
%\bibliographystyle{amsplain}
\bibliographystyle{elsarticle-num}
%\bibliography{<myrefs>}
%\bibliography{myrefs}
%%\bibliography{referenze}

\begin{thebibliography}{99}
%\bibitem{rif1} Tizio, R., \emph{Using thebibliography}, ...
%\bibitem[Caio, 2002]{rif2} Caio, E., \emph{Another title}, ...
%\bibitem{rif3} Sempronio, F., \emph{Using \LaTeX}, ...
%\end{thebibliography}
%
%
%%%%%%%%%%%%%%%%%%%%%%%%%%%%%%%%%%%%%%%%%%%%%%%%%%%%%%%%%%%%%%%%%%%%%%%%%%%%%%%%%
%
%
%
%%\vfill\eject
%\begin{thebibliography}{0}
%%%%%%%%%%%%%%%%%%%%%%%%%%%%%%%%%%%%
%%\bibitem{SAM}{G.  Amendola, S. Carillo, Thermal work and minimum free
%%energy in a heat conductor with memory}, {\it Quart. J. of Mech. and
%%Appl. Math.} {\bf 57(3)}, (2004) 429--446.
%
%%\bibitem{Cattaneo} {C. Cattaneo}, {Sulla conduzione
%%del calore}, {\it Atti Sem. Mat. Fis. Universit\'a di ~ Modena} {\bf 3},
%%(1948) { 83--101}.
%
%%\bibitem{JNMP}{ S. Carillo}, {Some remarks on materials with memory:
%%heat conduction and viscoelasticity}, {\it J. Nonlinear Math. Phys.}
%%{\bf 12} suppl. 1, (2005)  163--178.
%
%%\bibitem{FGR} { M. Fabrizio, G. Gentili, D.W.Reynolds},
%%{On rigid heat conductors with memory}, {\it Int. J. Eng. Sci. } {\bf
%%36}, (1998) { 765--782}.
%
%%%%%%%%%%%%%%%%%%%%%%%%%%%%%%%%%%%%%%%%%%%%%%%%%%%%%%%%%%%%%%%%%%%%%%%%%%%%%%%
\bibitem{ACF}{F. Alabau-Boussouira, P. Cannarsa, G. Fragnelli},
{ Carleman estimates for degenerate parabolic operators with
applications to null controllability}, { \it J. Evol. Equ. } {\bf 6}, no.
2, (2006) 161--204.
%
\bibitem{Ba} {J.M. Ball}, { Strongly continuous semigroups, weak solutions, and the variation of constants formula}, { \it Proceedings of the American Mathematical Society}  {\bf 63}, (1977)  370--373.
%
\bibitem{BS} {J.M. Ball, M. Slemrod}, { Nonharmonic Fourier series and the stabilization of distributed semi-linear control systems}, { \it Comm. Pure. Appl. Math.} {\bf 32}, (1979)  555--587.
%%%%%
%Baudena, M., F. DÕAndrea, and A. Provenzale (2008), A model for soil-vegetation-atmosphere interactions in water-limited
%ecosystems, Water Resour. Res., 44, W12429, doi:10.1029/2008WR007172.
%%%%%
%
\bibitem{Prov1}{M. Baudena, F. DÕAndrea, A. Provenzale},
{ A model for soil-vegetation-atmosphere interactions in water-limited
ecosystems}, { \it Water Resour. Res.}, {\bf 44}, {W12429, doi:10.1029/2008WR007172}, (2008) 1--9.
%
%@incollection{BDDM1,
%  author = 	 {A. Bensoussan and G. Da Prato and G. Delfour and S.K. Mitter},
%  booktitle = 	 {Representation and control of infinite dimensional systems},
%  series = 	 {Vol.1},
%  publisher = 	 {Systems Control Found. Appl.%, Birkh\ddot{a}user
%  },
%  howpublished =  {Boston},
%  year = 	 {1992},
%}
\bibitem{BDDM1} {A. Bensoussan, G. Da Prato, G. Delfour, S.K. Mitter}, { Representation and control of infinite dimensional systems}, {\bf Vol. 1}, {\it Systems Control Found. Appl.}, (1992).
%
\bibitem{BDDM2} {A. Bensoussan, G. Da Prato, G. Delfour, S.K. Mitter}, { Representation and control of infinite dimensional systems}, {\bf Vol. 2}, {\it Systems Control Found. Appl.}, (1993).
%
\bibitem{BR} {H. Brezis}, { Functional Analysis, Sobolev Spaces and Partial Differential Equations}, {\it Universitext, Springer}, (2010).
%
\bibitem{B1} {M. I. Budyko},
{On the origin of glacial epochs}, {\it Meteor. Gidsol.}, {\bf 2}, (1968) 3--8.


\bibitem{B2} {M. I. Budyko},
{The effect of solar radiation variations on the climate of the
earth}, { \it Tellus} {\bf 21}, (1969) 611--619.

\bibitem{CMP} {M. Campiti, G. Metafune, D. Pallara}, { Degenerate self-adjoint evolution equations on the unit
interval}, { \it Semigroup Forum} {\bf 57}, (1998)  1--36.
%
\bibitem{CFproceedings1} {P. Cannarsa, G. Floridia}, {Approximate controllability for linear degenerate parabolic problems with bilinear control}, {\it Proc. Evolution Equations and Materials with Memory 2010}, Casa Editrice Universit La Sapienza Roma, (2011), 19--36.

%

\bibitem{CF2} {P. Cannarsa, G. Floridia}, {Approximate multiplicative controllability for degenerate parabolic problems with Robin boundary conditions}, {\it Communications in Applied and Industrial Mathematics}, doi=10.1685/journal.caim.376, issn=2038-0909, url=http://caim.simai.eu/index.php/caim/article/view/376, (2011), 1--16.
%
\bibitem{CFK} {P. Cannarsa, G. Floridia, A.Y. Khapalov}, { On multiplicative controllability of the 2-D reaction-diffusion equation on a disc}, { \it preprint}.

%
\bibitem{CFKP} {P. Cannarsa, G. Floridia, A.Y. Khapalov, F.S. Priuli}, {Controllability of a swimming model for incompressible Navier-Stokes equations}, {\it work in progress}.

%
\bibitem{CK} {P. Cannarsa, A.Y. Khapalov}, {Multiplicative controllability for the one dimensional parabolic
equation with target states admitting finitely many changes of
sign}, {\it Discrete and Continuous Dynamical Systems-Ser. B}, {\bf 14},  no.4, (2010), 1293--1311.
%
\bibitem{CMV2} {P. Cannarsa, P. Martinez, J. Vancostenoble}, { Persistent regional contrallability for a class of degenerate parabolic equations}, { \it Commun. Pure Appl. Anal.} {\bf 3}, (2004) 607--635.
%

\bibitem{CMV1} {P. Cannarsa, P. Martinez, J. Vancostenoble}, { Null controllability of the degenerate heat
equations}, { \it Adv. Differential Equations} {\bf 10}, (2005) 153--190.
%
\bibitem{CMV3} {P. Cannarsa, P. Martinez, J. Vancostenoble}, { Carleman estimates for a class of degenerate parabolic operators}, { \it SIAM J. Control Optim.} {\bf 47}, no.1, (2008) 1--19.
%
\bibitem{CV} {P. Cannarsa, V. Vespri}, { On Maximal $L^p$ Regularity for the Abstract Cauchy Problem}, { \it Boll. Unione Mat. Ital. Sez.B Artic. Ric. Mat.} {\bf 6}, no.5, (1986) 165--175.

%%
%@Article{CV,
%  author = 	 {P. Cannarsa and V. Vespri},
%  title = 	 {On Maximal $L^p$ Regularity for the Abstract Cauchy Problem},
%  journal = 	 {Boll. Unione Mat. Ital. Sez.B Artic. Ric. Mat.},
%  year = 	 {1986},
%  volume = 	 {6},
%  number =       {5},
%  pages = 	 {165--175},
%}
\bibitem{D2} {J.I. Diaz}, {Mathematical analysis of some diffusive energy balance models in Climatology}, 
{\it Mathematics, Climate and Environment}, (1993) 28--56.
%
\bibitem{D1} {J.I. Diaz}, {On the controllability of some simple climate models}, {\it Environment, Economics and their Mathematical Models}, %{\bf 48},
 (1994) 29--44.
%
\bibitem{D} {J.I. Diaz}, {On the mathematical treatment of energy balance climate models}, {\it The mathematics of models for climatology and
environment, (Puerto de la Cruz, 1995), NATO ASI } Ser.I {\it Glob.
Environ. Change}, {\bf 48}, {Springer, Berlin}, (1997) 217--251.
%%%%
%(Puerto de la Cruz, 1995), NATO ASI } Ser.I {\it Glob. Environ. Change, Springer, Berlin},

%%%%%%%%%
%@Article{D,
%  author = 	 {J.I. Diaz},
%  title = 	 {On the mathematical treatment of energy balance
%climate models},
%  journal = 	 {The mathematics of models for climatology and
%environment, (Puerto de la Cruz, 1995), NATO ASI},
%  year = 	 {1997},
%  volume = 	 {48},
%  number =       {Ser.I, Glob.
%Environ. Change, Springer, Berlin},
%  pages = 	 {217--251},
%}
%@Article{D1,
%  author = 	 {J.I. Diaz},
%  title = 	 {On the controllability of some simple climate models},
%  journal = 	 {Environment, Economics and their Mathematical Models},
%  year = 	 {1994},
%  volume = 	 {},
%  number =       {},
%  pages = 	 {29--44},
%}
%@Article{D2,
%  author = 	 {J.I. Diaz},
%  title = 	 {Mathematical analysis of some diffusive energy balance models in Climatology},
%  journal = 	 {Mathematics, Climate and Environment,},
%  year = 	 {1993},
%  volume = 	 {48},
%  number =       {},
%  pages = 	 {28--56},
%}

%%%%%%%%%

\bibitem{DH} {J.I. Diaz, G. Hetzer, L. Tello}, { An Energy Balance Climate Model with Hysteresis}, {\it Nonlinear
Analysis}, {\bf 64}, (2006) 2053--2074.
%%%
\bibitem{edw} {R.E. Edwards}, { Functional Analysis: Theory and Applications}, {\it Dover Books on Mathematics}, (1965).
%%
%\bibitem{Evans} {L.C. Evans}, { Partial differential equations}, {\it Grad. Stud. Math.}, {\bf 19}, (1998).

%@incollection{Evans,
%  author = 	 {L.C. Evans},
%  booktitle = 	 {Partial differential equations},
%  series = 	 {Vol. 19},
%  publisher = 	 {Grad. Stud. Math.},
%  howpublished =  {Providence, Rhode Island},
%  year = 	{1998},
%}

%%
\bibitem{FC} {E. Fernandez-Cara}, { Null controllability of the semilinear heat equation}, { \it ESAIM COCV}, {\bf 2}, (1997) 87--103.   %%%%%%%%%FIN QUI
%
\bibitem{FCZ} {E. Fernandez-Cara, E. Zuazua}, { Controllability for blowing up semilinear parabolic
equations},
 { \it C. R. Acad. Sci. Paris Ser. I Math.}, {\bf 330}, (2000) 199--204.
%
\bibitem{Fichera} {G. Fichera}, {On a degenerate evolution problem}, {\it Partial differential equations with real analysis}, {H. Begeher, A. Jeffrey},  {Pitman}, (1992), 1--28.
%
\bibitem{CF3} {G. Floridia}, {Approximate multiplicative controllability for degenerate parabolic problems and Regularity properties of elliptic and parabolic systems}, {{Ph. D. Thesis}, {\it University of Catania}}, {Supervisor: Prof. Piermarco Cannarsa}, (2011), 1--161.

%
\bibitem{FI} {A. Fursikov, O. Imanuvilov}, {Controllability of evolution equations}, { \it Res. Inst. Math., GARC, Seoul National University, Lecture Note Ser.},
  {\bf 34}, (1996).
%

\bibitem{H} {G. Hetzer},
{The number of stationary solutions for a one-dimensional
Budyko-type climate model}, {\it Nonlinear Anal. Real World Appl.}
{\bf 2}, (2001) 259--272.

%
\bibitem{K1} {A.Y. Khapalov}, {Global approximate controllability properties for the semilinear heat equation with superlinear term}, { \it Rev. Mat. Complut.},
{\bf 12}, (1999) 511--535.
%
\bibitem{K2} {A.Y. Khapalov}, {A class of globally controllable semilinear heat equations with superlinear terms}, \textit{J. Math. Anal. Appl.},
  {\bf 242}, (2000) 271--283.
%

\bibitem{K} {A.Y. Khapalov},
{Global non-negative controllability of the semilinear
parabolic equation governed by bilinear control} {\it ESAIM: Controle,
Optimisation et Calculus des Variations} {\bf 7}, (2002) 269--283.
%

\bibitem{K3} {A.Y. Khapalov}, {On bilinear controllability of the parabolic equation with the reaction-diffusion term satisfying Newton's Law, in the special issue of the}, \textit{ J. Comput. Appl. Math.},
    {dedicated to the memory of J.-L. Lions }, {\bf 21}, no.1, (2002) 275--297.
%
\bibitem{K4} {A.Y. Khapalov}, {Controllability of the semilinear parabolic equation governed by a multiplicative control in the reaction term: A qualitative approach}, {\it SIAM J. Control Optim.}, {\bf 41}, no. 6 (2003) 1886--1900.
    %Available as
%\textit{Tech. Rep. 01--7, Washington State University,} Department
%of Mathematics (submitted).



\bibitem{KB} {A.Y. Khapalov}, {Controllability of partial differential equations governed
by multiplicative controls}, {\it Lecture Series in Mathematics, Springer}, {\bf
1995}, (2010).
%
%
%@book{KS,
%  author = 	 {D. Kinderlehrer and G. Stampacchia },
%  title = 	 {An introduction to variational inequalities and their
%applications},
%  publisher = 	 {Academic Press},
%  howpublished =  {New York},
%  year = 	 {1980},
%}

\bibitem{KS} {D. Kinderlehrer, G. Stampacchia, }{ An introduction to variational inequalities and their
applications}, {\it Pure and Applied Mathematics} {\bf 88}, {\it Academic Press}, New York, (1980).
%
%\bibitem{LSU} {O.H. Ladyzhenskaya, V.A. Solonikov, N.N. Ural'ceva}, { Linear and Quasi-linear Equations of Parabolic Type.} {\it AMS, Providence, Rhode Island} (1968).

\bibitem{Lunardi} {A. Lunardi}, { Analytic semigroups and optimal regularity in parabolic problems}, {\it Progr. Nonlinear Differential Equations Appl.%, Birkh\ddot{a}user
  }, %{\bf 88}, %{\it Progr. Nonlinear Differential Equations Appl.}, 
  Boston, (1995).

%@incollection{Lunardi,
%  author = 	 {A. Lunardi},
%  booktitle = 	 {Analytic semigroups and optimal regularity in parabolic problems},
%  series = 	 {},
%  publisher = 	 {Progr. Nonlinear Differential Equations Appl.%, Birkh\ddot{a}user
%  },
%  howpublished =  {Boston},
%  year = 	{1995},
%}


\bibitem{OR} {O.A. Oleinik and E.V. Radkevich}, { Second order equations with nonnegative characteristic form}, {\it Applied Mathematical Sciences},
 %{\bf 88}, %{\it Progr. Nonlinear Differential Equations Appl.}, 
(1983).

\bibitem{Pazy} {A. Pazy}, { Semigroups of linear operators and applications to partial differential equations}, {\it American Mathematical Society} {\bf 44}, %{\it Progr. Nonlinear Differential Equations Appl.}, 
Providence, R.I., (1973).



%@incollection{OR,
%  author = 	 {O.A. Oleinik and E.V. Radkevich},
%  booktitle = 	 {Second order equations with nonnegative characteristic form},
%  series = 	 {},
%  publisher = 	 {American Mathematical Society},
%  howpublished =  {Providence, R.I.},
%  year = 	 {1973},
%}


%@incollection{Pazy,
%  author = 	 {A. Pazy},
%  booktitle = 	 {Semigroups of linear operators and applications to partial differential equations},
%  series = 	 {44},
%  publisher = 	 {Applied Mathematical Sciences},
%  howpublished =  {},
%  year = 	 {1983},
%}
%




%\bibitem{M} {C.B.,Jr Morrey}, {Multiple Integrals in the Calculus of Variations},
%{\it Springer, New York}, (1966).


\bibitem{S} {W. D. Sellers,}
{A climate model based on the energy balance of the
earth-atmosphere system}, {\it J. Appl. Meteor.}, {\bf 8},
(1969) 392--400.


\bibitem{TV} {J. Tort, J. Vancostenoble,}
{Determination of the insolation function in the nonlinear Sellers climate model}, {\it Ann. I. H. Poincar}, {\bf 29}, (2012) 683--713.
%THE HABITABLE ZONE OF EARTH-LIKE PLANETS WITH DIFFERENT LEVELS
%OF ATMOSPHERIC PRESSURE
\bibitem{Prov2} {G. Vladilo, G. Murante, L. Silva, A. Provenzale,
G. Ferri, G. Ragazzini},
{The habitable zone of Earth-like Planets with different levels of atmospheric pressure}, {\it The Astrophysical Journal},
(2013) 1--23.
%%%%%%%%%%%%%%%%%%%%%%%%%%%%%%%%%%%%%%%%%%%%%%%%%%%%%%%%%%%%%%%%%%%%%%%%%%%%%%%
\end{thebibliography}

\end{document}